\documentclass[fleqn]{article}

\usepackage[a4paper, left=2cm, right=2cm, top=3cm, bottom=3cm]{geometry}                           % page format
\usepackage[applemac]{inputenc}                                                                    % Mac OS X character set
\usepackage[T1]{fontenc}                                                                           % character set
\usepackage[colorlinks, linkcolor=black, citecolor=black, urlcolor=black]{hyperref}                % hyperlinked pdf-output
\usepackage{ifthen}                                                                                % if-then-package for writing macros
\usepackage{amsmath}                                                                               % AMS-LaTeX
\usepackage{amssymb}                                                                               % AMS-LaTeX
\usepackage{amsthm}                                                                                % AMS-LaTeX
\usepackage{upgreek}                                                                               % upright greek letters
\usepackage{tikz}                                                                                  % TikZ for graphics

\usetikzlibrary{matrix,calc}                                                  

\allowdisplaybreaks                                                                                % allow pagebreaks in math environments

\swapnumbers
\theoremstyle{definition}
\newtheorem{definition}{Definition}[section]
\newtheorem{corollary}[definition]{Corollary}
\newtheorem{example}[definition]{Example}
\newtheorem{notation}[definition]{Notation}
\newtheorem*{notation*}{Notation}
\newtheorem{proposition}[definition]{Proposition}
\newtheorem{question}[definition]{Question}
\newtheorem{remark}[definition]{Remark}
\newtheorem*{remark*}{Remark}
\newtheorem{theorem}[definition]{Theorem}
\newtheorem*{theorem*}{Theorem}
\newtheorem{workingbase}[definition]{Working base}

\makeatletter
\renewcommand*\p@enumii{}                                                                          % format iterated enumerate environments
\makeatother

\newcommand{\BIGOP}[1]{\mathop{\mathchoice%
{\raise-0.22em\hbox{\huge $#1$}}%
{\raise-0.05em\hbox{\Large $#1$}}%
{\hbox{\large $#1$}}%
{#1}}}

\def\morchoice#1#2{\begingroup\setbox0=\hbox{$#1\xrightarrow{#2}$}%
        \setbox1=\hbox{$#1\longrightarrow$}%
        \ifdim\wd0<\wd1
        \stackrel{#2}\longrightarrow
        \else
        \xrightarrow{#2}\fi\endgroup}

\newcommand{\booktitle}[1]{\textsl{#1}}                                                            % booknames
\newcommand{\eigenname}[1]{\textsc{#1}}                                                            % eigennames
\newcommand{\newnotion}[1]{\textit{#1}}                                                            % new notions

\newenvironment{smallpmatrix}{\left( \begin{smallmatrix}}{\end{smallmatrix} \right)}               % small matrices with round brackets
\newcommand{\nbd}{\nobreakdash-\hspace{0pt}}                                                       % no-break-dash

\DeclareMathOperator{\AutomorphismGroup}{\mathrm{Aut}}                                             % automorphism group
\DeclareMathOperator{\Cokernel}{\mathrm{Coker}}                                                    % cokernel of a morphism
\DeclareMathOperator{\ConstantFunctor}{\mathrm{Const}}                                             % constant functor
\DeclareMathOperator{\Coskeleton}{\mathrm{Cosk}}                                                   % coskeleton
\DeclareMathOperator{\DiagonalFunctor}{\mathrm{Diag}}                                              % diagonal functor
\DeclareMathOperator{\GroupPart}{\mathrm{Gp}}	                                                      % group part of a crossed module
\DeclareMathOperator{\Image}{\mathrm{Im}}                                                          % image of a morphism
\DeclareMathOperator{\Kernel}{\mathrm{Ker}}                                                        % kernel of a morphism
\DeclareMathOperator{\ModulePart}{\mathrm{Mp}}                                                     % module part of a crossed module
\DeclareMathOperator{\Ob}{\mathrm{Ob}}                                                             % set of objects
\DeclareMathOperator{\ShiftUp}{\mathrm{Sh}}                                                        % shift up functor
\DeclareMathOperator{\Truncation}{\mathrm{Trunc}}                                                  % truncation
\newcommand{\AbGrp}{\mathbf{AbGrp}}                                                                % category of abelian groups
\newcommand{\act}{\cdot}                                                                           % group action
\newcommand{\ascendinginterval}[2]{\lceil #1, #2 \rceil}                                           % ascending interval
\newcommand{\bigcart}{\BIGOP{\times}}                                                              % cartesian product over a family
\newcommand{\bigintersection}{\bigcap}                                                             % intersection over a family
\newcommand{\BoundaryGroup}[1][]{\mathrm{B}_{#1}}                                                  % boundary group
\newcommand{\canonicalepimorphism}[1][]{\uppi^{#1}}                                                % canonical epimorphism
\newcommand{\canonicalmonomorphism}[1][]{\upiota^{#1}}                                             % canonical monomorphism
\newcommand{\cart}{\times}                                                                         % cartesian product
\newcommand{\Cat}{\mathbf{Cat}}                                                                    % category of categories
\newcommand{\CatNaturalNumber}[1]{[#1]}                                                            % category induced by a natural number
\newcommand{\ClassifyingSimplicialSet}[1][]{\ifthenelse{\equal{#1}{}}{\mathrm{B}}{\mathrm{B}^{(#1)}}} % classifying simplicial set
\newcommand{\CoboundaryGroup}[1][]{\mathrm{B}^{#1}}                                                % coboundary group
\newcommand{\CocycleGroup}[1][]{\mathrm{Z}^{#1}}                                                   % cocycle group
\newcommand{\CochainComplex}[1][]{\mathrm{Ch}^{#1}}                                                % cochain complex
\newcommand{\cocycleofextension}[1]{\mathrm{z}^{#1}}                                               % n-cocycle of an ext. w. resp. to lift. system
\newcommand{\cohomologyclassofextension}[1][]{\mathrm{cl}^{#1}}                                    % cohomology class of an extension
\newcommand{\CohomologyGroup}[1][]{\mathrm{H}^{#1}}                                                % cohomology group
\newcommand{\comp}{\circ}                                                                          % composition
\newcommand{\Complexes}[1][]{\mathbf{C}^{#1}}	                                                      % category of n-complexes over a ptd. category
\newcommand{\CrMod}{\mathbf{CrMod}}                                                                % category of crossed modules
\newcommand{\CycleGroup}[1][]{\mathrm{Z}_{#1}}                                                     % cycle group
\newcommand{\CyclicGroup}{\mathrm{C}}	                                                              % cyclic group
\newcommand{\degeneracy}[1]{\mathrm{s}_{#1}}                                                       % degeneracy
\newcommand{\degeneracymodel}[1]{\upsigma^{#1}}                                                    % degeneracy model
\newcommand{\descendinginterval}[2]{\lfloor #1, #2 \rfloor}                                        % descending interval
\newcommand{\differential}{\partial}                                                               % differential
\newcommand{\directprod}{\times}                                                                   % direct product
\newcommand{\directsum}{\oplus}                                                                    % direct sum
\newcommand{\ExtensionClasses}[2][]{\mathrm{Ext}_{#1}^{#2}}                                        % set of ext. cl. in a Grothendieck universe
\newcommand{\Extensions}[2][]{\underline{\mathrm{Ext}}_{#1}^{#2}}                                  % set of extensions in a Grothendieck universe
\newcommand{\extensionequivalent}[1][]{\approx_{#1}}                                               % extension equivalent
\newcommand{\face}[1]{\mathrm{d}_{#1}}                                                             % face
\newcommand{\facemodel}[1]{\updelta^{#1}}                                                          % face model
\newcommand{\fibreprod}[3][]{\mathbin{{_{#2}^{}}\times_{#3}^{#1}}}                                 % fibre product as subset of the cart. product
\newcommand{\functorcategory}[2]{\pmb{(} #1, #2 \pmb{)}}                                           % functor category
\newcommand{\Grp}{\mathbf{Grp}}                                                                    % category of groups
\newcommand{\Hom}{\mathrm{Hom}}                                                                    % homomorphism set
\newcommand{\HomologyGroup}[1][]{\mathrm{H}_{#1}}                                                  % homology group
\newcommand{\HomotopyGroup}[1][]{\uppi_{#1}}                                                       % homotopy group
\newcommand{\id}{\mathrm{id}}                                                                      % identity
\newcommand{\inc}{\mathrm{inc}}                                                                    % inclusion morphism
\newcommand{\Integers}{\mathbb{Z}}                                                                 % set of integers
\newcommand{\intersection}{\cap}                                                                   % intersection
\newcommand{\isomorphic}{\cong}                                                                    % isomorphic
\newcommand{\KanClassifyingSimplicialSet}{\overline{\KanResolvingSimplicialSet}}                   % Kan classifying simplicial set
\newcommand{\KanResolvingSimplicialSet}{\mathrm{W}}                                                % Kan resolving simplicial set
\newcommand{\leftadjoint}{\dashv}                                                                  % left adjoint
\newcommand{\Map}{\mathrm{Map}}                                                                    % mapping set
\newcommand{\map}{\rightarrow}                                                                     % map
\newcommand{\MooreComplex}[1][]{\mathrm{M}_{#1}}                                                   % Moore complex
\newcommand{\morphism}[1][]{\mathpalette\morchoice{#1}}                                            % morphism
\newcommand{\Naturals}{\mathbb{N}}                                                                 % set of natural numbers
\newcommand{\Nerve}[1][]{\mathrm{N}_{#1}}                                                          % nerve of a category
\newcommand{\op}{\mathrm{op}}                                                                      % opposite category
\newcommand{\PathSimplicialObject}[1][]{\ifthenelse{\equal{#1}{}}{\mathrm{P}}{\mathrm{P}^{(#1)}}}  % path simplicial object of a simplicial object
\newcommand{\pointisation}[2][]{\ifthenelse{\equal{#1}{}}{{#2}^{\mathrm{pt}}}{{#2}^{\mathrm{pt}, #1}}} % pointisation of an Kan n-cocycle
\newcommand{\pointiser}[2][]{\ifthenelse{\equal{#1}{}}{\mathrm{p}_{#2}}{\mathrm{p}_{#2}^{#1}}}     % pointiser of an Kan n-cocycle
\newcommand{\postnikovinvariant}[2]{\mathrm{k}_{#2}^{#1}}                                          % postnikov invariants
\newcommand{\PosettoCat}{\mathtt{Cat}}                                                             % functor poset -> cat
\newcommand{\quo}{\mathrm{quo}}                                                                    % quotient morphism
\newcommand{\ResolvingSimplicialSet}[1][]{\ifthenelse{\equal{#1}{}}{\mathrm{E}}{\mathrm{E}^{(#1)}}}% resolving simplicial set
\newcommand{\sCategory}[1][]{\mathbf{s}^{#1}}                                                      % category of simplicial objects in a category
\newcommand{\semidirect}[1][]{\mathbin{{_{#1}}\!\rtimes}}                                          % n-fold semidirect product
\newcommand{\Set}{\mathbf{Set}}                                                                    % category of sets
\newcommand{\sGrp}[1][]{\sCategory[#1]\Grp}                                                        % category of n-simplicial groups
\newcommand{\SimplexTypes}{\mathbf{\Delta}}	                                                      % category of simplex types
\newcommand{\sSet}[1][]{\sCategory[#1]\Set}                                                        % category of n-simplicial sets
\newcommand{\StandardExtension}{\mathrm{E}}                                                        % standard extension of a 3-cocycle
\newcommand{\standardisation}[2][]{\ifthenelse{\equal{#1}{}}{{#2}^{\mathrm{st}}}{{#2}^{\mathrm{st}, #1}}} % standardisation of a Kan 2-cocycle
\newcommand{\standardiser}[2][]{\ifthenelse{\equal{#1}{}}{\mathrm{s}_{#2}}{\mathrm{s}_{#2}^{#1}}}  % standardiser of a Kan 2-cocycle
\newcommand{\standardsectionsystem}{\mathrm{s}}                                                    % standard section system
\newcommand{\structuremorphism}[1][]{\upmu^{#1}}                                                   % structure morphism of a crossed module
\newcommand{\subobjecteq}{\leq}                                                                    % subobject
\newcommand{\triv}{\mathrm{triv}}                                                                  % trivial morphism
\newcommand{\union}{\cup}                                                                          % union
\newcommand{\unit}{\upvarepsilon}                                                                  % unit

\tikzset{equality/.style={-, double}}
\tikzset{diagram/.style={matrix of math nodes, row sep=#1, column sep=#1, text height=1.6ex, text depth=0.45ex}, diagram/.default=2.5em, draw, inner sep=0pt, nodes={inner sep=0.333em}}

\title{On the second cohomology group of a simplicial group}
\author{Sebastian Thomas}
\date{September 29, 2010}

\begin{document}

\maketitle

\renewcommand{\thefootnote}{\fnsymbol{footnote}}
\footnotetext[0]{Mathematics Subject Classification 2010: 55U10, 18G30, 18D35, 20J06. \\ This article has been published in condensed form in Homology, Homotopy and Applications \textbf{12}(2) (2010), pp.~167--210.}
% 55U10: Simplicial sets and complexes
% 18G30: Simplicial sets, simplicial objects (in a category)
% 18D35: Structured objects in a category (group objects, etc.)
% 20J06: Cohomology of groups
\renewcommand{\thefootnote}{\arabic{footnote}}

\begin{abstract}
We give an algebraic proof for the result of \eigenname{Eilenberg} and \eigenname{Mac\,Lane} that the second cohomology group of a simplicial group \(G\) can be computed as a quotient of a fibre product involving the first two homotopy groups and the first Postnikov invariant of \(G\). Our main tool is the theory of crossed module extensions of groups.
\end{abstract}

\section{Introduction} \label{sec:introduction}

In~\cite{eilenberg_maclane:1946:determination_of_the_second_homology_and_cohomology_groups_of_a_space_by_means_of_homotopy_invariants}, \eigenname{Eilenberg} and \eigenname{Mac\,Lane} assigned to an arcwise connected pointed topological space \(X\) a topological invariant \(k^3 \in \CohomologyGroup[3](\HomotopyGroup[1](X), \HomotopyGroup[2](X))\), that is, a \(3\)-cohomology class of the fundamental group \(\HomotopyGroup[1](X)\) with coefficients in the \(\HomotopyGroup[1](X)\)-module \(\HomotopyGroup[2](X)\), which is nowadays known as the first Postnikov invariant of \(X\). Thereafter, they showed that the second cohomology group of \(X\) with coefficients in an abelian group \(A\) only depends on \(\HomotopyGroup[1](X)\), \(\HomotopyGroup[2](X)\) and \(k^3\). Explicitly, they described this dependency as follows. We let \(\CochainComplex(\HomotopyGroup[1](X), A)\) denote the cochain complex of \(\HomotopyGroup[1](X)\) with coefficients in \(A\) and \(\Hom_{\HomotopyGroup[1](X)}(\HomotopyGroup[2](X), A)\) denote the group of \(\HomotopyGroup[1](X)\)-equivariant group homomorphisms from \(\HomotopyGroup[2](X)\) to \(A\), where \(\HomotopyGroup[1](X)\) is supposed to act trivially on \(A\).

\begin{theorem*}[{\eigenname{Eilenberg}, \eigenname{Mac\,Lane}, 1946~\cite[thm.~2]{eilenberg_maclane:1946:determination_of_the_second_homology_and_cohomology_groups_of_a_space_by_means_of_homotopy_invariants}}]
We choose a \(3\)-cocycle \(z^3 \in \CocycleGroup[3](\HomotopyGroup[1](X), \HomotopyGroup[2](X))\) such that \(k^3 = z^3 \CoboundaryGroup[3](\HomotopyGroup[1](X), \HomotopyGroup[2](X))\). The second cohomology group \(\CohomologyGroup[2](X, A)\) is isomorphic to the quotient group
\[Z^2 / B^2,\]
where \(Z^2\) is defined to be the fibre product of
\[\begin{tikzpicture}[baseline=(m-2-1.base)]
  \matrix (m) [diagram]{
    & \Hom_{\HomotopyGroup[1](X)}(\HomotopyGroup[2](X), A) \\
    \CochainComplex[2](\HomotopyGroup[1](X), A) & \CochainComplex[3](\HomotopyGroup[1](X), A) \\};
  \path[->, font=\scriptsize] let \p1=(1.25em, 0) in
    (m-1-2) edge (m-2-2)
    (m-2-1) edge node[above] {\(\differential\)} (m-2-2);
\end{tikzpicture}\]
with vertical map given by \(\varphi \mapsto z^3 \varphi\), and where \(B^2\) is defined to be the subgroup
\[B^2 := \{0\} \directprod \CoboundaryGroup[2](\HomotopyGroup[1](X), A) \subobjecteq Z^2 \subobjecteq \Hom_{\HomotopyGroup[1](X)}(\HomotopyGroup[2](X), A) \directprod \CochainComplex[2](\HomotopyGroup[1](X), A).\]
\end{theorem*}

In this article, we give an algebraic proof of the simplicial group version of the theorem of \eigenname{Eilenberg} and \eigenname{Mac\,Lane}, cf.\ theorem~\ref{th:second_cohomology_group_of_crossed_module_via_homotopy_invariants}\ref{th:second_cohomology_group_of_crossed_module_via_homotopy_invariants:simplicial_group}. Since simplicial groups are algebraic models for path connected homotopy types of CW-spaces, this yields an algebraic proof for their original theorem mentioned above.

It turns out to be convenient to work on the level of crossed modules. To any simplicial group \(G\), we can attach its crossed module segment \(\Truncation^1 G\), while to any crossed module \(V\), we can attach its simplicial group coskeleton \(\Coskeleton_1 V\). We have \(\CohomologyGroup[2](G, A) \isomorphic \CohomologyGroup[2](\Coskeleton_1 \Truncation^1 G, A)\). Moreover, the crossed module segment of \(G\) suffices to define the Postnikov invariant \(k^3\) of \(G\) via choices of certain sections, see~\cite[ch.~IV, sec.~5]{brown:1982:cohomology_of_groups} or~\cite[sec.~4]{thomas:2009:the_third_cohomology_group_classifies_crossed_module_extensions}. These sections pervade our algebraic approach.

Related to this theorem, \eigenname{Ellis}~\cite[th.~10]{ellis:1992:homology_of_2-types} has shown that there exists a long exact sequence involving the second cohomology group \(\CohomologyGroup[2](V, A)\) of a crossed module \(V\) starting with
\[0 \morphism \CohomologyGroup[2](\HomotopyGroup[0](V), A) \morphism \CohomologyGroup[2](V, A) \morphism \Hom_{\HomotopyGroup[0](V)}(\HomotopyGroup[1](V), A).\]
This part of his sequence is also a consequence of our \eigenname{Eilenberg}-\eigenname{Mac\,Lane}-type description of \(\CohomologyGroup[2](V, A)\), cf.\ theorem~\ref{th:second_cohomology_group_of_crossed_module_via_homotopy_invariants}. (\footnote{Our notation here differs from \eigenname{Ellis}' by a dimension shift.})

Concerning Postnikov invariants, cf.\ also~\cite{bullejos_faro_garcia-munoz:2005:postnikov_invariants_of_crossed_complexes}, where general Postnikov invariants for crossed complexes, which are generalisations of crossed modules, are constructed.

\paragraph{Outline} \label{par:outline}
In section~\ref{sec:preliminaries_on_simplicial_objects_crossed_modules_cohomology_and_extensions}, we recall some basic facts from simplicial algebraic topology, in particular cohomology of simplicial groups. We will recall how simplicial groups, crossed modules and (ordinary) groups interrelate. Finally, we will give a brief outline how a cohomology class can be attached to a crossed module -- and hence to a simplicial group -- and conversely.

In section~\ref{sec:low_dimensional_cohomology_of_a_simplicial_group}, we will consider the low-dimensional cohomology groups of a simplicial group. The aim of this section is to give algebraic proofs of the well-known facts that the first cohomology group depends only on the group segment and the second cohomology group depends only on the crossed module segment of the given simplicial group. This gives already a convenient description of simplicial group cohomology in dimensions \(0\) and \(1\), and can be seen in dimension \(2\) as a reduction step allowing us to work with crossed modules in the following.

In section~\ref{sec:crossed_module_extensions_and_standard_2-cocycles}, we introduce a certain standardised form of \(2\)-cocycles and \(2\)-coboundaries of a crossed module, which suffices to compute the second cohomology group. On the other hand, this standardisation directly yields the groups \(Z^2\) and \(B^2\) occurring in the description of \eigenname{Eilenberg} and \eigenname{Mac\,Lane}.

We apply our results of sections~\ref{sec:low_dimensional_cohomology_of_a_simplicial_group} and~\ref{sec:crossed_module_extensions_and_standard_2-cocycles} in section~\ref{sec:second_eilenberg-maclane_cohomology_group} to simplicial groups, thus obtaining the analogon of \eigenname{Eilenberg}s and \eigenname{Mac\,Lane}s theorem. Finally, we discuss some corollaries and examples.

\paragraph{Acknowledgement} \label{par:acknowledgement}

I thank \eigenname{Matthias K{\"u}nzer} for many useful discussions on this article and for directing me to the article of \eigenname{Eilenberg} and \eigenname{Mac\,Lane}~\cite{eilenberg_maclane:1946:determination_of_the_second_homology_and_cohomology_groups_of_a_space_by_means_of_homotopy_invariants}.

\subsection*{Conventions and notations} \label{ssec:conventions_and_notations}

We use the following conventions and notations.

\begin{itemize}
\item The composite of morphisms \(f\colon X \map Y\) and \(g\colon Y \map Z\) is usually denoted by \(f g\colon X \map Z\). The composite of functors \(F\colon \mathcal{C} \map \mathcal{D}\) and \(G\colon \mathcal{D} \map \mathcal{E}\) is usually denoted by \(G \comp F\colon \mathcal{C} \map \mathcal{E}\).
\item We use the notations \(\Naturals = \{1, 2, 3, \dots\}\) and \(\Naturals_0 = \Naturals \cup \{0\}\).
\item Given a map \(f\colon X \rightarrow Y\) and subsets \(X' \subseteq X\), \(Y' \subseteq Y\) with \(X' f \subseteq Y'\), we write \(f|_{X'}^{Y'}\colon X' \rightarrow Y', x' \mapsto x' f\). Moreover, we abbreviate \(f|_{X'} := f|_{X'}^Y\) and \(f|^{Y'} := f|_X^{Y'}\).
\item Given integers \(a, b \in \Integers\), we write \([a, b] := \{z \in \Integers \mid a \leq z \leq b\}\) for the set of integers lying between \(a\) and \(b\). If we need to specify orientation, then we write \(\ascendinginterval{a}{b} := (z \in \Integers \mid a \leq z \leq b)\) for the \newnotion{ascending interval} and \(\descendinginterval{a}{b} = (z \in \Integers \mid a \geq z \geq b)\) for the \newnotion{descending interval}. Whereas we formally deal with tuples, we use the element notation; for example, we write \(\prod_{i \in \ascendinginterval{1}{3}} g_i = g_1 g_2 g_3\) and \(\prod_{i \in \descendinginterval{3}{1}} g_i = g_3 g_2 g_1\) or \((g_i)_{i \in \descendinginterval{3}{1}} = (g_3, g_2, g_1)\) for group elements \(g_1\), \(g_2\), \(g_3\).
\item Given tuples \((x_j)_{j \in A}\) and \((x_j)_{j \in B}\) with disjoint index sets \(A\) and \(B\), we write \((x_j)_{j \in A} \union (x_j)_{j \in B}\) for their concatenation.
\item Given groups \(G\) and \(H\), we denote by \(\triv\colon G \map H\) the trivial group homomorphism \(g \mapsto 1\).
\item Given a group homomorphism \(\varphi\colon G \map H\), we denote its kernel by \(\Kernel \varphi\), its cokernel by \(\Cokernel \varphi\) and its image by \(\Image \varphi\). Moreover, we write \(\inc = \inc^{\Kernel \varphi}\colon \Kernel \varphi \map G\) for the inclusion and \(\quo = \quo^{\Cokernel \varphi}\colon H \map \Cokernel \varphi\) for the quotient morphism.
\item The distinguished point in a pointed set \(X\) will be denoted by \(* = *^X\).
\item The fibre product of group homomorphisms \(\varphi_1\colon G_1 \map H\) and \(\varphi_2\colon G_2 \map H\) will be denoted by \(G_1 \fibreprod{\varphi_1}{\varphi_2} G_2\).
\end{itemize}

\paragraph{A remark on functoriality} \label{par:a_remark_on_functoriality}

Most constructions defined below, for example \(\MooreComplex\), \(\CochainComplex\), etc., are functorial, although we only describe them on the objects of the respective source categories. For the definitions on the morphisms and other details, we refer the reader for example to~\cite{thomas:2007:co_homology_of_crossed_modules}.

\paragraph{A remark on Grothendieck universes} \label{par:a_remark_on_grothendieck_universes}

To avoid set-theoretical difficulties, we work with Grothendieck universes~\cite[exp.~I, sec.~0]{artin_grothendieck_verdier:1972:sga_4_1} in this article. In particular, every category has an object \emph{set} and a morphism \emph{set}.

We suppose given a Grothendieck universe \(\mathfrak{U}\). A \newnotion{set in \(\mathfrak{U}\)} (or \(\mathfrak{U}\)-set) is a set that is an element of \(\mathfrak{U}\), a \newnotion{map in \(\mathfrak{U}\)} (or \(\mathfrak{U}\)-map) is a map between \(\mathfrak{U}\)-sets. The \newnotion{category of \(\mathfrak{U}\)-sets} consisting of the set of \(\mathfrak{U}\)-sets, that is, of \(\mathfrak{U}\), as object set and the set of \(\mathfrak{U}\)-maps as morphism set will be denoted by \(\Set_{(\mathfrak{U})}\). A \newnotion{group in \(\mathfrak{U}\)} (or \(\mathfrak{U}\)-group) is a group whose underlying set is a \(\mathfrak{U}\)-set, a \newnotion{group homomorphism in \(\mathfrak{U}\)} (or \(\mathfrak{U}\)-group homomorphism) is a group homomorphism between \(\mathfrak{U}\)-groups. The \newnotion{category of \(\mathfrak{U}\)-groups} consisting of the set of \(\mathfrak{U}\)-groups as object set and the set of \(\mathfrak{U}\)-group homomorphisms as morphism set will be denoted by \(\Grp_{(\mathfrak{U})}\).

Because we do not want to overload our text with the usage of Grothendieck universes, we may suppress them in notation, provided we work with a single fixed Grothendieck universe. For example, instead of
\begin{remark*}
We suppose given a Grothendieck universe \(\mathfrak{U}\). The forgetful functor \(\Grp_{(\mathfrak{U})} \map \Set_{(\mathfrak{U})}\) is faithful.
\end{remark*}
we may just write
\begin{remark*}
The forgetful functor \(\Grp \map \Set\) is faithful.
\end{remark*}

Grothendieck universes will play a role in the discussion of crossed module extensions, cf.\ section~\ref{ssec:crossed_module_extensions}.

\section{Preliminaries on simplicial objects, crossed modules, cohomology and extensions} \label{sec:preliminaries_on_simplicial_objects_crossed_modules_cohomology_and_extensions}

In this section, we recall some standard definitions and basic facts of simplicial algebraic topology and crossed modules. Concerning simplicial algebraic topology, the reader is refered for example to the books of \eigenname{Goerss} and \eigenname{Jardine}~\cite{goerss_jardine:1999:simplicial_homotopy_theory} or \eigenname{May}~\cite{may:1992:simplicial_objects_in_algebraic_topology}, and a standard reference on crossed modules is the survey of \eigenname{Brown}~\cite{brown:1999:groupoids_and_crossed_objects_in_algebraic_topology}.

The main purpose of this section is to fix notation and to explain how the cocycle formulas in the working base~\ref{wb:low_dimensional_analysed_cocycles} can be deduced. The reader willing to believe the working base~\ref{wb:low_dimensional_analysed_cocycles} can start to read at that point, occasionally looking up notation.

\subsection{The category of simplex types} \label{ssec:the_category_of_simplex_types}

We suppose given a Grothendieck universe containing an infinite set. For \(n \in \Naturals_0\), we let \(\CatNaturalNumber{n}\) denote the category induced by the totally ordered set \([0, n]\) with the natural order, and we let \(\SimplexTypes\) be the full subcategory in \(\Cat\) defined by \(\Ob \SimplexTypes := \{\CatNaturalNumber{n} \mid n \in \Naturals_0\}\), called the \newnotion{category of simplex types}.

For \(n \in \Naturals\), \(k \in [0, n]\), we let \(\facemodel{k}\colon \CatNaturalNumber{n - 1} \map \CatNaturalNumber{n}\) be the injection that omits \(k\), and for \(n \in \Naturals_0\), \(k \in [0, n]\), we let \(\degeneracymodel{k}\colon \CatNaturalNumber{n + 1} \map \CatNaturalNumber{n}\) be the surjection that repeats \(k\).

There is a \newnotion{shift functor} \(\ShiftUp\colon \SimplexTypes \map \SimplexTypes\) given by \(\ShiftUp{\CatNaturalNumber{n}} := \CatNaturalNumber{n + 1}\) and
\[i (\ShiftUp \theta) := \begin{cases} i \theta & \text{for \(i \in [0, m]\)}, \\ n + 1 & \text{for \(i = m + 1\)}, \end{cases}\]
for \(i \in [0, m + 1]\), morphisms \(\theta \in {_{\SimplexTypes}}(\CatNaturalNumber{m}, \CatNaturalNumber{n})\), where \(m, n \in \Naturals_0\).

\subsection{Simplicial objects} \label{ssec:simplicial_objects}

The \newnotion{category of simplicial objects} in a given category \(\mathcal{C}\) is defined to be the functor category \(\sCategory \mathcal{C} := \functorcategory{\SimplexTypes^\op}{\mathcal{C}}\). Moreover, the \newnotion{category of bisimplicial objects} \(\sCategory[2] \mathcal{C}\) in \(\mathcal{C}\) is defined to be \(\functorcategory{\SimplexTypes^\op \cart \SimplexTypes^\op}{\mathcal{C}}\). The objects resp.\ morphisms of \(\sCategory \mathcal{C}\) are called \newnotion{simplicial objects} in \(\mathcal{C}\) resp.\ \newnotion{simplicial morphisms} in \(\mathcal{C}\). Likewise for \(\sCategory[2] \mathcal{C}\).

Given a simplicial object \(X\) in a category \(\mathcal{C}\), we abbreviate \(X_n := X_{\CatNaturalNumber{n}}\) for \(n \in \Naturals_0\). Moreover, given a morphism \(\theta \in {_{\SimplexTypes}}(\CatNaturalNumber{m}, \CatNaturalNumber{n})\), where \(m, n \in \Naturals_0\), we write \(X_{\theta}\) for the image of \(\theta\) under \(X\) and call this morphism the \newnotion{simplicial operation} induced by \(\theta\). Similarly for bisimplicial objects. The images of the morphisms \(\facemodel{k}\) resp.\ \(\degeneracymodel{k}\) under a simplicial object \(X\) in a given category \(\mathcal{C}\) are denoted by \(\face{k} = \face{k}^X := X_{\facemodel{k}}\), called the \(k\)-th \newnotion{face}, for \(k \in [0, n]\), \(n \in \Naturals\), resp.\ \(\degeneracy{k} = \degeneracy{k}^X := X_{\degeneracymodel{k}}\), called the \(k\)-th \newnotion{degeneracy}, for \(k \in [0, n]\), \(n \in \Naturals_0\). For the simplicial identities between the faces and degeneracies in our composition order, see for example~\cite[prop.~(1.14)]{thomas:2007:co_homology_of_crossed_modules}. We use the ascending and descending interval notation for composites of faces resp.\ degeneracies, that is, we write \(\face{\descendinginterval{l}{k}} := \face{l} \face{l - 1} \dots \face{k}\) resp.\ \(\degeneracy{\ascendinginterval{k}{l}} := \degeneracy{k} \degeneracy{k + 1} \dots \degeneracy{l}\).

Given an object \(X \in \Ob \mathcal{C}\), we have the \newnotion{constant simplicial object} \(\ConstantFunctor X\) in \(\mathcal{C}\) with \(\ConstantFunctor_n X := X\) for \(n \in \Naturals_0\) and \(\ConstantFunctor_{\theta} X := 1_X\) for \(\theta \in {_{\SimplexTypes}}(\CatNaturalNumber{m}, \CatNaturalNumber{n})\), \(m, n \in \Naturals_0\).

Given a bisimplicial object \(X \in \Ob \sCategory[2] \mathcal{C}\), we have the \newnotion{diagonal simplicial object} \(\DiagonalFunctor X\), defined by \(\DiagonalFunctor_n X := X_{n, n}\) and \(\DiagonalFunctor_{\theta} X := X_{\theta, \theta}\) for \(\theta \in {_{\SimplexTypes}}(\CatNaturalNumber{m}, \CatNaturalNumber{n})\), \(m, n \in \Naturals_0\).

A \newnotion{simplicial set} resp.\ a \newnotion{simplicial map} is a simplicial object resp.\ a simplicial morphism in \(\Set_{(\mathfrak{U})}\) for some Grothendieck universe \(\mathfrak{U}\).
A \newnotion{simplicial group} resp.\ a \newnotion{simplicial group homomorphism} is a simplicial object resp.\ a simplicial morphism in \(\Grp_{(\mathfrak{U})}\) for some Grothendieck universe \(\mathfrak{U}\).

\subsection{Complexes of groups} \label{ssec:complexes_of_groups}

We denote by \(\PosettoCat(\Integers)\) the ordered set \(\Integers\) considered as a category. Given \(m, n \in \Integers\) with \(m < n\), the unique morphism from \(m\) to \(n\) will be written by \((m, n)\). A \newnotion{complex} of groups is a functor \(G\colon \PosettoCat(\Integers) \map \Grp_{(\mathfrak{U})}\) for some Grothendieck universe \(\mathfrak{U}\) such that \(G(n, n + 2)\) is trivial for all \(n \in \Integers\). The groups \(G^n = G_{- n} := G(n)\) for \(n \in \Integers\) are called the \newnotion{entries}, the homomorphisms \(\differential = \differential^{G, n} = \differential_{- n}^G := G(n, n + 1)\) for \(n \in \Integers\) are called the \newnotion{differentials} of \(G\).

We suppose given complexes of groups \(G\) and \(H\). A \newnotion{complex morphism} from \(G\) to \(H\) is a natural transformation \(\varphi\colon G \map H\). The group homomorphisms \(\varphi^n = \varphi_{- n} := \varphi(n)\) for \(n \in \Integers\) are called the \newnotion{entries} of \(\varphi\).

When defining a complex of groups \(G\) by declaring entries \(G^n\) and differentials \(\differential^{G, n}\) only for \(n \in \Naturals_0\) resp.\ \(G_n\) for \(n \in \Naturals_0\) and \(\differential^G_n\) for \(n \in \Naturals\), this means that the not explicitly defined entries and differentials are set to be trivial. Likewise for morphisms of complexes of groups.

Given a complex of groups \(G\), we define \(\CocycleGroup[n] G := \Kernel \differential^{G, n}\) and \(\CoboundaryGroup[n] G := \Image \differential^{G, n - 1}\) for \(n \in \Integers\). A complex of groups \(G\) is said to be \newnotion{normal} if \(\CoboundaryGroup[n] G\) is a normal subgroup in \(\CocycleGroup[n] G\) for all \(n \in \Integers\). If \(G\) is normal, we define \(\CohomologyGroup[n] G := \CocycleGroup[n] G / \CoboundaryGroup[n] G\) for \(n \in \Integers\). Moreover, we also write \(\CycleGroup[n] G := \CocycleGroup[- n] G\), etc.

\subsection{The Moore complex of a simplicial group} \label{ssec:the_moore_complex_of_a_simplicial_group}

We suppose given a simplicial group \(G\). The \newnotion{Moore complex} of \(G\) is the complex of groups \(\MooreComplex G\) with entries \(\MooreComplex[n] G := \bigintersection_{k \in [1, n]} \Kernel \face{k} \subobjecteq G_n\) for \(n \in \Naturals_0\) and differentials \(\differential := \face{0}|_{\MooreComplex[n] G}^{\MooreComplex[n - 1] G}\) for \(n \in \Naturals\). In particular, \(\MooreComplex[0] G = G_0\). The boundary group \(\BoundaryGroup[n] \MooreComplex G\) is a normal subgroup of \(G_n\) for all \(n \in \Naturals_0\), so in particular, the Moore complex \(\MooreComplex G\) is a normal complex of groups.

\subsection{Simplicial group actions} \label{ssec:simplicial_group_actions}

We suppose given a category \(\mathcal{C}\) and a group \(G\). Recall that a (\newnotion{group}) \newnotion{action} of \(G\) on an object \(X \in \Ob \mathcal{C}\) is a group homomorphism \(\alpha\colon G^\op \map \AutomorphismGroup_{\mathcal{C}} X\).

A \newnotion{\(G\)-set} consists of a set \(X\) together with an action \(\alpha\) of \(G\) on \(X\). By abuse of notation, we often refer to the \(G\)-set as well as to its underlying set by \(X\). The action \(\alpha\) is called the \newnotion{\(G\)-action} of the \(G\)-set \(X\). Given a \(G\)-set \(X\) with \(G\)-action \(\alpha\), we often write \(g x = g \act x := x (g \alpha)\) for \(x \in X\), \(g \in G\). For an element \(x \in X\), we let \(G x := \{g x \mid g \in G\}\) and \(X / G := \{G x \mid x \in X\}\).

A \newnotion{\(G\)-module} consists of a (not necessarily abelian) group \(M\) together with an action \(\alpha\) of \(G\) on \(M\), that is, a group homomorphism \(\alpha\colon G^\op \map \AutomorphismGroup_{\Grp} M\). By abuse of notation, we often refer to the module over \(G\) as well as to its underlying group by \(M\). The action \(\alpha\) is called the \newnotion{\(G\)-action} of the \(G\)-module \(M\). Given a \(G\)-module \(M\) with \(G\)-action \(\alpha\), we often write \({^g}m := m (g \alpha)\) for \(m \in M\), \(g \in G\). A \(G\)-module \(M\) is said to be \newnotion{abelian} if its underlying group is abelian. (\footnote{In the literature, a \(G\)-module for a given group \(G\) is often called a \newnotion{\(G\)-group} while an abelian \(G\)-module is just a \newnotion{\(G\)-module}. However, the module part of a crossed module is in general a non-abelian module over the group part, cf.\ section~\ref{ssec:crossed_modules}; this would be more complicated to phrase using the terms from the literature.}) As usual, we often write \(M\) additively in this case, and we write \(g m = g \act m := m (g \alpha)\) for \(m \in M\), \(g \in G\), where \(\alpha\) denotes the \(G\)-action of \(M\).

A \(G\)-module structure on \(G\) itself is provided by the conjugation homomorphism \(G^\op \map \AutomorphismGroup G, g \mapsto {^g}(-)\), where \({^g}x = g x g^{- 1}\) for \(x, g \in G\).

We suppose given a simplicial group \(G\). A \newnotion{\(G\)-simplicial set} consists of a simplicial set \(X\) together with actions of \(G_n\) on \(X_n\) for \(n \in \Naturals_0\) such that \((g_n x_n) X_{\theta} = (g_n G_{\theta}) (x_n X_{\theta})\) for all \(g_n \in G_n\), \(x_n \in X_n\), \(\theta \in {_{\SimplexTypes}}(\CatNaturalNumber{m}, \CatNaturalNumber{n})\), where \(m, n \in \Naturals_0\). Given a \(G\)-simplicial set \(X\), we obtain an induced simplicial structure on the sets \(X_n / G_n\) for \(n \in \Naturals_0\), and the resulting simplicial set is denoted by \(X / G\).

An (\newnotion{abelian}) \newnotion{\(G\)-simplicial module} consists of a simplicial (abelian) group \(M\) together with actions of \(G_n\) on \(M_n\) for \(n \in \Naturals_0\) such that \(({^{g_n}}m_n) M_\theta = {^{g_n G_\theta}}(m_n M_\theta)\) for all \(\theta \in {_{\SimplexTypes}}(\CatNaturalNumber{m}, \CatNaturalNumber{n})\), where \(m, n \in \Naturals_0\).

\subsection{Crossed modules} \label{ssec:crossed_modules}

A \newnotion{crossed module} consists of a group \(G\), a \(G\)-module \(M\) and a group homomorphism \(\mu\colon M \rightarrow G\) such that the following two axioms hold.
\begin{itemize}
\item[(Equi)] \emph{Equivariance}. We have \(({^g}m) \mu = {^g}(m \mu)\) for all \(m \in M\), \(g \in G\).
\item[(Peif)] \emph{Peiffer identity}. We have \({^{n \mu}}m = {^n}m\) for all \(m, n \in M\).
\end{itemize}
Here, the action of the elements of \(G\) on \(G\) resp.\ of \(M\) on \(M\) denote in each case the conjugation. We call \(G\) the \newnotion{group part}, \(M\) the \newnotion{module part} and \(\mu\) the \newnotion{structure morphism} of the crossed module. Given a crossed module \(V\) with group part \(G\), module part \(M\) and structure morphism \(\mu\), we write \(\GroupPart V := G\), \(\ModulePart V := M\) and \(\structuremorphism = \structuremorphism[V] := \mu\). For a list of examples of crossed modules, we refer the reader to~\cite[sec.~2]{ellis:1992:homology_of_2-types} and~\cite[ex.~(5.2), ex.~(5.6)]{thomas:2007:co_homology_of_crossed_modules}.

We let \(V\) and \(W\) be crossed modules. A \newnotion{morphism of crossed modules} (or \newnotion{crossed module morphism}) from \(V\) to \(W\) consists of group homomorphisms \(\varphi_0\colon \GroupPart V \rightarrow \GroupPart W\) and \(\varphi_1\colon \ModulePart V \rightarrow \ModulePart W\) such that \(\varphi_1 \structuremorphism[W] = \structuremorphism[V] \varphi_0\) and such that \(({^g}m)\varphi_1 = {^{g \varphi_0}}(m \varphi_1)\) holds for all \(m \in \ModulePart V\), \(g \in \GroupPart V\). The group homomorphisms \(\varphi_0\) resp.\ \(\varphi_1\) are said to be the \newnotion{group part} resp.\ the \newnotion{module part} of the morphism of crossed modules. Given a crossed module morphism \(\varphi\) from \(V\) to \(W\) with group part \(\varphi_0\) and module part \(\varphi_1\), we write \(\GroupPart \varphi := \varphi_0\) and \(\ModulePart \varphi := \varphi_1\). Composition of morphisms of crossed modules is defined by the composition on the group parts and on the module parts.

We let \(\mathfrak{U}\) be a Grothendieck universe. A crossed module \(V\) is said to be \newnotion{in \(\mathfrak{U}\)} (or a \newnotion{\(\mathfrak{U}\)-crossed module}) if \(\GroupPart V\) is a group in \(\mathfrak{U}\) and \(\ModulePart V\) is a \(G\)-module in \(\mathfrak{U}\). The \newnotion{category of \(\mathfrak{U}\)-crossed modules} consisting of \(\mathfrak{U}\)-crossed modules as objects and morphisms of \(\mathfrak{U}\)-crossed modules as morphisms will be denoted by \(\CrMod = \CrMod_{(\mathfrak{U})}\).

Given a crossed module \(V\), the image \(\Image \structuremorphism\) is a normal subgroup of \(\GroupPart V\) and the kernel \(\Kernel \structuremorphism\) is a central subgroup of \(\ModulePart V\). Moreover, the action of \(\GroupPart V\) on \(\ModulePart V\) restricts to a trivial action of \(\Image \structuremorphism\) on \(\Kernel \structuremorphism\). In particular, an action of \(\Cokernel \structuremorphism\) on \(\Kernel \structuremorphism\) is induced. See for example~\cite[sec.~3.1, sec.~3.2]{brown:1999:groupoids_and_crossed_objects_in_algebraic_topology} or~\cite[prop.~(5.3)]{thomas:2007:co_homology_of_crossed_modules}.

\begin{notation*}
Given a crossed module \(V\), the module part \(\ModulePart V\) resp.\ its opposite \((\ModulePart V)^{\op}\) act on (the underlying set of) the group part \(\GroupPart V\) by \(m g := (m \structuremorphism) g\) resp.\ \(g m := g (m \structuremorphism)\) for \(m \in \ModulePart V\), \(g \in \GroupPart V\). Using this, we get for example \({^{m g}}n = {^m}({^g}n)\) and \(g m = g (m \structuremorphism) = ({^g}m) g\) for \(m, n \in \ModulePart V\), \(g \in \GroupPart V\), cf.~\cite[p.~5]{thomas:2009:the_third_cohomology_group_classifies_crossed_module_extensions}. Also note that \((m g) n = m (g n)\) for \(m, n \in \ModulePart V\), \(g \in \GroupPart V\).

Given a set \(X\) and a map \(f\colon \GroupPart V \map X\), we usually write \(m f := m \structuremorphism f\) for \(m \in \ModulePart V\). Similarly for maps \(\GroupPart V \cart \GroupPart V \map X\), etc.

Moreover, given crossed modules \(V\) and \(W\) and a morphism of crossed modules \(\varphi\colon V \map W\), we may write \(m \varphi\) and \(g \varphi\) instead of \(m (\ModulePart \varphi)\) and \(g (\GroupPart \varphi)\). Using this, we have \((m g) \varphi = (m \varphi) (g \varphi)\) for \(m \in \ModulePart V\), \(g \in \GroupPart V\), cf.\ again~\cite[p.~5]{thomas:2009:the_third_cohomology_group_classifies_crossed_module_extensions}.
\end{notation*}

\subsection{Truncation and coskeleton} \label{ssec:truncation_and_coskeleton}

Groups, crossed modules and simplicial groups can be related to each other in the following way.

We suppose given a simplicial group \(G\). Then we define the group \(\Truncation^0 G := \MooreComplex[0] G / \BoundaryGroup[0] \MooreComplex G = G_0 / \BoundaryGroup[0] \MooreComplex G\), called the \newnotion{group segment} (or the \newnotion{\(0\)-truncation}) of \(G\). Moreover, we define a crossed module \(\Truncation^1 G\), called the \newnotion{crossed module segment} (or the \newnotion{\(1\)-truncation}) of \(G\), as follows. The group part is defined by \(\GroupPart \Truncation^1 G := \MooreComplex[0] G = G_0\), the module part is defined by \(\ModulePart \Truncation^1 G := \MooreComplex[1] G / \BoundaryGroup[1] \MooreComplex G\), the structure morphism is given by \((g_1 \BoundaryGroup[1] \MooreComplex G) \structuremorphism[\Truncation^1 G] := g_1 \differential = g_1 \face{0}\) for \(g_1 \in \MooreComplex[1] G\) and the group part acts on the module part by \({^{g_0}}(g_1 \BoundaryGroup[1] \MooreComplex G) := {^{g_0 \degeneracy{0}}}g_1 \BoundaryGroup[1] \MooreComplex G\) for \(g_i \in \MooreComplex[i] G\), \(i \in [0, 1]\).

Next, we suppose given a crossed module \(V\). We define a group \(\Truncation_1^0 V := \Cokernel \structuremorphism\), called the \newnotion{group segment} (or the \newnotion{\(0\)-truncation}) of \(V\). Moreover, we obtain a simplicial group \(\Coskeleton_1 V\), the \newnotion{coskeleton simplicial group} (or just the \newnotion{coskeleton}) of \(V\), constructed as follows. We suppose given \(n \in \Naturals_0\). Denoting the elements in \((\ModulePart V)^{\directprod n} \directprod \GroupPart V\) by \((m_i, g)_{i \in \descendinginterval{n - 1}{0}} := (m_i)_{i \in \descendinginterval{n - 1}{0}} \cup (g)\), we equip this set with a multiplication given by
\[(m_i, g)_{i \in \descendinginterval{n - 1}{0}} (m_i', g')_{i \in \descendinginterval{n - 1}{0}} := (m_i \, {^{(\prod_{k \in \descendinginterval{i - 1}{0}} m_k) g}}m_i', g g')_{i \in \descendinginterval{n - 1}{0}}\]
for \(m_i, m_i' \in \ModulePart V\), where \(i \in \descendinginterval{n - 1}{0}\), \(g, g' \in \GroupPart V\). The resulting group will be denoted by \(\ModulePart V \semidirect[n] \GroupPart V\). For \(\theta \in {_{\SimplexTypes}}(\CatNaturalNumber{m}, \CatNaturalNumber{n})\), we define a group homomorphism \(\ModulePart V \semidirect[\theta] \GroupPart V\colon \ModulePart V \semidirect[n] \GroupPart V \map \ModulePart V \semidirect[m] \GroupPart V\) by
\[(m_j, g)_{j \in \descendinginterval{n - 1}{0}} (\ModulePart V \semidirect[\theta] \GroupPart V) := (\prod_{k \in \descendinginterval{(i + 1) \theta - 1}{i \theta}} m_k, (\prod_{k \in \descendinginterval{0 \theta - 1}{0}} m_k) g)_{i \in \descendinginterval{m - 1}{0}}.\]
The resulting simplicial group \(\Coskeleton_1 V := \ModulePart V \semidirect[*] \GroupPart V\) is the coskeleton of \(V\). (\footnote{The category of crossed modules is equivalent to the category of (\newnotion{strict}) \newnotion{categorical groups}, cf.~\cite[thm.~1]{brown_spencer:1976:g-groupoids_crossed_modules_and_the_fundamental_groupoid_of_a_topological_group},~\cite[sec.~6, thm.]{porst:2008:strict_2-groups_are_crossed_modules} or~\cite[thm.~(5.25)]{thomas:2007:co_homology_of_crossed_modules}. The coskeleton functor from crossed modules to simplicial groups can be obtained via a nerve functor from the category of categorical groups to the category of simplicial groups. Cf.~\cite[sec.~1]{bullejos_cegarra_duskin:1993:on_catn-groups_and_homotopy_types} and~\cite[ch.~VI, \S\S 1--2]{thomas:2007:co_homology_of_crossed_modules}. For another truncation-coskeleton-pair, cf.~\cite[exp.~V, sec.~7.1]{artin_grothendieck_verdier:1972:sga_4_2} and~\cite[sec.~(0.7)]{duskin:1975:simplicial_methods_and_the_interpretation_of_triple_cohomology}.})

Finally, we suppose given a group \(G\). Then we obtain a simplicial group \(\Coskeleton_0 G := \ConstantFunctor G\), the \newnotion{coskeleton simplicial group} (or just the \newnotion{coskeleton}) of \(G\). Moreover, we obtain a crossed module \(\Coskeleton_0^1 G\), called the \newnotion{coskeleton crossed module} (or the \newnotion{\(1\)-coskeleton}) of \(G\), as follows. The group part is defined by \(\GroupPart \Coskeleton_0 G := G\) and the module part is defined by \(\ModulePart \Coskeleton_0 G := \{1\}\) (and the structure morphism and the action of the group part on the module part are both trivial).

All mentioned truncation and coskeleton constructions are functorial and the resulting truncation functors are left adjoint to the resulting coskeleton functors. The unit \(\unit\colon \id_{\sGrp} \map \Coskeleton_0 \comp \Truncation^0\) is given by \(g_n (\unit_G)_n = g_n \face{\descendinginterval{n}{1}} \BoundaryGroup[0] \MooreComplex G\) for \(g_n \in G_n\), \(n \in \Naturals_0\), \(G \in \Ob \sGrp\), cf.~\cite[prop.~(4.15)]{thomas:2007:co_homology_of_crossed_modules}. The unit \(\unit\colon \id_{\sGrp} \map \Coskeleton_1 \comp \Truncation^1\) fulfills \(g_0 (\unit_G)_0 = (g_0)\) for \(g_0 \in G_0\) and \(g_1 (\unit_G)_1 = (g_1 (g_1 \face{1} \degeneracy{0})^{- 1} \BoundaryGroup[1] \MooreComplex G, g_1 \face{1})\) for \(g_1 \in G_1\), \(G \in \Ob \sGrp\), cf.\ for example~\cite[def.~(6.11), def.~(6.15), rem.~(6.14), prop.~(6.9), th.~(5.25)]{thomas:2007:co_homology_of_crossed_modules}.

Further, we have \(\Truncation^0 \comp \Coskeleton_0 \isomorphic \id_{\Grp}\), \(\Truncation^1 \comp \Coskeleton_1 \isomorphic \id_{\CrMod}\) and \(\Truncation_1^0 \comp \Coskeleton_0^1 \isomorphic \id_{\Grp}\), as well as \(\Coskeleton_0 = \Coskeleton_1 \comp \Coskeleton_0^1\) and \(\Truncation^0 = \Truncation_1^0 \comp \Truncation^1\).
\[\begin{tikzpicture}[baseline=(m-2-1.base)]
  \matrix (m) [matrix of math nodes, row sep=4.0em, column sep=6.5em, text height=1.6ex, text depth=0.45ex, inner sep=0pt, nodes={inner sep=0.333em}]{
    \sGrp & \CrMod \\
    \sGrp & \Grp \\};
  \path[->, font=\scriptsize]
    (m-1-1) edge[equality] (m-2-1)
    ([yshift=0.5ex]m-1-1.east) edge node[above] {\(\Truncation^1\)} ([yshift=0.5ex]m-1-2.west)
    ([yshift=-0.5ex]m-1-2.west) edge node[below] {\(\Coskeleton_1\)} ([yshift=-0.5ex]m-1-1.east)
    ([xshift=0.5ex]m-1-2.south) edge node[right] {\(\Truncation_1^0\)} ([xshift=0.5ex]m-2-2.north)
    ([xshift=-0.5ex]m-2-2.north) edge node[left] {\(\Coskeleton_0^1\)} ([xshift=-0.5ex]m-1-2.south)
    ([yshift=0.5ex]m-2-1.east) edge node[above] {\(\Truncation^0\)} ([yshift=0.5ex]m-2-2.west)
    ([yshift=-0.5ex]m-2-2.west) edge node[below] {\(\Coskeleton_0\)} ([yshift=-0.5ex]m-2-1.east);
\end{tikzpicture}\]

Given a group \(G\), we have
\[\MooreComplex(\Coskeleton_0 G) = (\dots \morphism 1 \morphism 1 \morphism G),\]
and given a crossed module \(V\), we have
\[\MooreComplex(\Coskeleton_1 V) = (\dots \morphism 1 \morphism \ModulePart V \morphism[{\structuremorphism[V]}] \GroupPart V);\]
cf.~\cite[prop.~(6.22)]{thomas:2007:co_homology_of_crossed_modules}.

\subsection{Homotopy groups} \label{ssec:homotopy_groups}

Given a simplicial group \(G\), we call \(\HomotopyGroup[n](G) := \HomologyGroup[n] \MooreComplex G\) the \newnotion{\(n\)-th homotopy group} of \(G\) for \(n \in \Naturals_0\). It is abelian for \(n \in \Naturals\), and we have \(\HomotopyGroup[0] = \Truncation^0\).

The \newnotion{homotopy groups} of a crossed module \(V\) are defined by
\[\HomotopyGroup[n](V) := \begin{cases} \Cokernel \structuremorphism & \text{for \(n = 0\),} \\ \Kernel \structuremorphism & \text{for \(n = 1\),} \\ \{1\} & \text{for \(n \in \Naturals_0 \setminus \{0, 1\}\).} \end{cases}\]
As remarked in section~\ref{ssec:crossed_modules}, the first homotopy group \(\HomotopyGroup[1](V)\) carries the structure of an abelian \(\HomotopyGroup[0](V)\)-module, where the action of \(\HomotopyGroup[0](V)\) on \(\HomotopyGroup[1](V)\) is induced by the action of \(\GroupPart V\) on \(\ModulePart V\), that is, for \(k \in \HomotopyGroup[1](V)\) and \(p \in \HomotopyGroup[0](V)\) we have \({^p}k = {^g}k\) for any \(g \in \GroupPart V\) with \(g (\Image \structuremorphism) = p\).

We remark that the definitions of the homotopy groups of a crossed module are adapted to the definition of the homotopy groups of a simplicial group: Given a crossed module \(V\), we have \(\HomotopyGroup[n](V) \isomorphic \HomotopyGroup[n](\Coskeleton_1 V)\) for all \(n \in \Naturals_0\), cf.\ for example~\cite[ch.~VI, \S 3]{thomas:2007:co_homology_of_crossed_modules}. Moreover, given a simplicial group \(G\), we have \(\HomotopyGroup[n](G) = \HomotopyGroup[n](\Truncation^1 G)\) for \(n \in \{0, 1\}\). (\footnote{In particular, given a simplicial group \(G\), we have \(\HomotopyGroup[n](G) \isomorphic \HomotopyGroup[n](\Coskeleton_1 \Truncation^1 G)\) for \(n \in \{0, 1\}\). This property fails for the truncation-coskeleton pair in~\cite[sec.~(0.7)]{duskin:1975:simplicial_methods_and_the_interpretation_of_triple_cohomology}.})

\subsection{Semidirect product decomposition} \label{ssec:semidirect_product_decomposition}

We suppose given a simplicial group \(G\). The group of \(n\)-simplices \(G_n\), where \(n \in \Naturals_0\), is isomorphic to an iterated semidirect product in terms of the entries \(\MooreComplex[k] G\) for \(k \in [0, n]\) of the Moore complex \(\MooreComplex G\). For example, we have \(G_0 = \MooreComplex[0] G\) and \(G_1 \isomorphic \MooreComplex[1] G \semidirect \MooreComplex[0] G\) and \(G_2 \isomorphic (\MooreComplex[2] G \semidirect \MooreComplex[1] G) \semidirect (\MooreComplex[1] G \semidirect \MooreComplex[0] G)\), where \(\MooreComplex[0] G\) acts on \(\MooreComplex[1] G\) via \({^{g_0}}g_1 := {^{g_0 \degeneracy{0}}}g_1\) for \(g_i \in \MooreComplex[i] G\), \(i \in \{0, 1\}\), \(\MooreComplex[1] G\) acts on \(\MooreComplex[2] G\) via \({^{g_1}}g_2 := {^{g_1 \degeneracy{0}}}g_2\) for \(g_i \in \MooreComplex[i] G\), \(i \in \{1, 2\}\) and \(\MooreComplex[1] G \semidirect \MooreComplex[0] G\) acts on \(\MooreComplex[2] G \semidirect \MooreComplex[1] G\) via \({^{(g_1, g_0)}}(g_2, h_1) := ({^{(g_1 \degeneracy{1}) (g_0 \degeneracy{0} \degeneracy{1})}}(g_2 (h_1 \degeneracy{0})) \, {^{(g_1 \degeneracy{0}) (g_0 \degeneracy{0} \degeneracy{1})}}((h_1 \degeneracy{0})^{- 1}), {^{g_1 (g_0 \degeneracy{0})}}h_1)\) for \(g_i, h_i \in \MooreComplex[i] G\), \(i \in [0, 2]\). The isomorphisms are given by
\begin{align*}
& \varphi_1\colon G_1 \map \MooreComplex[1] G \semidirect \MooreComplex[0] G, g_1 \mapsto (g_1 (g_1 \face{1} \degeneracy{0})^{- 1}, g_1 \face{1}), \\
& \varphi_1^{- 1}\colon \MooreComplex[1] G \semidirect \MooreComplex[0] G \map G_1, (g_1, g_0) \mapsto g_1 (g_0 \degeneracy{0})
\end{align*}
and
\begin{align*}
& \varphi_2\colon G_2 \map (\MooreComplex[2] G \semidirect \MooreComplex[1] G) \semidirect (\MooreComplex[1] G \semidirect \MooreComplex[0] G), \\
& \qquad g_2 \mapsto ((g_2 (g_2 \face{2} \degeneracy{1})^{- 1} (g_2 \face{2} \degeneracy{0}) (g_2 \face{1} \degeneracy{0})^{- 1}, (g_2 \face{1}) (g_2 \face{2})^{- 1}), ((g_2 \face{2}) (g_2 \face{2} \face{1} \degeneracy{0})^{- 1}, g_2 \face{2} \face{1})), \\
& \varphi_2^{- 1}\colon (\MooreComplex[2] G \semidirect \MooreComplex[1] G) \semidirect (\MooreComplex[1] G \semidirect \MooreComplex[0] G) \map G_2, ((g_2, h_1), (g_1, g_0)) \mapsto g_2 (h_1 \degeneracy{0}) (g_1 \degeneracy{1}) (g_0 \degeneracy{0} \degeneracy{1}).
\end{align*}
For more details, see~\cite{carrasco_cegarra:1991:group-theoretic_algebraic_models_for_homotopy_types} or~\cite[ch.~IV, \S 2]{thomas:2007:co_homology_of_crossed_modules}.

\subsection{Cohomology of simplicial sets} \label{ssec:cohomology_of_simplicial_sets}

We suppose given a simplicial set \(X\) and an abelian group \(A\). The \newnotion{cochain complex} of \(X\) with coefficients in \(A\) is the complex of abelian groups \(\CochainComplex_{\sSet}(X, A)\) with abelian groups \(\CochainComplex[n]_{\sSet}(X, A) := \Map(X_n, A)\) for \(n \in \Naturals_0\) and differentials defined by \(x (c \differential) := \sum_{k \in [0, n + 1]} (- 1)^k (x \face{k}) c\) for \(x \in X_{n + 1}\), \(c \in \CochainComplex[n]_{\sSet}(X, A)\). We call \(\CochainComplex[n]_{\sSet}(X, A)\) the \newnotion{\(n\)-th cochain group} of \(X\) with coefficients in \(A\). Moreover, we define the \newnotion{\(n\)-th cocycle group} \(\CocycleGroup[n]_{\sSet}(X, A) := \CocycleGroup[n] \CochainComplex_{\sSet}(X, A)\), the \newnotion{\(n\)-th coboundary group} \(\CoboundaryGroup[n]_{\sSet}(X, A) := \CoboundaryGroup[n] \CochainComplex_{\sSet}(X, A)\) and the \newnotion{\(n\)-th cohomology group} \(\CohomologyGroup[n]_{\sSet}(X, A) := \CohomologyGroup[n] \CochainComplex_{\sSet}(X, A) = \CocycleGroup[n]_{\sSet}(X, A) / \CoboundaryGroup[n]_{\sSet}(X, A)\) of \(X\) with coefficients in \(A\) (\footnote{In the literature, \(\CocycleGroup[n]_{\sSet}(X, A)\) resp.\ \(\CoboundaryGroup[n]_{\sSet}(X, A)\) resp.\ \(\CohomologyGroup[n]_{\sSet}(X, A)\) are often defined by an isomorphic complex of abelian groups (cf.\ for example~\cite[def.~(2.18)]{thomas:2007:co_homology_of_crossed_modules}) and are just denoted \(\CocycleGroup[n](X, A)\) resp.\ \(\CoboundaryGroup[n](X, A)\) resp.\ \(\CohomologyGroup[n](X, A)\).}). An element \(c \in \CochainComplex[n]_{\sSet}(X, A)\) resp.\ \(z \in \CocycleGroup[n]_{\sSet}(X, A)\) resp.\ \(b \in \CoboundaryGroup[n]_{\sSet}(X, A)\) resp.\ \(h \in \CohomologyGroup[n]_{\sSet}(X, A)\) is said to be an \newnotion{\(n\)-cochain} resp.\ an \newnotion{\(n\)-cocycle} resp.\ an \newnotion{\(n\)-coboundary} resp.\ an \newnotion{\(n\)-cohomology class} of \(X\) with coefficients in \(A\).

\subsection{Nerves} \label{ssec:nerves}

Given a group \(G\), we define a simplicial set \(\Nerve G\), called the \newnotion{nerve} of \(G\), by \(\Nerve[n] G := G^{\directprod n}\) for \(n \in \Naturals_0\) and by
\[(g_j)_{j \in \descendinginterval{n - 1}{0}} (\Nerve[\theta] G) := (\prod_{j \in \descendinginterval{(i + 1) \theta - 1}{i \theta}} g_j)_{i \in \descendinginterval{m - 1}{0}}\]
for \((g_j)_{j \in \descendinginterval{n - 1}{0}} \in \Nerve[n] G\) and \(\theta \in {_\SimplexTypes}(\CatNaturalNumber{m}, \CatNaturalNumber{n})\), where \(m, n \in \Naturals_0\).

Moreover, given a simplicial group \(G\), the nerve functor for groups (in a suitable Grothendieck universe) yields the \newnotion{nerve} of the simplicial group \(G\), defined by \(\Nerve G := \Nerve \comp G\). We obtain a simplicial object in \(\sSet\). By abuse of notation, we also denote the corresponding bisimplicial set by \(\Nerve G\).
\[\begin{tikzpicture}[baseline=(m-2-1.base)]
  \matrix (m) [diagram]{
    & \sSet \\
    \SimplexTypes^\op & \Grp \\};
  \path[->, font=\scriptsize]
    (m-2-1) edge node[above] {\(G\)} (m-2-2)
            edge node[left] {\(\Nerve G\)} (m-1-2)
    (m-2-2) edge node[right] {\(\Nerve\)} (m-1-2);
\end{tikzpicture}\]

\subsection{Cohomology of simplicial groups with coefficients in an abelian group} \label{ssec:cohomology_of_simplicial_groups_with_coefficients_in_an_abelian_group}

Cohomology of simplicial sets can be used to define cohomology of a simplicial group \(G\). This is done via the \newnotion{Kan classifying simplicial set} \(\KanClassifyingSimplicialSet G\) of \(G\), see \eigenname{Kan}~\cite[def.~10.3]{kan:1958:on_homotopy_theory_and_c_s_s_groups}, which is given by \(\KanClassifyingSimplicialSet_n G := \bigcart_{j \in \descendinginterval{n - 1}{0}} G_j\) for all \(n \in \Naturals_0\) and
\[(g_j)_{j \in \descendinginterval{n - 1}{0}} (\KanClassifyingSimplicialSet_{\theta} G) := (\prod_{j \in \descendinginterval{(i + 1) \theta - 1}{i \theta}} g_j G_{\theta|_{\CatNaturalNumber{i}}^{\CatNaturalNumber{j}}})_{i \in \descendinginterval{m - 1}{0}}\]
for \((g_j)_{j \in \descendinginterval{n - 1}{0}} \in \KanClassifyingSimplicialSet_n G\) and \(\theta \in {_{\SimplexTypes}}(\CatNaturalNumber{m}, \CatNaturalNumber{n})\), where \(m, n \in \Naturals_0\), cf.\ for example~\cite[rem.~(4.19)]{thomas:2007:co_homology_of_crossed_modules}. In particular, the faces are given by
\[(g_j)_{j \in \descendinginterval{n - 1}{0}} \face{k}^{\KanClassifyingSimplicialSet G} = \begin{cases} (g_{j + 1} \face{0}^{G})_{j \in \descendinginterval{n - 2}{0}} & \text{for \(k = 0\),} \\ (g_{j + 1} \face{k}^{G})_{j \in \descendinginterval{n - 2}{k}} \union ((g_k \face{k}^{G}) g_{k - 1}) \union (g_j)_{j \in \descendinginterval{k - 2}{0}} & \text{for \(k \in [1, n - 1]\),} \\ (g_j)_{j \in \descendinginterval{n - 2}{0}} & \text{for \(k = n\),} \end{cases}\]
for \((g_j)_{j \in \descendinginterval{n - 1}{0}} \in \KanClassifyingSimplicialSet_n G\), \(n \in \Naturals\). The \newnotion{cochain complex} of \(G\) with coefficients in an abelian group \(A\) is defined to be \(\CochainComplex(G, A) = \CochainComplex_{\sGrp}(G, A) := \CochainComplex_{\sSet}(\KanClassifyingSimplicialSet G, A)\). Moreover, we define the \newnotion{\(n\)-th cocycle group} \(\CocycleGroup[n](G, A) = \CocycleGroup[n]_{\sGrp}(G, A) := \CocycleGroup[n]_{\sSet}(\KanClassifyingSimplicialSet G, A)\), etc., for \(n \in \Naturals_0\). The differentials of \(\CochainComplex(G, A)\) are given by
\begin{align*}
(g_j)_{j \in \descendinginterval{n}{0}} (c \differential) & = (g_{j + 1} \face{0})_{j \in \descendinginterval{n - 1}{0}} c + \sum_{k \in [1, n]} (- 1)^k ((g_{j + 1} \face{k})_{j \in \descendinginterval{n - 1}{k}} \union ((g_k \face{k}) g_{k - 1}) \union (g_j)_{j \in \descendinginterval{k - 2}{0}}) c 
\\ & \qquad + (- 1)^{n + 1} (g_j)_{j \in \descendinginterval{n - 1}{0}} c
\end{align*}
for \((g_j)_{j \in \descendinginterval{n}{0}} \in \KanClassifyingSimplicialSet_{n + 1} G\), \(c \in \CochainComplex[n](G, A)\), \(n \in \Naturals_0\).

Instead of \(\KanClassifyingSimplicialSet G\), one can also use \(\DiagonalFunctor \Nerve G\), the diagonal simplicial set of the nerve of \(G\), see for example~\cite[app.~Q.3]{friedlander_mazur:1994:filtrations_on_the_homology_of_algebraic_varieties},~\cite[p.~41]{jardine:1981:algebraic_homotopy_theory_groups_and_k-theory} and~\cite{thomas:2007:co_homology_of_crossed_modules}. The simplicial sets \(\DiagonalFunctor \Nerve G\) and \(\KanClassifyingSimplicialSet G\) are simplicially homotopy equivalent~\cite[thm.]{thomas:2008:the_functors_wbar_and_diag_nerve_are_simplicially_homotopy_equivalent}, cf.\ also~\cite[thm.~1.1]{cegarra_remedios:2005:the_relationship_between_the_diagonal_and_the_bar_constructions_on_a_bisimplicial_set}, and thus \(\CohomologyGroup[n](G, A) = \CohomologyGroup[n]_{\sSet}(\KanClassifyingSimplicialSet G, A) \isomorphic \CohomologyGroup[n]_{\sSet}(\DiagonalFunctor \Nerve G, A)\) for \(n \in \Naturals_0\), where \(A\) is an abelian group. 

\subsection{Cohomology of simplicial groups with coefficients in an abelian module} \label{ssec:cohomology_of_simplicial_groups_with_coefficients_in_an_abelian_module}

To generalise cohomology of a simplicial group \(G\) with coefficients in an abelian group \(A\) to cohomology with coefficients in an abelian \(\HomotopyGroup[0](G)\)-module \(M\), we have to introduce a further notion on simplicial sets: Given a simplicial set \(X\), the \newnotion{path simplicial set} of \(X\) is the simplicial set \(\PathSimplicialObject X := X \circ (\ShiftUp)^\op\), which is simplicially homotopy equivalent to \(\ConstantFunctor X_0\)~\cite[8.3.14]{weibel:1997:an_introduction_to_homological_algebra}. The faces \(\face{n + 1}^X\colon \PathSimplicialObject_n X \map X_n\) for \(n \in \Naturals_0\) form a canonical simplicial map \(\PathSimplicialObject X \map X\).

Now we follow \eigenname{Quillen}~\cite[ch.~II, p.~6.16]{quillen:1967:homotopical_algebra} and consider for a given simplicial group \(G\) the \newnotion{Kan resolving simplicial set} \(\KanResolvingSimplicialSet G := \PathSimplicialObject \KanClassifyingSimplicialSet G\). The simplicial group \(G\) acts on \(\KanResolvingSimplicialSet G\) by \(g (g_j)_{j \in \descendinginterval{n}{0}} := (g g_n) \union (g_j)_{j \in \descendinginterval{n - 1}{0}}\) for \(g \in G_n\), \((g_j)_{j \in \descendinginterval{n}{0}} \in \KanResolvingSimplicialSet_n G\), \(n \in \Naturals_0\), and the canonical simplicial map \(\KanResolvingSimplicialSet G \map \KanClassifyingSimplicialSet G\) given by \(\KanResolvingSimplicialSet_n G \map \KanClassifyingSimplicialSet_n G, (g_j)_{j \in \descendinginterval{n}{0}} \mapsto (g_j)_{j \in \descendinginterval{n - 1}{0}}\) induces a simplicial bijection \(\KanResolvingSimplicialSet G / G \map \KanClassifyingSimplicialSet G\).

We suppose given an abelian \(\HomotopyGroup[0](G)\)-module \(M\). Then \(\ConstantFunctor M\) is a simplicial abelian \(\HomotopyGroup[0](G)\)-module, and the unit \(\unit\colon \id_{\sGrp} \map \Coskeleton_0 \comp \HomotopyGroup[0]\) of the adjunction \(\HomotopyGroup[0] = \Truncation^0 \leftadjoint \Coskeleton_0\) turns \(\ConstantFunctor M\) into an abelian \(G\)-simplicial module via \(g_n x_n := (g_n (\unit_G)_n) x_n = (g_n \face{\descendinginterval{n}{1}} \BoundaryGroup[0] \MooreComplex G) x_n\) for \(g_n \in G_n\), \(x_n \in M_n\), \(n \in \Naturals_0\). Since \(\unit_G\) is a simplicial group homomorphism, we have \(g_n G_{\theta} (\unit_G)_m = g_n (\unit_G)_n\) for all \(g_n \in G_n\), \(\theta \in {_{\SimplexTypes}}(\CatNaturalNumber{m}, \CatNaturalNumber{n})\), \(m, n \in \Naturals_0\).

We consider the subcomplex \(\CochainComplex_{\text{hom}}(G, M) = \CochainComplex_{\sGrp, \text{hom}}(G, M)\) of the cochain complex \(\CochainComplex_{\sSet}(\KanResolvingSimplicialSet G, M)\) with entries \(\CochainComplex[n]_{\text{hom}}(G, M) := \Map_{G_n}(\KanResolvingSimplicialSet_n G, M)\) and differentials given by
\[(g_j)_{j \in \descendinginterval{n + 1}{0}} (c \differential) := \sum_{k \in [0, n + 1]} (- 1)^k ((g_j)_{j \in \descendinginterval{n + 1}{0}} \face{k}) c\]
for \((g_j)_{j \in \descendinginterval{n + 1}{0}} \in \KanResolvingSimplicialSet_{n + 1} G\), \(c \in \CochainComplex[n]_{\text{hom}}(G, M)\), \(n \in \Naturals_0\), called the \newnotion{homogeneous cochain complex} of \(G\) with coefficients in \(M\). We want to introduce an isomorphic variant of \(\CochainComplex_{\text{hom}}(G, M)\) using transport of structure. We have
\begin{align*}
& (g_j)_{j \in \descendinginterval{n + 1}{0}} (c \differential) = g_{n + 1} \face{\descendinginterval{n + 1}{1}} \BoundaryGroup[0] \MooreComplex G \act \big( ((1) \union (g_{j + 1} \face{0})_{j \in \descendinginterval{n - 1}{0}}) c \\
& \qquad + \sum_{k \in [1, n]} (- 1)^k ((1) \union (g_{j + 1} \face{k})_{j \in \descendinginterval{n - 1}{k}} \union ((g_k \face{k}) g_{k - 1}) \union (g_j)_{j \in \descendinginterval{k - 2}{0}}) c \\
& \qquad + (- 1)^{n + 1} (g_n \face{\descendinginterval{n}{1}} \BoundaryGroup[0] \MooreComplex G) \act ((1) \union (g_j)_{j \in \descendinginterval{n - 1}{0}}) c \big)
\end{align*}
for \((g_j)_{j \in \descendinginterval{n + 1}{0}} \in \KanResolvingSimplicialSet_{n + 1} G\), \(c \in \CochainComplex[n]_{\text{hom}}(G, M)\), \(n \in \Naturals_0\). Thus \(\CochainComplex_{\text{hom}}(G, M)\) is isomorphic to a complex \(\CochainComplex(G, M)\), called the \newnotion{cochain complex} of \(G\) with coefficients in the abelian \(\HomotopyGroup[0](G)\)-module \(M\), with entries \(\CochainComplex[n](G, M) := \Map(\bigcart_{j \in \descendinginterval{n - 1}{0}} G_j, M) = \CochainComplex[n]_{\sSet}(\KanClassifyingSimplicialSet G, M)\) and differentials given by
\begin{align*}
(g_j)_{j \in \descendinginterval{n}{0}} (c \differential) & = (g_{j + 1} \face{0})_{j \in \descendinginterval{n - 1}{0}} c + \sum_{k \in [1, n]} (- 1)^k ((g_{j + 1} \face{k})_{j \in \descendinginterval{n - 1}{k}} \union ((g_k \face{k}) g_{k - 1}) \union (g_j)_{j \in \descendinginterval{k - 2}{0}}) c \\
& \qquad + (- 1)^{n + 1} (g_n \face{\descendinginterval{n}{1}} \BoundaryGroup[0] \MooreComplex G) \act (g_j)_{j \in \descendinginterval{n - 1}{0}} c
\end{align*}
for \((g_j)_{j \in \descendinginterval{n}{0}} \in \KanClassifyingSimplicialSet_{n + 1} G\), \(c \in \CochainComplex[n](G, M)\), \(n \in \Naturals_0\), and where an isomorphism \(\varphi\colon \CochainComplex_{\text{hom}}(G, M) \map \CochainComplex(G, M)\) is given by
\[(g_j)_{j \in \descendinginterval{n - 1}{0}} (c \varphi^n) = ((1) \union (g_j)_{j \in \descendinginterval{n - 1}{0}}) c\]
for \((g_j)_{j \in \descendinginterval{n - 1}{0}} \in \KanClassifyingSimplicialSet_n G\), \(c \in \CochainComplex[n]_{\text{hom}}(G, M)\), \(n \in \Naturals_0\). Moreover, we set \(\CocycleGroup[n]_{\sGrp}(G, M) = \CocycleGroup[n](G, M) := \CocycleGroup[n](\CochainComplex(G, M))\), etc., and call \(\CochainComplex[n](G, M)\) the \newnotion{\(n\)-th cochain group} of \(G\) with coefficients in \(M\), etc. We see that this definition coincides with \(\CochainComplex(G, A)\) for an abelian group \(A\) considered as an abelian \(\HomotopyGroup[0](G)\)-module with the trivial action of \(\HomotopyGroup[0](G)\).

Isomorphic substitution of \(G\) with its semidirect product decomposition, cf.\ section~\ref{ssec:semidirect_product_decomposition}, leads to an isomorphic substitution of the cochain complex \(\CochainComplex(G, M)\) to the \newnotion{analysed cochain complex} \(\CochainComplex_{\text{an}}(G, M) = \CochainComplex_{\sGrp, \text{an}}(G, M)\). Similarly, isomorphic substitution yields \(\CocycleGroup[n]_{\text{an}}(G, M) = \CocycleGroup[n]_{\sGrp, \text{an}}(G, M)\), etc., and we call \(\CochainComplex[n]_{\text{an}}(G, M)\) the \newnotion{\(n\)-th analysed cochain group} of \(G\) with coefficients in \(M\), etc. See~\ref{wb:low_dimensional_analysed_cocycles} for formulas in low dimensions.

Altogether, we have
\[\CochainComplex_{\text{hom}}(G, M) \isomorphic \CochainComplex(G, M) \isomorphic \CochainComplex_{\text{an}}(G, M).\]

\subsection{Cohomology of groups and cohomology of crossed modules} \label{ssec:cohomology_of_groups_and_cohomology_of_crossed_modules}

Since groups and crossed modules can be considered as truncated simplicial groups, the cohomology groups of these algebraic objects is defined via cohomology of simplicial groups.

Given a group \(G\) and an abelian \(G\)-module \(M\), we define the \newnotion{cochain complex} \(\CochainComplex(G, M) = \CochainComplex_{\Grp}(G, M) \linebreak % manual format
:= \CochainComplex_{\sGrp}(\Coskeleton_0 G, M)\) of \(G\) with coefficients in \(M\). Similarly, we set \(\CocycleGroup[n](G, M) = \CocycleGroup[n]_{\Grp}(G, M) := \linebreak % manual format
\CocycleGroup[n]_{\sGrp}(\Coskeleton_0 G, M)\) for \(n \in \Naturals_0\), etc., and call \(\CochainComplex[n](G, M)\) the \newnotion{\(n\)-th cochain group} of \(G\) with coefficients in \(M\), etc. Since \(\DiagonalFunctor \Nerve \Coskeleton_0 G = \KanClassifyingSimplicialSet \Coskeleton_0 G = \Nerve G\), this definition of cohomology coincides with the standard one via \(\ClassifyingSimplicialSet G := \Nerve G\) and \(\ResolvingSimplicialSet G := \PathSimplicialObject \ClassifyingSimplicialSet G\).

Given a crossed module \(V\) and an abelian \(\HomotopyGroup[0](V)\)-module \(M\), we define the \newnotion{cochain complex} \(\CochainComplex(V, M) = \CochainComplex_{\CrMod}(V, M) := \CochainComplex_{\sGrp}(\Coskeleton_1 V, M)\) of \(V\) with coefficients in \(M\). Similarly, we set \(\CocycleGroup[n](V, M) = \linebreak % manual format
\CocycleGroup[n]_{\CrMod}(V, M) := \CocycleGroup[n]_{\sGrp}(\Coskeleton_1 V, M)\) for \(n \in \Naturals_0\), etc., and call \(\CochainComplex[n](V, M)\) the \newnotion{\(n\)-th cochain group} of \(V\) with coefficients in \(M\), etc.

\[\begin{tikzpicture}[baseline=(m-1-1.base)]
  \matrix (m) [matrix of math nodes, row sep=2.5em, column sep=4.5em, text height=1.6ex, text depth=0.45ex, inner sep=0pt, nodes={inner sep=0.333em}]{
    \Grp & \CrMod & \sGrp & \sSet & \Complexes(\AbGrp) & \AbGrp \\};
  \path[->, font=\scriptsize]
    (m-1-1) edge node[above] {\(\Coskeleton_0^1\)} (m-1-2)
            edge[out=-15, in=-165] node[below] {\(\Coskeleton_0\)} (m-1-3)
    (m-1-2) edge node[above] {\(\Coskeleton_1\)} (m-1-3)
    ([yshift=0.5ex]m-1-3.east) edge node[above] {\(\KanClassifyingSimplicialSet\)} ([yshift=0.5ex]m-1-4.west)
    ([yshift=-0.5ex]m-1-3.east) edge node[below] {\(\DiagonalFunctor \comp \Nerve\)} ([yshift=-0.5ex]m-1-4.west)
    (m-1-4) edge node[above] {\(\CochainComplex_{\sSet}(-, M)\)} (m-1-5)
    (m-1-5) edge node[above] {\(\CohomologyGroup[n]\)} (m-1-6);
\end{tikzpicture}\]

The semidirect product decomposition of \(\Coskeleton_1 V\) is -- up to simplified notation -- already built into the definition of \(\Coskeleton_1 V\). So the cochain complex and the analysed cochain complex of \(\Coskeleton_1 V\) are essentially equal. Therefore there is no need to explicitly introduce analysed cochains for crossed modules.

\eigenname{Ellis} defines in~\cite[sec.~3]{ellis:1992:homology_of_2-types} the cohomology of a crossed module \(V\) with coefficients in an abelian group \(A\) via the composition \(\DiagonalFunctor \comp \Nerve\). In this article, we will make use of the Kan classifying simplicial set functor \(\KanClassifyingSimplicialSet\) instead of \(\DiagonalFunctor \comp \Nerve\) since \(\KanClassifyingSimplicialSet\) provides smaller objects, which is more convenient for direct calculations. For example, a \(2\)-cocycle in \(\CocycleGroup[2](\DiagonalFunctor \Nerve \Coskeleton_1 V, A)\) is a map \((\ModulePart V)^{\cart 4} \cart (\GroupPart V)^{\cart 2} \map A\), while a \(2\)-cocycle in \(\CocycleGroup[2](V, A) = \CocycleGroup[2](\KanClassifyingSimplicialSet \Coskeleton_1 V, A)\) is a map \(\ModulePart V \cart (\GroupPart V)^{\cart 2} \map A\).

\subsection{Pointed cochains} \label{ssec:pointed_cochains}

We let \(G\) be a simplicial group and \(M\) be an abelian \(\HomotopyGroup[0](G)\)-module. As we have seen above, an \(n\)-cochain of \(G\) with coefficients in \(M\) is just a map \(c\colon \KanClassifyingSimplicialSet_n G \map M\), where \(n \in \Naturals_0\). Since the sets \(\KanClassifyingSimplicialSet_n G = \bigcart_{j \in \descendinginterval{n - 1}{0}} G_j\) carry structures as direct products of groups for \(n \in \Naturals_0\), they are pointed in a natural way with \(1 = (1)_{\descendinginterval{n - 1}{0}}\) as distinguished points. Moreover, the module \(M\) is in particular an abelian group and therefore pointed with \(0\) as distinguished point. An \(n\)-cochain \(c \in \CochainComplex[n](G, M)\) is said to be \newnotion{pointed} if it is a pointed map, that is, if \(1 c = 0\). The subset of \(\CochainComplex[n](G, M)\) consisting of all pointed \(n\)-cochains of \(G\) with coefficients in \(M\) will be denoted by \(\CochainComplex[n]_{\text{pt}}(G, M) := \{c \in \CochainComplex[n](G, M) \mid \text{\(c\) is pointed}\}\). We set \(\CocycleGroup[n]_{\text{pt}}(G, M) := \CochainComplex[n]_{\text{pt}}(G, M) \intersection \CocycleGroup[n](G, M)\) for the set of pointed \(n\)-cocycles, \(\CoboundaryGroup[n]_{\text{pt}}(G, M) := \CochainComplex[n]_{\text{pt}}(G, M) \intersection \CoboundaryGroup[n](G, M)\) for the set of pointed \(n\)-coboundaries and \(\CohomologyGroup[n]_{\text{pt}}(G, M) := \CocycleGroup[n]_{\text{pt}}(G, M) / \CoboundaryGroup[n]_{\text{pt}}(G, M)\) for the set of pointed \(n\)-cohomology classes of \(G\) with coefficients in \(M\).

We suppose given an odd natural number \(n \in \Naturals\). Every \(n\)-cocycle \(z \in \CocycleGroup[n](G, M)\) is pointed, and hence we have \(\CocycleGroup[n]_{\text{pt}}(G, M) = \CocycleGroup[n](G, M)\), \(\CoboundaryGroup[n]_{\text{pt}}(G, M) = \CoboundaryGroup[n](G, M)\) and \(\CohomologyGroup[n]_{\text{pt}}(G, M) = \CohomologyGroup[n](G, M)\).
Moreover, we have \(\CoboundaryGroup[n + 1]_{\text{pt}}(G, M) = (\CochainComplex[n]_{\text{pt}}(G, M)) \differential\).

So we suppose given an even natural number \(n \in \Naturals\) and an \(n\)-cocycle \(z \in \CocycleGroup[n](G, M)\). The \newnotion{pointisation} of \(z\) is given by \(\pointisation{z} := z - \pointiser{z} \differential\), where the \newnotion{pointiser} of \(z\) is defined to be the \((n - 1)\)-cochain \(\pointiser{z} \in \CochainComplex[n - 1](G, M)\) given by \((g_j)_{j \in \descendinginterval{n - 2}{0}} \pointiser{z} := (1)_{j \in \descendinginterval{n - 1}{0}} z\) for \(g_j \in G_j\), \(j \in \descendinginterval{n - 2}{0}\). We obtain
\[(g_j)_{j \in \descendinginterval{n - 1}{0}} \pointisation{z} = (g_j)_{j \in \descendinginterval{n - 1}{0}} z - g_{n - 1} \face{\descendinginterval{n - 1}{1}} \BoundaryGroup[0] \MooreComplex G \act (1)_{j \in \descendinginterval{n - 1}{0}} z\]
for \(g_j \in G_j\), \(j \in \descendinginterval{n - 1}{0}\). In particular, the pointisation \(\pointisation{z}\) of every \(z \in \CocycleGroup[n](G, M)\) is pointed. Moreover, we have \(\CocycleGroup[n]_{\text{pt}}(G, M) = \{z \in \CocycleGroup[n](G, M) \mid \pointisation{z} = z\}\) and the embedding \(\CocycleGroup[n]_{\text{pt}}(G, M) \map \CocycleGroup[n](G, M)\) and the pointisation homomorphism \(\CocycleGroup[n](G, M) \map \CocycleGroup[n]_{\text{pt}}(G, M), z \mapsto \pointisation{z}\) induce mutually inverse isomorphisms between \(\CohomologyGroup[n]_{\text{pt}}(G, M)\) and \(\CohomologyGroup[n](G, M)\).

Altogether, we have
\[\CohomologyGroup[n](G, M) \isomorphic \CohomologyGroup[n]_{\text{pt}}(G, M)\]
for all \(n \in \Naturals\).

Given a crossed module \(V\) and an abelian \(\HomotopyGroup[0](V)\)-module \(M\), we write \(\CochainComplex_{\text{pt}}(V, M) := \CochainComplex_{\text{pt}}(\Coskeleton_1 V, M)\), etc. Similarly, given a group \(G\) and an abelian \(G\)-module \(M\), we write \(\CochainComplex_{\text{pt}}(G, M) := \CochainComplex_{\text{pt}}(\Coskeleton_0 G, M)\), etc.

\subsection{Crossed module extensions} \label{ssec:crossed_module_extensions}

We suppose given a group \(\Pi_0\) and an abelian \(\Pi_0\)-module \(\Pi_1\), which will be written multiplicatively.

A \newnotion{crossed module extension} (or \newnotion{\(2\)-extension}) of \(\Pi_0\) with \(\Pi_1\) consists of a crossed module \(E\) together with a group monomorphism \(\iota\colon \Pi_1 \map \ModulePart E\) and a group epimorphism \(\pi\colon \GroupPart E \map \Pi_0\) such that
\[\Pi_1 \morphism[\iota] \ModulePart E \morphism[{\structuremorphism}] \GroupPart E \morphism[\pi] \Pi_0\]
is an exact sequence of groups and such that the induced action of \(\Pi_0\) on \(\Pi_1\) caused by the action of the crossed module \(E\) coincides with the a priori given action of \(\Pi_0\) on \(\Pi_1\), that is, such that \({^{g}}(k \iota) = ({^{g \pi}}k) \iota\) for \(g \in \GroupPart E\) and \(k \in \Pi_1\). By abuse of notation, we often refer to the crossed module extension as well as to its underlying crossed module by \(E\). The homomorphisms \(\iota\) resp.\ \(\pi\) are said to be the \newnotion{canonical monomorphism} resp.\ the \newnotion{canonical epimorphism} of the crossed module extension \(E\). Given a crossed module extension \(E\) of \(\Pi_0\) with \(\Pi_1\) with canonical monomorphism \(\iota\) and canonical epimorphism \(\pi\), we write \(\canonicalmonomorphism = \canonicalmonomorphism[E] := \iota\) and \(\canonicalepimorphism = \canonicalepimorphism[E] := \pi\).

We suppose given a Grothendieck universe \(\mathfrak{U}\). A crossed module extension is said to be \newnotion{in \(\mathfrak{U}\)} (or a \newnotion{\(\mathfrak{U}\)-crossed module extension}) if its underlying crossed module is in \(\mathfrak{U}\). The set of crossed module extensions in \(\mathfrak{U}\) of \(G\) with \(M\) will be denoted by \(\Extensions{2}(G, M) = \Extensions[\mathfrak{U}]{2}(G, M)\).

By definition, we have \(\HomotopyGroup[0](E) \isomorphic \Pi_0\) and \(\HomotopyGroup[1](E) \isomorphic \Pi_1\) for every crossed module extension \(E\) of \(\Pi_0\) with \(\Pi_1\). Conversely, given an arbitrary crossed module \(V\), we have the crossed module extension
\[\HomotopyGroup[1](V) \morphism[\inc] \ModulePart V \morphism[\structuremorphism] \GroupPart V \morphism[\quo] \HomotopyGroup[0](V),\]
again denoted by \(V\). That is, the canonical monomorphism of \(V\) is \(\canonicalmonomorphism[V] = \inc^{\HomotopyGroup[1](V)}\), and the canonical epimorphism is \(\canonicalepimorphism[V] = \quo^{\HomotopyGroup[0](V)}\).

We let \(E\) and \(\tilde E\) be crossed module extensions of \(\Pi_0\) with \(\Pi_1\). An \newnotion{extension equivalence} from \(E\) to \(\tilde E\) is a morphism of crossed modules \(\varphi\colon E \map \tilde E\) such that \(\canonicalmonomorphism[\tilde E] = \canonicalmonomorphism[E] (\ModulePart \varphi)\) and \(\canonicalepimorphism[E] = (\GroupPart \varphi) \canonicalepimorphism[\tilde E]\).
\[\begin{tikzpicture}[baseline=(m-2-1.base)]
  \matrix (m) [diagram]{
    \Pi_1 & \ModulePart E & \GroupPart E & \Pi_0 \\
    \Pi_1 & \ModulePart \tilde E & \GroupPart \tilde E & \Pi_0 \\};
  \path[->, font=\scriptsize]
    (m-1-1) edge node[above] {\(\canonicalmonomorphism[E]\)} (m-1-2)
            edge[equality] (m-2-1)
    (m-1-2) edge node[above] {\(\structuremorphism[E]\)} (m-1-3)
            edge node[right] {\(\ModulePart \varphi\)} (m-2-2)
    (m-1-3) edge node[above] {\(\canonicalepimorphism[E]\)} (m-1-4)
            edge node[right] {\(\GroupPart \varphi\)} (m-2-3)
    (m-1-4) edge[equality] (m-2-4)
    (m-2-1) edge node[above] {\(\canonicalmonomorphism[\tilde E]\)} (m-2-2)
    (m-2-2) edge node[above] {\(\structuremorphism[\tilde E]\)} (m-2-3)
    (m-2-3) edge node[above] {\(\canonicalepimorphism[\tilde E]\)} (m-2-4);
\end{tikzpicture}\]

We suppose given a Grothendieck universe \(\mathfrak{U}\) and we let \({\extensionequivalent} = {\extensionequivalent[\mathfrak{U}]}\) be the equivalence relation on \(\Extensions[\mathfrak{U}]{2}(\Pi_0, \Pi_1)\) generated by the following relation: Given extensions \(E, \tilde E \in \Extensions[\mathfrak{U}]{2}(\Pi_0, \Pi_1)\), the extension \(E\) is in relation to the extension \(\tilde E\) if there exists an extension equivalence \(E \map \tilde E\). Given \(\mathfrak{U}\)-crossed module extensions \(E\) and \(\tilde E\) with \(E \extensionequivalent \tilde E\), we say that \(E\) and \(\tilde E\) are \newnotion{extension equivalent}. The set of equivalence classes of crossed module extensions in \(\mathfrak{U}\) of \(\Pi_0\) with \(\Pi_1\) with respect to \({\extensionequivalent[\mathfrak{U}]}\) is denoted by \(\ExtensionClasses{2}(\Pi_0, \Pi_1) = \ExtensionClasses[\mathfrak{U}]{2}(\Pi_0, \Pi_1) := \Extensions[\mathfrak{U}]{2}(\Pi_0, \Pi_1) / {\extensionequivalent[\mathfrak{U}]}\), and an element of \(\ExtensionClasses{2}(\Pi_0, \Pi_1)\) is said to be a \newnotion{crossed module extension class} of \(\Pi_0\) with \(\Pi_1\) \newnotion{in \(\mathfrak{U}\)} (or a \newnotion{\(\mathfrak{U}\)-crossed module extension class} of \(\Pi_0\) with \(\Pi_1\)).

The following theorem appeared in various guises, see \eigenname{Mac\,Lane}~\cite{maclane:1979:historical_note} and \eigenname{Ratcliffe}~\cite[th.~9.4]{ratcliffe:1980:crossed_extensions}. It has been generalised to crossed complexes by \eigenname{Holt}~\cite[th.~4.5]{holt:1979:an_interpretation_of_the_cohomology_groups_h_n_g_m} and, independently, \eigenname{Huebschmann}~\cite[p.~310]{huebschmann:1980:crossed_n-fold_extensions_of_groups_and_cohomology}. Moreover, there is a version for \(n\)-cat groups given by \eigenname{Loday}~\cite[th.~4.2]{loday:1982:spaces_with_finitely_many_non-trivial_homotopy_groups}.

\begin{theorem*}
There is a bijection between the set of crossed module extension classes \(\ExtensionClasses[\mathfrak{U}]{2}(\Pi_0, \Pi_1)\) and the third cohomology group \(\CohomologyGroup[3](\Pi_0, \Pi_1)\), where \(\mathfrak{U}\) is supposed to be a Grothendieck universe containing an infinite set. 
\end{theorem*}

This theorem can also be shown by arguments due to \eigenname{Eilenberg} and \eigenname{Mac\,Lane}, see~\cite[sec.~7, sec.~9]{eilenberg_maclane:1947:cohomology_theory_in_abstract_groups_II_group_extensions_with_a_non-abelian_kernel} and~\cite[sec.~7]{maclane:1949:cohomology_theory_in_abstract_groups_III_operator_homomorphisms_of_kernels}. A detailed proof following these arguments, using the language of crossed modules, can be found in the manuscript~\cite{thomas:2009:the_third_cohomology_group_classifies_crossed_module_extensions}, where a bijection \(\ExtensionClasses[\mathfrak{U}]{2}(\Pi_0, \Pi_1) \map \CohomologyGroup[3](\Pi_0, \Pi_1), [E]_{\extensionequivalent[\mathfrak{U}]} \mapsto \cocycleofextension{3}_E \CoboundaryGroup[3](\Pi_0, \Pi_1)\) is explicitly constructed. This construction is used throughout section~\ref{sec:crossed_module_extensions_and_standard_2-cocycles}. The inverse bijection \(z^3 \CoboundaryGroup[3](\Pi_0, \Pi_1) \mapsto [\StandardExtension(z^3)]_{\extensionequivalent[\mathfrak{U}]}\) is used in corollary~\ref{cor:cohomology_groups_of_cohomologous_3-cocycles_are_isomorphic}. We give a sketch of these constructions. That is, we indicate how a \(3\)-cohomology class of \(\Pi_0\) with coefficients in \(\Pi_1\) can be associated to a crossed module extension (class) of \(\Pi_0\) with \(\Pi_1\), and conversely, how a crossed module extension can be constructed from a given \(3\)-cohomology class.

Given pointed sets \(X_i\) for \(i \in I\) and \(Y\), where \(I\) is supposed to be an index set, let us call a map \(f\colon \bigcart_{i \in I} X_i \map Y\) \newnotion{componentwise pointed} if \((x_i)_{i \in I} f = *\) for all \((x_i)_{i \in I} \in \bigcart_{i \in I} X_i\) with \(x_i = *\) for some \(i \in I\). So in particular, interpreting groups as pointed sets in the usual way, a \(3\)-cochain \(c^3 \in \CochainComplex[3](\Pi_0, \Pi_1)\) is componentwise pointed if it fulfills \((q, p, 1) c^3 = (q, 1, p) c^3 = (1, q, p) c^3 = 1\) for all \(p, q \in \Pi_0\). The set of componentwise pointed \(3\){\nbd}cochains of \(\Pi_0\) with coefficients in \(\Pi_1\) will be denoted by \(\CochainComplex[3]_{\text{cpt}}(\Pi_0, \Pi_1)\), the set of componentwise pointed \(3\)-cocycles by \(\CocycleGroup[3]_{\text{cpt}}(\Pi_0, \Pi_1) := \CocycleGroup[3](\Pi_0, \Pi_1) \intersection \CochainComplex[3]_{\text{cpt}}(\Pi_0, \Pi_1)\), the set of componentwise pointed \(3\)-coboundaries by \(\CoboundaryGroup[3]_{\text{cpt}}(\Pi_0, \Pi_1) := \CoboundaryGroup[3](\Pi_0, \Pi_1) \intersection \CochainComplex[3]_{\text{cpt}}(\Pi_0, \Pi_1)\) and the set of componentwise pointed \(3\)-cohomology classes by \(\CohomologyGroup[3]_{\text{cpt}}(\Pi_0, \Pi_1) := \CocycleGroup[3]_{\text{cpt}}(\Pi_0, \Pi_1) / \CoboundaryGroup[3]_{\text{cpt}}(\Pi_0, \Pi_1)\). With these notations, we have \(\CohomologyGroup[3](\Pi_0, \Pi_1) \isomorphic \CohomologyGroup[3]_{\text{cpt}}(\Pi_0, \Pi_1)\). Analogously in other dimensions, cf.\ for example~\cite[cor.~(3.7)]{thomas:2009:the_third_cohomology_group_classifies_crossed_module_extensions}.

We suppose given a crossed module extension \(E\) of \(\Pi_0\) with \(\Pi_1\). First, we choose a lift of \(\id_{\Pi_0}\) along the underlying pointed map of \(\canonicalepimorphism\), that is, a pointed map \(Z^1\colon \Pi_0 \map \GroupPart E\) with \(Z^1 \canonicalepimorphism = \id_{\Pi_0}\). We obtain the componentwise pointed map
\[\cocycleofextension{2} = \cocycleofextension{2}_{E, Z^1}\colon \Pi_0 \cart \Pi_0 \map \Image \structuremorphism, (q, p) \mapsto (q Z^1) (p Z^1) ((q p) Z^1)^{- 1}\]
fulfilling the non-abelian \(2\)-cocycle condition
\[(r, q) \cocycleofextension{2} (r q, p) \cocycleofextension{2} = {^{r Z^1}}((q, p) \cocycleofextension{2}) (r, q p) \cocycleofextension{2}\]
for \(p, q, r \in \Pi_0\). Therefore, we will call \(\cocycleofextension{2}\) the \newnotion{non-abelian \(2\)-cocycle of \(E\) with respect to \(Z^1\)}. Next, we choose a componentwise pointed lift of \(\cocycleofextension{2}\) along \(\structuremorphism|^{\Image \structuremorphism}\), that is, a componentwise pointed map \(Z^2\colon \Pi_0 \cart \Pi_0 \map \ModulePart E\) with \(Z^2 (\structuremorphism|^{\Image \structuremorphism}) = \cocycleofextension{2}\). This leads to the map
\begin{align*}
& \cocycleofextension{3} = \cocycleofextension{3}_{E, (Z^2, Z^1)}\colon \Pi_0 \cart \Pi_0 \cart \Pi_0 \map \Pi_1, \\
& \qquad (r, q, p) \mapsto \big( (r, q) Z^2 (r q, p) Z^2 ((r, q p) Z^2)^{- 1} ({^{r Z^1}}((q, p) Z^2))^{- 1} \big) (\canonicalmonomorphism|^{\Image \canonicalmonomorphism})^{- 1},
\end{align*}
which is shown to be a componentwise pointed \(3\)-cocycle of \(\Pi_0\) with coefficients in \(\Pi_1\), that is, an element of \(\CocycleGroup[3]_{\text{cpt}}(\Pi_0, \Pi_1)\). One shows that the cohomology class of \(\cocycleofextension{3}\) is independent from the choices of \(Z^1\), \(Z^2\) and the representative \(E\) in its extension class.

A pair \((Z^2, Z^1)\) of componentwise pointed maps \(Z^1\colon \Pi_0 \map \GroupPart E\) and \(Z^2\colon \Pi_0 \cart \Pi_0 \map \ModulePart E\) such that \(Z^1 \canonicalepimorphism = \id_{\Pi_0}\) and \(Z^2 (\structuremorphism|^{\Image \structuremorphism}) = \cocycleofextension{2}\) is called a \newnotion{lifting system} for \(E\). Moreover, a pair \((s^1, s^0)\) of pointed maps \(s^0\colon \Pi_0 \map \GroupPart E\) and \(s^1\colon \Image \structuremorphism \map \ModulePart E\) such that \(s^0 \canonicalepimorphism = \id_{\Pi_0}\) and \(s^1 (\structuremorphism|^{\Image \structuremorphism}) = \id_{\Image \structuremorphism}\) is said to be a \newnotion{section system} for \(E\). Every section system \((s^1, s^0)\) for \(E\) provides a lifting system \((Z^2, Z^1)\) for \(E\) by setting \(Z^1 := s^0\) and \(Z^2 := \cocycleofextension{2}_{E, s^0} s^1\), called the lifting system \newnotion{coming from} \((s^1, s^0)\). The \(3\)-cocycle \(\cocycleofextension{3} \in \CocycleGroup[3]_{\text{cpt}}(\Pi_0, \Pi_1)\) constructed as indicated above will be called the \newnotion{\(3\)-cocycle of \(E\)} with respect to \((Z^2, Z^1)\). If \((Z^2, Z^1)\) comes from a section system \((s^1, s^0)\), we also write \(\cocycleofextension{3} = \cocycleofextension{3}_{E, (s^1, s^0)} := \cocycleofextension{3}_{E, (Z^2, Z^1)}\) and call this the \newnotion{\(3\)-cocycle of \(E\)} with respect to \((s^1, s^0)\). Finally, we call \(\cohomologyclassofextension(E) := \cocycleofextension{3} \CoboundaryGroup[3]_{\text{cpt}}(\Pi_0, \Pi_1)\) the \newnotion{cohomology class associated to \(E\)}.

We note two more facts: First, for every componentwise pointed \(3\)-cocycle \(z^3 \in \CocycleGroup[3]_{\text{cpt}}(\Pi_0, \Pi_1)\) with \(\cohomologyclassofextension(E) = z^3 \CoboundaryGroup[3]_{\text{cpt}}(\Pi_0, \Pi_1)\) there exists a lifting system \((Z^2, Z^1)\) such that \(\cocycleofextension{3}_{E, (Z^2, Z^1)} = z^3\), cf.\ for example~\cite[prop.~(5.19)]{thomas:2009:the_third_cohomology_group_classifies_crossed_module_extensions}. Second, given crossed module extensions \(E\) and \(\tilde E\) and an extension equivalence \(\varphi\colon E \map \tilde E\), there exist a section system \((s^1, s^0)\) for \(E\) and a section system \((\tilde s^1, \tilde s^0)\) for \(\tilde E\) such that \(\tilde s^0 = s^0 (\GroupPart \varphi)\) and \(s^1 (\ModulePart \varphi) = (\GroupPart \varphi)|_{\Image \structuremorphism[E]}^{\Image \structuremorphism[\tilde E]} \tilde s^1\), and moreover such that \(\cocycleofextension{3}_{E, (s^1, s^0)} = \cocycleofextension{3}_{\tilde E, (\tilde s^1, \tilde s^0)}\), cf.\ for example~\cite[prop.~(5.16)(b), prop.~(5.14)(c)]{thomas:2009:the_third_cohomology_group_classifies_crossed_module_extensions}.

Conversely, for a componentwise pointed \(3\)-cocycle \(z^3 \in \CocycleGroup[3]_{\text{cpt}}(\Pi_0, \Pi_1)\), the standard extension of \(\Pi_0\) with \(\Pi_1\) with respect to \(z^3\) is constructed as follows.

We let \(F\) be a free group on the underlying pointed set of \(\Pi_0\) with basis \(s^0 = Z^1\colon \Pi_0 \map F\), that is, \(F\) is a free group on the set \(\Pi_0 \setminus \{1\}\) and \(s^0\) maps \(x \in \Pi_0 \setminus \{1\}\) to the corresponding generator \(x s^0 \in F\), and \(1 s^0 = 1\). We let \(\pi\colon F \map \Pi_0\) be induced by \(\id_{\Pi_0}\colon \Pi_0 \map \Pi_0\). The basis \(s^0\) is a section of the underlying pointed map of \(\pi\). We let \(z^2\colon \Pi_0 \cart \Pi_0 \map \Kernel \pi, (q, p) \mapsto (q s^0) (p s^0) ((q p) s^0)^{- 1}\). We let \(\iota\colon \Pi_1 \map \Pi_1 \directprod \Kernel \pi, m \mapsto (m, 1)\) and \(\mu\colon \Pi_0 \directprod \Kernel \pi \map F, (m, f) \mapsto f\). We let \(s^1\colon \Kernel \pi \map \Pi_0 \directprod \Kernel \pi, f \mapsto (1, f)\) and we let \(Z^2\colon \Pi_0 \cart \Pi_0 \map \Pi_1 \directprod \Kernel \pi\) be given by \(Z^2 := z^2 s^1\). The direct product \(\Pi_1 \directprod \Kernel \pi\) is generated by \(\Image \iota \union \Image Z^2\) and carries the structure of an \(F\)-module uniquely determined on this set of generators by \({^{r Z^1}}(k \iota) := ({^r}k) \iota\) for \(k \in \Pi_1\), \(r \in \Pi_0\), and \({^{r Z^1}}((q, p) Z^2) := ((r, q, p) z^3 \iota)^{- 1} ((r, q) Z^2) ((r q, p) Z^2) ((r, q p) Z^2)^{- 1}\) for \(p, q, r \in G\).

These data define the \newnotion{standard extension} \(\StandardExtension(z^3)\) and the \newnotion{standard section system} \((\standardsectionsystem_{z^3}^1, \standardsectionsystem_{z^3}^0)\) for \(\StandardExtension(z^3)\): The group part of \(\StandardExtension(z^3)\) is given by \(\GroupPart \StandardExtension(z^3) := F\), the module part is given by \(\ModulePart \StandardExtension(z^3) := M \directprod \Kernel \pi\) and the structure morphism is given by \(\structuremorphism[\StandardExtension(z^3)] := \mu\). We have the canonical monomorphism \(\canonicalmonomorphism[\StandardExtension(z^3)] := \iota\) and the canonical epimorphism \(\canonicalepimorphism[\StandardExtension(z^3)] := \pi\). The section system \((\standardsectionsystem_{z^3}^1, \standardsectionsystem_{z^3}^0)\) is defined by \(\standardsectionsystem_{z^3}^0 := s^0\) and \(\standardsectionsystem_{z^3}^1 := s^1\).

By construction, the \(3\)-cocycle of \(\StandardExtension(z^3)\) with respect to the section system \((\standardsectionsystem_{z^3}^1, \standardsectionsystem_{z^3}^0)\) is \(z^3\). In particular, \(\cohomologyclassofextension(\StandardExtension(z^3)) = z^3 \CoboundaryGroup[3]_{\text{cpt}}(G, M)\).

\section{Low dimensional cohomology of a simplicial group} \label{sec:low_dimensional_cohomology_of_a_simplicial_group}

In this section, we will show that the zeroth cohomology group of a simplicial group depends only on the coefficient module, that the first cohomology group depends only on the \(0\)-truncation and that the second cohomology group depends only on the \(1\)-truncation.

Our results shall be achieved by means of calculations with analysed cocycles and coboundaries in low dimensions. Therefore, we restate their definitions explicitly.

\begin{workingbase} \label{wb:low_dimensional_analysed_cocycles} \
\begin{enumerate}
\item \label{wb:low_dimensional_analysed_cocycles:simplicial_group} We suppose given a simplicial group \(G\) and an abelian \(\HomotopyGroup[0](G)\)-module \(M\). The analysed cochain complex \(\CochainComplex_{\text{an}}(G, M)\) starts with the following entries. (\footnote{To simplify notation, we identify \((\MooreComplex[1] G \cart \MooreComplex[0] G) \cart (\MooreComplex[0] G)\) with \(\MooreComplex[1] G \cart \MooreComplex[0] G \cart \MooreComplex[0] G\), etc.})
\begin{align*}
\CochainComplex[0]_{\text{an}}(G, M) & = \Map(\{1\}, M), \\
\CochainComplex[1]_{\text{an}}(G, M) & = \Map(\MooreComplex[0] G, M), \\
\CochainComplex[2]_{\text{an}}(G, M) & = \Map(\MooreComplex[1] G \cart \MooreComplex[0] G \cart \MooreComplex[0] G, M), \\
\CochainComplex[3]_{\text{an}}(G, M) & = \Map(\MooreComplex[2] G \cart \MooreComplex[1] G \cart \MooreComplex[1] G \cart \MooreComplex[0] G \cart \MooreComplex[1] G \cart \MooreComplex[0] G \cart \MooreComplex[0] G, M).
\end{align*}
The differentials are given by
\[(g_0) (c \differential) = 1 c - g_0 \BoundaryGroup[0] \MooreComplex G \act 1 c\]
for \(g_0 \in \MooreComplex[0] G\), \(c \in \CochainComplex[0]_{\text{an}}(G, M)\), by
\[(g_1, h_0, g_0) (c \differential) = (g_1 h_0) c - (h_0 g_0) c + h_0 \BoundaryGroup[0] \MooreComplex G \act (g_0) c\]
for \(g_0, h_0 \in \MooreComplex[0] G\), \(g_1 \in \MooreComplex[1] G\), \(c \in \CochainComplex[1]_{\text{an}}(G, M)\), and by
\begin{align*}
& (g_2, k_1, h_1, k_0, g_1, h_0, g_0) (c \differential) \\
& = ((g_2 \differential) k_1, (h_1 \differential) k_0, (g_1 \differential) h_0) c - (k_1 h_1, k_0, h_0 g_0) c + (h_1 \, {^{k_0 \degeneracy{0}}}g_1, k_0 h_0, g_0) c - k_0 \BoundaryGroup[0] \MooreComplex G \act (g_1, h_0, g_0) c
\end{align*}
for \(g_0, h_0, k_0 \in \MooreComplex[0] G\), \(g_1, h_1, k_1 \in \MooreComplex[1] G\), \(g_2 \in \MooreComplex[2] G\), \(c \in \CochainComplex[2]_{\text{an}}(G, M)\).
\item \label{wb:low_dimensional_analysed_cocycles:crossed_module} We suppose given a crossed module \(V\) and an abelian \(\HomotopyGroup[0](V)\)-module \(M\). The cochain complex \(\CochainComplex(V, M)\) starts with the following entries.
\begin{align*}
\CochainComplex[0](V, M) & = \Map(\{1\}, M), \\
\CochainComplex[1](V, M) & = \Map(\GroupPart V, M), \\
\CochainComplex[2](V, M) & = \Map(\ModulePart V \cart \GroupPart V \cart \GroupPart V, M), \\
\CochainComplex[3](V, M) & = \Map(\ModulePart V \cart \ModulePart V \cart \GroupPart V \cart \ModulePart V \cart \GroupPart V \cart \GroupPart V, M).
\end{align*}
The differentials are given by
\[(g) (c \differential) = 1 c - g (\Image \structuremorphism) \act 1 c\]
for \(g \in \GroupPart V\), \(c \in \CochainComplex[0](V, M)\), by
\[(m, h, g) (c \differential) = (m h) c - (h g) c + h (\Image \structuremorphism) \act (g) c\]
for \(g, h \in \GroupPart V\), \(m \in \ModulePart V\), \(c \in \CochainComplex[1](V, M)\), and by
\[(p, n, k, m, h, g) (c \differential) = (p, n k, m h) c - (p n, k, h g) c + (n \, {^k}m, k h, g) c - k (\Image \structuremorphism) \act (m, h, g) c\]
for \(g, h, k \in \GroupPart V\), \(m, n, p \in \ModulePart V\), \(c \in \CochainComplex[2](V, M)\).
\end{enumerate}
\end{workingbase}
\begin{proof} \
\begin{enumerate}
\item We show how the differential \(\differential\colon \CochainComplex[2]_{\text{an}}(G, M) \map \CochainComplex[3]_{\text{an}}(G, M)\) of the analysed cochain complex is computed using transport of structure, the easier lower dimensional cases are left to the reader.

The corresponding entries of the cochain complex are \(\CochainComplex[2](G, M) = \Map(G_1 \cart G_0, M)\) and \(\CochainComplex[3](G, M) = \Map(G_2 \cart G_1 \cart G_0, M)\). Now the semidirect product decompositions of \(G_0\), \(G_1\) and \(G_2\) are given by the isomorphisms
\begin{align*}
& \varphi_0\colon G_0 \map \MooreComplex[0] G, g_0 \mapsto g_0, \\
& \varphi_0^{- 1}\colon \MooreComplex[0] G \map G_0, g_0 \mapsto g_0, \\
& \varphi_1\colon G_1 \map \MooreComplex[1] G \semidirect \MooreComplex[0] G, g_1 \mapsto (g_1 (g_1 \face{1} \degeneracy{0})^{- 1}, g_1 \face{1}), \\
& \varphi_1^{- 1}\colon \MooreComplex[1] G \semidirect \MooreComplex[0] G \map G_1, (g_1, g_0) \mapsto g_1 (g_0 \degeneracy{0}), \\
& \varphi_2\colon G_2 \map (\MooreComplex[2] G \semidirect \MooreComplex[1] G) \semidirect (\MooreComplex[1] G \semidirect \MooreComplex[0] G), \\
& \qquad g_2 \mapsto ((g_2 (g_2 \face{2} \degeneracy{1})^{- 1} (g_2 \face{2} \degeneracy{0}) (g_2 \face{1} \degeneracy{0})^{- 1}, (g_2 \face{1}) (g_2 \face{2})^{- 1}), ((g_2 \face{2}) (g_2 \face{2} \face{1} \degeneracy{0})^{- 1}, g_2 \face{2} \face{1})), \\
& \varphi_2^{- 1}\colon (\MooreComplex[2] G \semidirect \MooreComplex[1] G) \semidirect (\MooreComplex[1] G \semidirect \MooreComplex[0] G) \map G_2, ((g_2, h_1), (g_1, g_0)) \mapsto g_2 (h_1 \degeneracy{0}) (g_1 \degeneracy{1}) (g_0 \degeneracy{0} \degeneracy{1}).
\end{align*}
Moreover, the image \(c' \differential \in \CochainComplex[3](G, M)\) of a \(2\)-cochain \(c' \in \CochainComplex[2](G, M)\) is defined by
\[(g_2, g_1, g_0) (c' \differential) = (g_2 \face{0}, g_1 \face{0}) c' - (g_2 \face{1}, (g_1 \face{1}) g_0) c' + ((g_2 \face{2}) g_1, g_0) c' - (g_2 \face{2} \face{1} \BoundaryGroup[0] \MooreComplex G) (g_1, g_0) c'.\]
Hence we obtain
\begin{align*}
\CochainComplex[2]_{\text{an}}(G, M) & = \Map((\MooreComplex[1] G \cart \MooreComplex[0] G) \cart \MooreComplex[0] G, M), \\
\CochainComplex[3]_{\text{an}}(G, M) & = \Map((\MooreComplex[2] G \cart \MooreComplex[1] G \cart \MooreComplex[1] G \cart \MooreComplex[0] G) \cart (\MooreComplex[1] G \cart \MooreComplex[0] G) \cart \MooreComplex[0] G, M),
\end{align*}
and, using the isomorphisms \(\varphi_i\) for \(i \in \{0, 1, 2\}\), the image \(c \differential \in \CochainComplex[3]_{\text{an}}(G, M)\) of an analysed \(2\)-cochain \(c \in \CochainComplex[2]_{\text{an}}(G, M)\) is given by
\[c \differential = (\varphi_2^{- 1} \cart \varphi_1^{- 1} \cart \varphi_0^{- 1}) (((\varphi_1 \cart \varphi_0) c) \differential),\]
that is, we have
\begin{align*}
& ((g_2, k_1, h_1, k_0), (g_1, h_0), g_0) (c \differential) = ((g_2, k_1, h_1, k_0) \varphi_2^{- 1}, (g_1, h_0) \varphi_1^{- 1}, g_0 \varphi_0^{- 1}) (((\varphi_1 \cart \varphi_0) c) \differential) \\
& = (g_2 (k_1 \degeneracy{0}) (h_1 \degeneracy{1}) (k_0 \degeneracy{0} \degeneracy{1}), g_1 (h_0 \degeneracy{0}), g_0) (((\varphi_1 \cart \varphi_0) c) \differential) \\
& = ((g_2 (k_1 \degeneracy{0}) (h_1 \degeneracy{1}) (k_0 \degeneracy{0} \degeneracy{1})) \face{0}, (g_1 (h_0 \degeneracy{0})) \face{0}) ((\varphi_1 \cart \varphi_0) c) \\
& \qquad - ((g_2 (k_1 \degeneracy{0}) (h_1 \degeneracy{1}) (k_0 \degeneracy{0} \degeneracy{1})) \face{1}, ((g_1 (h_0 \degeneracy{0})) \face{1}) g_0) ((\varphi_1 \cart \varphi_0) c) \\
& \qquad + (((g_2 (k_1 \degeneracy{0}) (h_1 \degeneracy{1}) (k_0 \degeneracy{0} \degeneracy{1})) \face{2}) (g_1 (h_0 \degeneracy{0})), g_0) ((\varphi_1 \cart \varphi_0) c) \\
& \qquad - (g_2 (k_1 \degeneracy{0}) (h_1 \degeneracy{1}) (k_0 \degeneracy{0} \degeneracy{1})) \face{2} \face{1} \BoundaryGroup[0] \MooreComplex G \act (g_1 (h_0 \degeneracy{0}), g_0) ((\varphi_1 \cart \varphi_0) c) \\
& = (((g_2 \differential) k_1 (h_1 \differential \degeneracy{0}) (k_0 \degeneracy{0})) \varphi_1, ((g_1 \differential) h_0) \varphi_0) c - ((k_1 h_1 (k_0 \degeneracy{0})) \varphi_1, (h_0 g_0) \varphi_0) c \\
& \qquad + ((h_1 (k_0 \degeneracy{0}) g_1 (h_0 \degeneracy{0})) \varphi_1, g_0 \varphi_0) c - k_0 \BoundaryGroup[0] \MooreComplex G \act ((g_1 (h_0 \degeneracy{0})) \varphi_1, g_0 \varphi_0) c \\
& = (((g_2 \differential) k_1 (h_1 \differential \degeneracy{0}) (k_0 \degeneracy{0}) (k_0 \degeneracy{0})^{- 1} (h_1 \differential \degeneracy{0})^{- 1}, (h_1 \differential) k_0), (g_1 \differential) h_0) c - ((k_1 h_1 (k_0 \degeneracy{0}) (k_0 \degeneracy{0})^{- 1}, k_0), h_0 g_0) c \\
& \qquad + ((h_1 (k_0 \degeneracy{0}) g_1 (h_0 \degeneracy{0}) (h_0 \degeneracy{0})^{- 1} (k_0 \degeneracy{0})^{- 1}, k_0 h_0), g_0) c - k_0 \BoundaryGroup[0] \MooreComplex G \act ((g_1 (h_0 \degeneracy{0}) (h_0 \degeneracy{0})^{- 1}, h_0), g_0) c \\
& = (((g_2 \differential) k_1, (h_1 \differential) k_0), g_1 h_0) c - (((k_1 \differential) h_1, k_0), h_0 g_0) c + ((h_1 \, {^{k_0}}g_1, k_0 h_0), g_0) c \\
& \qquad - k_0 \BoundaryGroup[0] \MooreComplex G \act ((g_1, h_0), g_0) c
\end{align*}
for \(g_0, h_0, k_0 \in \MooreComplex[0] G\), \(g_1, h_1, k_1 \in \MooreComplex[1] G\), \(g_2 \in \MooreComplex[2] G\).
\item This follows from~\ref{wb:low_dimensional_analysed_cocycles:simplicial_group} and the definition of crossed module cohomology via \(\Coskeleton_1\), cf.\ section~\ref{ssec:cohomology_of_groups_and_cohomology_of_crossed_modules}. \qedhere
\end{enumerate}
\end{proof}

We immediately obtain the following result about the zeroth cohomology group, which states that it only depends on the module of coefficients (and therefore implicitly on the zeroth homotopy group by our choice of coefficients).

\begin{proposition} \label{prop:zeroth_cohomology_group_of_a_simplicial_group}
Given a simplicial group \(G\) and an abelian \(\HomotopyGroup[0](G)\)-module \(M\), we have
\[\CohomologyGroup[0](G, M) \isomorphic \CohomologyGroup[0](\HomotopyGroup[0](G), M) \isomorphic \{m \in M \mid \text{\(p m = m\) for all \(p \in \HomotopyGroup[0](G)\)}\}.\]
\end{proposition}

\begin{corollary} \label{cor:zeroth_cohomology_group_of_a_crossed_module}
Given a crossed module \(V\) and an abelian \(\HomotopyGroup[0](V)\)-module \(M\), we have
\[\CohomologyGroup[0](V, M) \isomorphic \CohomologyGroup[0](\HomotopyGroup[0](V), M) \isomorphic \{m \in M \mid \text{\(p m = m\) for all \(p \in \HomotopyGroup[0](V)\)}\}.\]
\end{corollary}

We suppose given a simplicial group \(G\), an abelian group \(A\) and \(n \in \{0, 1\}\). In propositions~\ref{prop:first_cohomology_group_of_a_simplicial_group_and_its_0-truncation} and~\ref{prop:second_cohomology_group_of_a_simplicial_group_and_its_1-truncation}, we will show that \(\CohomologyGroup[n + 1](G, A) \isomorphic \CohomologyGroup[n + 1](\Truncation^n G, A)\). Using homotopy theory of topological spaces, this can be seen as follows.

We consider the unit component \(\unit_G\colon G \map \Coskeleton_n \Truncation^n G\) of the adjunction \(\Truncation^n \leftadjoint \Coskeleton_n\) and claim that \(\HomotopyGroup[k] \unit_G\) is an isomorphism for \(k \in [0, n]\), cf.\ section~\ref{ssec:truncation_and_coskeleton}. If \(n = 0\), one reads off that \(\Truncation^0 \unit_G\) is an isomorphism and hence \(\HomotopyGroup[0] \unit_G\) is an isomorphism since \(\HomotopyGroup[0] = \Truncation^0\). If \(n = 1\), one reads off that \(\GroupPart(\Truncation^1 \unit_G)\) and \(\ModulePart(\Truncation^1 \unit_G)\) are isomorphisms, hence \(\Truncation^1 \unit_G\) is an isomorphism and thus \(\HomotopyGroup[k] \unit_G = \HomotopyGroup[k](\Truncation^1 \unit_G)\) are isomorphisms for \(k \in [0, 1]\), cf.~\cite[prop.~(6.25)]{thomas:2007:co_homology_of_crossed_modules}.

The canonical simplicial map \(\KanResolvingSimplicialSet G \map \KanClassifyingSimplicialSet G\) is a Kan fibration with fiber \(G\), and \(\KanResolvingSimplicialSet G\) is contractible, see~\cite[ch.~V, lem.~4.1, lem.~4.6]{goerss_jardine:1999:simplicial_homotopy_theory}. Analogously for \(\Coskeleton_n \Truncation^n G\), so the induced long exact homotopy sequence~\cite[ch.~VII, 4.1, 4.2, 5.3]{lamotke:1968:semisimpliziale_algebraische_topologie} shows that \(\HomotopyGroup[k](\KanClassifyingSimplicialSet \unit_G)\) are isomorphisms for \(k \in [0, n + 1]\). It follows that \(\HomotopyGroup[k](|{\KanClassifyingSimplicialSet \unit_G}|)\) are isomorphisms for \(k \in [0, n + 1]\), see~\cite[ch.~I, prop.~11.1]{goerss_jardine:1999:simplicial_homotopy_theory} and~\cite[ch.~VII, 10.9]{lamotke:1968:semisimpliziale_algebraische_topologie}. The Whitehead theorem~\cite[ch.~VII, th.~11.2~I(b)]{bredon:1993:topology_and_geometry} provides isomorphisms \(\HomologyGroup[k](|{\KanClassifyingSimplicialSet \unit_G}|)\) for \(k \in [0, n + 1]\). The universal coefficient theorem~\cite[ch.~V, cor.~7.2]{bredon:1993:topology_and_geometry} yields isomorphisms \(\CohomologyGroup[k](|{\KanClassifyingSimplicialSet \unit_G}|, A)\) for \(k \in [0, n + 1]\). Finally, \(\CohomologyGroup[k]({\KanClassifyingSimplicialSet \unit_G}, A)\) are isomorphisms for \(k \in [0, n + 1]\) by~\cite[th.~6.3]{kan:1957:on_c_s_s_complexes}. In particular, one obtains \(\CohomologyGroup[n + 1](G, A) \isomorphic \CohomologyGroup[n + 1](\Coskeleton_n \Truncation^n G, A) = \CohomologyGroup[n + 1](\Truncation^n G, A)\), as desired.

However, we will not make use of these topological arguments. Following the overall intention of this article, we will give direct algebraic proofs of these results. Moreover, we will use proposition~\ref{prop:characterising_properties_of_analysed_2-cocycles}\ref{prop:characterising_properties_of_analysed_2-cocycles:crossed_modules} several times in section~\ref{sec:crossed_module_extensions_and_standard_2-cocycles}, in particular in the proofs of proposition~\ref{prop:module_part_and_group_part_of_standardisation_of_pointed_2-cocycles_and_coboundaries} and proposition~\ref{prop:characterising_properties_of_standard_2-cocycles_and_coboundaries_of_crossed_module_extensions}.

\begin{proposition} \label{prop:analysed_1-cocycle_group_of_simplicial_groups_as_a_kernel}
We suppose given a simplicial group \(G\) and an abelian \(\HomotopyGroup[0](G)\)-module \(M\). The first analysed cocycle group \(\CocycleGroup[1]_{\text{an}}(G, M)\) is the kernel of
\[\inc^{\CocycleGroup[1](\MooreComplex[0] G, M)} \CochainComplex[1](\differential^{\MooreComplex G}, M)\colon \CocycleGroup[1](\MooreComplex[0] G, M) \map \CochainComplex[1](\MooreComplex[1] G, M),\]
that is, we have
\[\CocycleGroup[1]_{\text{an}}(G, M) = \{z_0 \in \CocycleGroup[1](\MooreComplex[0] G, M) \mid z_0|_{\BoundaryGroup[0] \MooreComplex G} = 0\}.\]
\end{proposition}
\begin{proof}
For every element \(z \in \CocycleGroup[1]_{\text{an}}(G, M)\), we have
\[0 = (1, h_0, g_0) (z \differential^{\CochainComplex_{\text{an}}(G, M)}) = (h_0) z - (h_0 g_0) z + h_0 \BoundaryGroup[0] \MooreComplex G \act (g_0) z = (h_0, g_0) (z \differential^{\CochainComplex(\MooreComplex[0] G, M)})\]
for all \(g_0, h_0 \in \MooreComplex[0] G\) as well as
\[0 = (g_1, 1, 1) (z \differential^{\CochainComplex_{\text{an}}(G, M)}) = (g_1 \differential) z\]
for all \(g_1 \in \MooreComplex[1] G\), that is, \(\CocycleGroup[1]_{\text{an}}(G, M) \subseteq \CocycleGroup[1](\MooreComplex[0] G, M)\) and \(z|_{\BoundaryGroup[0] \MooreComplex G} = 0\). Conversely, given a \(1\)-cocycle \(z_0 \in \CocycleGroup[1](\MooreComplex[0] G, M)\) with \(z_0|_{\BoundaryGroup[0] \MooreComplex G} = 0\), it follows that
\begin{align*}
(g_1, h_0, g_0) (z_0 \differential^{\CochainComplex_{\text{an}}(G, M)}) & = ((g_1 \differential) h_0) z_0 - (h_0 g_0) z_0 + h_0 \BoundaryGroup[0] \MooreComplex G \act (g_0) z_0 \\
& = (g_1) z_0 + (g_1 \differential) \BoundaryGroup[0] \MooreComplex G \act (h_0) z_0 - (h_0 g_0) z_0 + h_0 \BoundaryGroup[0] \MooreComplex G \act (g_0) z_0 \\
& = (h_0) z_0 - (h_0 g_0) z_0 + h_0 \BoundaryGroup[0] \MooreComplex G \act (g_0) z_0 = (h_0, g_0) (z_0 \differential^{\CochainComplex(\MooreComplex[0] G, M)}) = 0
\end{align*}
for \(g_1 \in \MooreComplex[1] G\), \(g_0, h_0 \in \MooreComplex[0] G\), that is, \(z_0 \in \CocycleGroup[1]_{\text{an}}(G, M)\). Altogether, we have
\[\CocycleGroup[1]_{\text{an}}(G, M) = \{z_0 \in \CocycleGroup[1](\MooreComplex[0] G, M) \mid z_0|_{\BoundaryGroup[0] \MooreComplex G} = 0\}. \qedhere\]
\end{proof}

Recall that \(\CochainComplex(\Truncation^0 G, M) = \CochainComplex_{\text{an}}(\Coskeleton_0 \Truncation^0 G, M)\) for every simplicial group \(G\).

\begin{proposition} \label{prop:first_cohomology_group_of_a_simplicial_group_and_its_0-truncation}
Given a simplicial group \(G\) and an abelian \(\HomotopyGroup[0](G)\)-module \(M\), the unit component \(\unit_G\colon G \map \Coskeleton_0 \Truncation^0 G\) of the adjunction \(\Truncation^0 \leftadjoint \Coskeleton_0\) induces an isomorphism
\[\CocycleGroup[1]_{\text{an}}(\unit_G, M)\colon \CocycleGroup[1](\Truncation^0 G, M) \map \CocycleGroup[1]_{\text{an}}(G, M),\]
which in turn induces isomorphisms \(\CoboundaryGroup[1]_{\text{an}}(\unit_G, M)\) and \(\CohomologyGroup[1]_{\text{an}}(\unit_G, M)\). In particular, we have
\[\CohomologyGroup[1](G, M) \isomorphic \CohomologyGroup[1](\Truncation^0 G, M).\]
\end{proposition}
\begin{proof}
We let \(\canonicalepimorphism\colon \MooreComplex[0] G \map \MooreComplex[0] G / \BoundaryGroup[0] \MooreComplex G = \Truncation^0 G\) denote the canonical epimorphism, cf.\ section~\ref{ssec:truncation_and_coskeleton}. The induced group homomorphism \(\CocycleGroup[1]_{\text{an}}(\unit_G, M)\) is given by \((g_0) (z' \CocycleGroup[1]_{\text{an}}(\unit_G, M)) = (g_0 \canonicalepimorphism) z'\) for \(g_0 \in \MooreComplex[0] G\), \(z' \in \CocycleGroup[1](\Truncation^0 G, M)\). Thus we have \(z' \CocycleGroup[1]_{\text{an}}(\unit_G, M) = 0\) if and only if already \(z' = 0\), that is, \(\CocycleGroup[1]_{\text{an}}(\unit_G, M)\) is injective.

To show surjectivity, we suppose given an analysed \(1\)-cochain \(z \in \CocycleGroup[1]_{\text{an}}(G, M)\). We choose a section of the underlying pointed map of \(\canonicalepimorphism\), that is, a pointed map \(s\colon \Truncation^0 G \map \MooreComplex[0] G\) with \(s \canonicalepimorphism = \id_{\Truncation^0 G}\). Then \((q s) (p s) ((q p) s)^{- 1} \in \Kernel \canonicalepimorphism = \BoundaryGroup[0] \MooreComplex G\) and therefore, by proposition~\ref{prop:analysed_1-cocycle_group_of_simplicial_groups_as_a_kernel},
\begin{align*}
((q s) (p s)) z & = ((q s) (p s) ((q p) s)^{- 1} ((q p) s)) z = ((q s) (p s) ((q p) s)^{- 1}) z + ((q s) (p s) ((q p) s)^{- 1}) \BoundaryGroup[0] \MooreComplex G \act ((q p) s) z \\
& = ((q p) s) z
\end{align*}
for all \(p, q \in \Truncation^0 G\). Now the pointed map \(z'\colon \Truncation^0 G \map M\) defined by \((p) z' := (p s) z\) for \(p \in \Truncation^0 G\) is a \(1\)-cocycle in \(\CocycleGroup[1](\Truncation^0 G, M)\)
since
\begin{align*}
(q, p) (z' \differential^{\CochainComplex(\Truncation^0 G, M)}) & = (q) z' - (q p) z' + q \act (p) z' = (q s) z - ((q p) s) z + q s \canonicalepimorphism \act (p s) z \\
& = (q s) z - ((q s) (p s)) z + (q s) \BoundaryGroup[0] \MooreComplex G \act (p s) z = (1, q s, p s) (z \differential^{\CochainComplex_{\text{an}}(G, M)}) = 0
\end{align*}
for all \(p, q \in \Truncation^0 G\). Further, \(g_0 (g_0 \canonicalepimorphism s)^{- 1} \in \Kernel \canonicalepimorphism = \BoundaryGroup[0] \MooreComplex G\) implies, using proposition~\ref{prop:analysed_1-cocycle_group_of_simplicial_groups_as_a_kernel},
\begin{align*}
0 & = (g_0 (g_0 \canonicalepimorphism s)^{- 1}) z = (g_0) z + g_0 \BoundaryGroup[0] \MooreComplex G \act ((g_0 \canonicalepimorphism s)^{- 1}) z = (g_0) z + (g_0 \canonicalepimorphism s) \BoundaryGroup[0] \MooreComplex G \act ((g_0 \canonicalepimorphism s)^{- 1}) z \\
& = (g_0) z - (g_0 \canonicalepimorphism s) z + ((g_0 \canonicalepimorphism s) (g_0 \canonicalepimorphism s)^{- 1}) z= (g_0) z - (g_0) (z' \CocycleGroup[1]_{\text{an}}(\unit_G, M))
\end{align*}
and therefore \((g_0) (z' \CocycleGroup[1]_{\text{an}}(\unit_G, M)) = (g_0) z\) for all \(g_0 \in \MooreComplex[0] G\), that is, \(z' \CocycleGroup[1]_{\text{an}}(\unit_G, M) = z\). Thus \(\CocycleGroup[1]_{\text{an}}(\unit_G, M)\) is surjective. Altogether, \(\CocycleGroup[1]_{\text{an}}(\unit_G, M)\) is an isomorphism of abelian groups.

Now the injectivity of \(\CocycleGroup[1]_{\text{an}}(\unit_G, M)\) implies the injectivity of the restriction \(\CoboundaryGroup[1]_{\text{an}}(\unit_G, M)\). To show that this is also an isomorphism, it remains to show that for every analysed \(1\)-coboundary \(b \in \CoboundaryGroup[1]_{\text{an}}(G, M)\), the \(1\)-cocycle \(b' \in \CocycleGroup[1](\Truncation^0 G, M)\) given by \((p) b' := (p s) b\) for \(p \in \Truncation^0 G\) is in fact a \(1\)-coboundary, that is, an element in \(\CoboundaryGroup[1](\Truncation^0 G, M)\). Indeed, given \(b \in \CoboundaryGroup[1]_{\text{an}}(G, M)\) and an analysed \(0\)-cochain \(c \in \CochainComplex[0]_{\text{an}}(G, M)\) with \(b = c \differential^{\CochainComplex_{\text{an}}(G, M)}\), it follows that
\[(p) b' = (p s) b = 1 c - (p s) \BoundaryGroup[0] \MooreComplex G \act 1 c = 1 c - p \act 1 c = (p) (c \differential^{\CochainComplex(\Truncation^0 G, M)})\]
for all \(p \in \Truncation^0 G\) and hence \(b' = c \differential^{\CochainComplex(\Truncation^0 G, M)} \in \CoboundaryGroup[1](\Truncation^0 G, M)\).

Thus we have shown that \(\CocycleGroup[1]_{\text{an}}(\unit_G, M)\) and \(\CoboundaryGroup[1]_{\text{an}}(\unit_G, M)\) are isomorphisms, and hence \(\CohomologyGroup[1]_{\text{an}}(\unit_G, M)\) is also an isomorphism. In particular, we have
\[\CohomologyGroup[1](G, M) \isomorphic \CohomologyGroup[1]_{\text{an}}(G, M) \isomorphic \CohomologyGroup[1](\Truncation^0 G, M). \qedhere\]
\end{proof}

\begin{corollary} \label{cor:first_cohomology_group_of_a_simplicial_group}
Given a simplicial group \(G\) and an abelian \(\HomotopyGroup[0](G)\)-module \(M\), we have
\[\CohomologyGroup[1](G, M) \isomorphic \CohomologyGroup[1](\HomotopyGroup[0](G), M).\]
\end{corollary}

\begin{corollary} \label{cor:first_cohomology_group_of_a_crossed_module}
Given a crossed module \(V\) and an abelian \(\HomotopyGroup[0](V)\)-module \(M\), we have
\[\CohomologyGroup[1](V, M) \isomorphic \CohomologyGroup[1](\HomotopyGroup[0](V), M).\]
\end{corollary}

We recall a simple fact of \(2\)-cocycles of (ordinary) groups:

\begin{remark} \label{rem:projections_of_2-cocycles_of_groups}
We let \(G\) be a group and \(M\) be an abelian \(G\)-module. For every \(2\)-cocycle \(z \in \CocycleGroup[2](G, M)\), we have \((g, 1) z = g \act (1, 1) z\) and \((1, g) z = (1, 1) z\) for all \(g \in G\).
\end{remark}
\begin{proof}
Given a \(2\)-cocycle \(z \in \CocycleGroup[2](G, M)\), we have
\[0 = (g, 1, 1) (z \differential) = (g, 1) z - (g, 1) z + (g, 1) z - g \act (1, 1) z = (g, 1) z - g \act (1, 1) z,\]
that is, \((g, 1) z = g \act (1, 1) z\), and
\[0 = (1, 1, g) (z \differential) = (1, 1) z - (1, g) z + (1, g) z - (1, g) z = (1, 1) z - (1, g) z,\]
that is, \((1, g) z = (1, 1) z\) for all \(g \in G\).
\end{proof}

\begin{corollary} \label{cor:2-cocycle_of_groups_are_compentwise_pointed_if_and_only_if_they_are_pointed}
We let \(G\) be a group and \(M\) be an abelian \(G\)-module. A \(2\)-cocycle \(z \in \CocycleGroup[2](G, M)\) is componentwise pointed if and only if it is pointed.
\end{corollary}

To simplify our calculations, we give a bit more convenient description of the analysed \(2\)-cocycles.

\begin{definition}[Moore decomposition of analysed \(2\)-cochains] \label{def:moore_decomposition_of_analysed_2-cochains} \
\begin{enumerate}
\item \label{def:moore_decomposition_of_analysed_2-cochains:simplicial_groups} We let \(G\) be a simplicial group and \(M\) be an abelian \(\HomotopyGroup[0](G)\)-module. Given an analysed \(2\)-cochain \(c \in \CochainComplex[2]_{\text{an}}(G, M)\), the \(1\)-cochain \(c_{\MooreComplex[1]} \in \CochainComplex[1](\MooreComplex[1] G, M)\) defined by \((g_1) c_{\MooreComplex[1]} := (g_1, 1, 1) c\) for \(g_1 \in \MooreComplex[1] G\) is called the \newnotion{\(\MooreComplex[1]\)-part} of \(c\), and the \(2\)-cochain \(c_{\MooreComplex[0]} \in \CochainComplex[2](\MooreComplex[0] G, M)\) defined by \((h_0, g_0) c_{\MooreComplex[0]} := (1, h_0, g_0) c\) for \(g_0, h_0 \in \MooreComplex[0] G\) is called the \newnotion{\(\MooreComplex[0]\)-part} of \(c\).
\item \label{def:moore_decomposition_of_analysed_2-cochains:crossed_modules} We let \(V\) be a crossed module and \(M\) be an abelian \(\HomotopyGroup[0](V)\)-module. Given a \(2\)-cochain \(c \in \CochainComplex[2](V, M)\), we call the \(\MooreComplex[1]\)-part of \(c\) also the \newnotion{module part} of \(c\) and write \(c_{\ModulePart} := c_{\MooreComplex[1]}\), and we call the \(\MooreComplex[0]\)-part of \(c\) also the \newnotion{group part} of \(c\) and write \(c_{\GroupPart} := c_{\MooreComplex[0]}\). That is, \((m) c_{\ModulePart} = (m, 1, 1) c\) for \(m \in \ModulePart V\) and \((h, g) c_{\GroupPart} = (1, h, g) c\) for \(g, h \in \GroupPart V\).
\end{enumerate}
\end{definition}

\begin{proposition} \label{prop:characterising_properties_of_analysed_2-cocycles} \
\begin{enumerate}
\item \label{prop:characterising_properties_of_analysed_2-cocycles:simplicial_groups} We suppose given a simplicial group \(G\) and an abelian \(\HomotopyGroup[0](G)\)-module \(M\). An analysed \(2\)-cochain \(z \in \CochainComplex[2]_{\text{an}}(G, M)\) is an analysed \(2\)-cocycle if and only if it fulfills the following conditions.
\begin{enumerate}
\item \label{prop:characterising_properties_of_analysed_2-cocycles:simplicial_groups:decomposition} We have \((g_1, h_0, g_0) z = (g_1) z_{\MooreComplex[1]} - (g_1 \differential, h_0) z_{\MooreComplex[0]} + (h_0, g_0) z_{\MooreComplex[0]}\) for \(g_1 \in \MooreComplex[1] G\), \(g_0, h_0 \in \MooreComplex[0] G\).
\item \label{prop:characterising_properties_of_analysed_2-cocycles:simplicial_groups:group_part_is_2-cocycle} The \(\MooreComplex[0]\)-part \(z_{\MooreComplex[0]}\) is a \(2\)-cocycle of \(\MooreComplex[0] G\) with coefficients in \(M\), that is, \(z_{\MooreComplex[0]} \in \CocycleGroup[2](\MooreComplex[0] G, M)\).
\item \label{prop:characterising_properties_of_analysed_2-cocycles:simplicial_groups:multiplicativity} We have \((h_1 g_1) z_{\MooreComplex[1]} = (h_1) z_{\MooreComplex[1]} + (g_1) z_{\MooreComplex[1]} - (h_1 \differential, g_1 \differential) z_{\MooreComplex[0]}\) for \(g_1, h_1 \in \MooreComplex[1] G\).
\item \label{prop:characterising_properties_of_analysed_2-cocycles:simplicial_groups:action} We have \(({^{g_0 \degeneracy{0}}}g_1) z_{\MooreComplex[1]} = g_0 \BoundaryGroup[0] \MooreComplex G \act (g_1) z_{\MooreComplex[1]} + ({^{g_0}}(g_1 \differential), g_0) z_{\MooreComplex[0]} - (g_0, g_1 \differential) z_{\MooreComplex[0]}\) for \(g_1 \in \MooreComplex[1] G\), \(g_0 \in \MooreComplex[0] G\).
\item \label{prop:characterising_properties_of_analysed_2-cocycles:simplicial_groups:m2-images} We have \((g_2 \differential) z_{\MooreComplex[1]} = (1) z_{\MooreComplex[1]}\) for \(g_2 \in \MooreComplex[2] G\).
\end{enumerate}
\item \label{prop:characterising_properties_of_analysed_2-cocycles:crossed_modules} We suppose given a crossed module \(V\) and an abelian \(\HomotopyGroup[0](V)\)-module \(M\). A \(2\)-cochain \(z \in \CochainComplex[2](V, M)\) is a \(2\)-cocycle if and only if it fulfills the following conditions.
\begin{enumerate}
\item \label{prop:characterising_properties_of_analysed_2-cocycles:crossed_modules:decomposition} We have \((m, h, g) z = (m) z_{\ModulePart} - (m, h) z_{\GroupPart} + (h, g) z_{\GroupPart}\) for \(m \in \ModulePart V\), \(g, h \in \GroupPart V\).
\item \label{prop:characterising_properties_of_analysed_2-cocycles:crossed_modules:group_part_is_2-cocycle} The group part \(z_{\GroupPart}\) is a \(2\)-cocycle of \(\GroupPart V\) with coefficients in \(M\), that is, \(z_{\GroupPart} \in \CocycleGroup[2](\GroupPart V, M)\).
\item \label{prop:characterising_properties_of_analysed_2-cocycles:crossed_modules:multiplicativity} We have \((n m) z_{\ModulePart} = (n) z_{\ModulePart} + (m) z_{\ModulePart} - (n, m) z_{\GroupPart}\) for \(m, n \in \ModulePart V\).
\item \label{prop:characterising_properties_of_analysed_2-cocycles:crossed_modules:action} We have \(({^g}m) z_{\ModulePart} = g (\Image \structuremorphism) \act (m) z_{\ModulePart} + ({^g}m, g) z_{\GroupPart} - (g, m) z_{\GroupPart}\) for \(m \in \ModulePart V\), \(g \in \GroupPart V\).
\end{enumerate}
\end{enumerate}
\end{proposition}
\begin{proof} \
\begin{enumerate}
\item First, we suppose given an analysed \(2\)-cocycle \(z \in \CocycleGroup[2]_{\text{an}}(G, M)\). We verify the asserted formulas:
\begin{enumerate}
\setcounter{enumii}{1}
\item We have
\begin{align*}
0 & = (1, 1, 1, k_0, 1, h_0, g_0) (z \differential) \\
& = (1, k_0, h_0) z - (1, k_0, h_0 g_0) z + (1, k_0 h_0, g_0) z - k_0 \BoundaryGroup[0] \MooreComplex G \act (1, h_0, g_0) z \\
& = (k_0, h_0) z_{\MooreComplex[0]} - (k_0, h_0 g_0) z_{\MooreComplex[0]} + (k_0 h_0, g_0) z_{\MooreComplex[0]} - k_0 \BoundaryGroup[0] \MooreComplex G \act (h_0, g_0) z_{\MooreComplex[0]}
\end{align*}
for \(g_0, h_0, k_0 \in \MooreComplex[0] G\), that is, \(z_{\MooreComplex[0]} \in \CocycleGroup[2](\MooreComplex[0] G, M)\).
\setcounter{enumii}{0}
\item First, we prove the formula for \(h_0 = 1\), then for \(g_0 = 1\) and finally for the general case.

We have
\begin{align*}
0 & = (1, g_1, 1, 1, 1, g_0, g_0^{- 1}) (z \differential) = (g_1, 1, g_0) z - (g_1, 1, 1) z + (1, g_0, g_0^{- 1}) z - (1, g_0, g_0^{- 1}) z \\
& = (g_1, 1, g_0) z - (g_1) z_{\MooreComplex[1]},
\end{align*}
that is, \((g_1, 1, g_0) z = (g_1) z_{\MooreComplex[1]}\) for \(g_1 \in \MooreComplex[1] G\), \(g_0 \in \MooreComplex[0] G\).

Next, we obtain
\begin{align*}
0 & = (1, 1, g_1, 1, 1, h_0, 1) (z \differential) = (1, g_1 \differential, h_0) z - (g_1, 1, h_0) z + (g_1, h_0, 1) z - (1, h_0, 1) z \\
& = (g_1 \differential, h_0) z_{\MooreComplex[0]} - (g_1) z_{\MooreComplex[1]} + (g_1, h_0, 1) z - (h_0, 1) z_{\MooreComplex[0]},
\end{align*}
that is, \((g_1, h_0, 1) z = (g_1) z_{\MooreComplex[1]} - (g_1 \differential, h_0) z_{\MooreComplex[0]} + (h_0, 1) z_{\MooreComplex[0]}\) for \(g_1 \in \MooreComplex[1] G\), \(h_0 \in \MooreComplex[0] G\).

Finally, we get, using~\ref{prop:characterising_properties_of_analysed_2-cocycles:simplicial_groups:group_part_is_2-cocycle} and remark~\ref{rem:projections_of_2-cocycles_of_groups}, 
\begin{align*}
0 & = (1, g_1, 1, h_0, 1, 1, g_0) (z \differential) = (g_1, h_0, 1) z - (g_1, h_0, g_0) z + (1, h_0, g_0) z - h_0 \BoundaryGroup[0] \MooreComplex G \act (1, 1, g_0) z \\
& = (g_1) z_{\MooreComplex[1]} - (g_1 \differential, h_0) z_{\MooreComplex[0]} + (h_0, 1) z_{\MooreComplex[0]} - (g_1, h_0, g_0) z + (h_0, g_0) z_{\MooreComplex[0]} - h_0 \BoundaryGroup[0] \MooreComplex G \act (1, g_0) z_{\MooreComplex[0]} \\
& = (g_1) z_{\MooreComplex[1]} - (g_1 \differential, h_0) z_{\MooreComplex[0]} - (g_1, h_0, g_0) z + (h_0, g_0) z_{\MooreComplex[0]}
\end{align*}
that is, \((g_1, h_0, g_0) z = (g_1) z_{\MooreComplex[1]} - (g_1 \differential, h_0) z_{\MooreComplex[0]} + (h_0, g_0) z_{\MooreComplex[0]}\) for \(g_1 \in \MooreComplex[1] G\), \(g_0, h_0 \in \MooreComplex[0] G\).
\setcounter{enumii}{2}
\item We have
\begin{align*}
0 & = (1, 1, h_1, 1, g_1, 1, 1) (z \differential) = (1, h_1 \differential, g_1 \differential) z - (h_1, 1, 1) z + (h_1 g_1, 1, 1) z - (g_1, 1, 1) z \\
& = (h_1 \differential, g_1 \differential) z_{\MooreComplex[0]} - (h_1) z_{\MooreComplex[1]} + (h_1 g_1) z_{\MooreComplex[1]} - (g_1) z_{\MooreComplex[1]},
\end{align*}
that is, \((h_1 g_1) z_{\MooreComplex[1]} = (h_1) z_{\MooreComplex[1]} + (g_1) z_{\MooreComplex[1]} - (h_1 \differential, g_1 \differential) z_{\MooreComplex[0]}\) for \(g_1, h_1 \in \MooreComplex[1] G\).
\item We have, using~\ref{prop:characterising_properties_of_analysed_2-cocycles:simplicial_groups:decomposition},
\begin{align*}
0 & = (1, 1, 1, g_0, g_1, 1, 1) (z \differential) = (1, g_0, g_1 \differential) z - (1, g_0, 1) z + ({^{g_0 \degeneracy{0}}}g_1, g_0, 1) z - g_0 \BoundaryGroup[0] \MooreComplex G \act (g_1, 1, 1) z \\
& = (g_0, g_1 \differential) z_{\MooreComplex[0]} + ({^{g_0 \degeneracy{0}}}g_1) z_{\MooreComplex[1]} - ({^{g_0}}(g_1 \differential), g_0) z_{\MooreComplex[0]} - g_0 \BoundaryGroup[0] \MooreComplex G \act (g_1) z_{\MooreComplex[1]},
\end{align*}
that is, \(({^{g_0 \degeneracy{0}}}g_1) z_{\MooreComplex[1]} = g_0 \BoundaryGroup[0] \MooreComplex G \act (g_1) z_{\MooreComplex[1]} + ({^{g_0}}(g_1 \differential), g_0) z_{\MooreComplex[0]} - (g_0, g_1 \differential) z_{\MooreComplex[0]}\) for \(g_1 \in \MooreComplex[1] G\), \(g_0 \in \MooreComplex[0] G\).
\item We have
\begin{align*}
0 & = (g_2, 1, 1, 1, 1, 1, 1) (z \differential) = (g_2 \differential, 1, 1) z - (1, 1, 1) z + (1, 1, 1) z - (1, 1, 1) z \\
& = (g_2 \differential) z_{\MooreComplex[1]} - (1) z_{\MooreComplex[1]},
\end{align*}
that is, \((g_2 \differential) z_{\MooreComplex[1]} = (1) z_{\MooreComplex[1]}\) for \(g_2 \in \MooreComplex[2] G\).
\end{enumerate}

Now let us conversely suppose given an analysed \(2\)-cochain \(z \in \CochainComplex[2]_{\text{an}}(G, M)\) that fulfills the properties~\ref{prop:characterising_properties_of_analysed_2-cocycles:simplicial_groups:decomposition} to~\ref{prop:characterising_properties_of_analysed_2-cocycles:simplicial_groups:m2-images}. Then we compute
\begin{align*}
& (g_2, k_1, h_1, g_1, k_0, h_0, g_0) (z \differential) \\
& = ((g_2 \differential) k_1, (h_1 \differential) k_0, (g_1 \differential) h_0) z - (k_1 h_1, k_0, h_0 g_0) z + (h_1 \, {^{k_0 \degeneracy{0}}}g_1, k_0 h_0, g_0) z - k_0 \BoundaryGroup[0] \MooreComplex G \act (g_1, h_0, g_0) z \\
& = ((g_2 \differential) k_1) z_{\MooreComplex[1]} - (k_1 \differential, (h_1 \differential) k_0) z_{\MooreComplex[0]} + ((h_1 \differential) k_0, (g_1 \differential) h_0) z_{\MooreComplex[0]} - (k_1 h_1) z_{\MooreComplex[1]} + ((k_1 h_1) \differential, k_0) z_{\MooreComplex[0]} \\
& \qquad - (k_0, h_0 g_0) z_{\MooreComplex[0]} + (h_1 \, {^{k_0 \degeneracy{0}}}g_1) z_{\MooreComplex[1]} - ((h_1 \, {^{k_0 \degeneracy{0}}}g_1) \differential, k_0 h_0) z_{\MooreComplex[0]} + (k_0 h_0, g_0) z_{\MooreComplex[0]} \\
& \qquad - k_0 \BoundaryGroup[0] \MooreComplex G \act (g_1) z_{\MooreComplex[1]} + k_0 \BoundaryGroup[0] \MooreComplex G \act (g_1 \differential, h_0) z_{\MooreComplex[0]} - k_0 \BoundaryGroup[0] \MooreComplex G \act (h_0, g_0) z_{\MooreComplex[0]} \\
& = ((g_2 \differential) k_1) z_{\MooreComplex[1]} - (k_1 h_1) z_{\MooreComplex[1]} + (h_1 \, {^{k_0 \degeneracy{0}}}g_1) z_{\MooreComplex[1]} - k_0 \BoundaryGroup[0] \MooreComplex G \act (g_1) z_{\MooreComplex[1]} - (k_1 \differential, (h_1 \differential) k_0) z_{\MooreComplex[0]} \\
& \qquad + ((h_1 \differential) k_0, (g_1 \differential) h_0) z_{\MooreComplex[0]} + ((k_1 \differential) (h_1 \differential), k_0) z_{\MooreComplex[0]} - ((h_1 \differential) \, {^{k_0}}(g_1 \differential), k_0 h_0) z_{\MooreComplex[0]} \\
& \qquad + k_0 \BoundaryGroup[0] \MooreComplex G \act (g_1 \differential, h_0) z_{\MooreComplex[0]} - (k_0, h_0 g_0) z_{\MooreComplex[0]} + (k_0 h_0, g_0) z_{\MooreComplex[0]} - k_0 \BoundaryGroup[0] \MooreComplex G \act (h_0, g_0) z_{\MooreComplex[0]} \\
& = (g_2 \differential) z_{\MooreComplex[1]} + (k_1) z_{\MooreComplex[1]} - (1, k_1 \differential) z_{\MooreComplex[0]} - (k_1 h_1) z_{\MooreComplex[1]} + (h_1) z_{\MooreComplex[1]} + ({^{k_0 \degeneracy{0}}}g_1) z_{\MooreComplex[1]} - (h_1 \differential, ({^{k_0 \degeneracy{0}}}g_1) \differential) z_{\MooreComplex[0]} \\
& \qquad - ({^{k_0 \degeneracy{0}}}g_1) z_{\MooreComplex[1]} + (({^{k_0 \degeneracy{0}}}g_1) \differential, k_0) z_{\MooreComplex[0]} - (k_0, g_1 \differential) z_{\MooreComplex[0]} - (k_1 \differential, (h_1 \differential) k_0) z_{\MooreComplex[0]} + ((h_1 \differential) k_0, (g_1 \differential) h_0) z_{\MooreComplex[0]} \\
& \qquad + ((k_1 \differential) (h_1 \differential), k_0) z_{\MooreComplex[0]} - ((h_1 \differential) \, {^{k_0}}(g_1 \differential), k_0 h_0) z_{\MooreComplex[0]} + k_0 \BoundaryGroup[0] \MooreComplex G \act (g_1 \differential, h_0) z_{\MooreComplex[0]} - (k_0, h_0) z_{\MooreComplex[0]} \\
& = (k_1) z_{\MooreComplex[1]} - (k_1 h_1) z_{\MooreComplex[1]} + (h_1) z_{\MooreComplex[1]} - (h_1 \differential, {^{k_0}}(g_1 \differential)) z_{\MooreComplex[0]} + ({^{k_0}}(g_1 \differential), k_0) z_{\MooreComplex[0]} - (k_0, g_1 \differential) z_{\MooreComplex[0]} \\
& \qquad - (k_1 \differential, (h_1 \differential) k_0) z_{\MooreComplex[0]} + ((h_1 \differential) k_0, (g_1 \differential) h_0) z_{\MooreComplex[0]} + ((k_1 \differential) (h_1 \differential), k_0) z_{\MooreComplex[0]} - ((h_1 \differential) \, {^{k_0}}(g_1 \differential), k_0 h_0) z_{\MooreComplex[0]} \\
& \qquad + k_0 \BoundaryGroup[0] \MooreComplex G \act (g_1 \differential, h_0) z_{\MooreComplex[0]} - (k_0, h_0) z_{\MooreComplex[0]} \\
& = (k_1 \differential, h_1 \differential) z_{\MooreComplex[0]} - (k_1 \differential, (h_1 \differential) k_0) z_{\MooreComplex[0]} + ((k_1 \differential) (h_1 \differential), k_0) z_{\MooreComplex[0]} - (h_1 \differential, {^{k_0}}(g_1 \differential)) z_{\MooreComplex[0]} \\
& \qquad - ((h_1 \differential) \, {^{k_0}}(g_1 \differential), k_0 h_0) z_{\MooreComplex[0]} + ({^{k_0}}(g_1 \differential), k_0) z_{\MooreComplex[0]} - (k_0, h_0) z_{\MooreComplex[0]} - (k_0, g_1 \differential) z_{\MooreComplex[0]} \\
& \qquad + k_0 \BoundaryGroup[0] \MooreComplex G \act (g_1 \differential, h_0) z_{\MooreComplex[0]} + ((h_1 \differential) k_0, (g_1 \differential) h_0) z_{\MooreComplex[0]} \\
& = (h_1 \differential, k_0) z_{\MooreComplex[0]} - ({^{k_0}}(g_1 \differential), k_0 h_0) z_{\MooreComplex[0]} - (h_1 \differential, {^{k_0}}(g_1 \differential) k_0 h_0) z_{\MooreComplex[0]} + ({^{k_0}}(g_1 \differential), k_0 h_0) z_{\MooreComplex[0]} \\
& \qquad - ({^{k_0}}(g_1 \differential) k_0, h_0) z_{\MooreComplex[0]} - (k_0, (g_1 \differential) h_0) z_{\MooreComplex[0]} + (k_0 (g_1 \differential), h_0) z_{\MooreComplex[0]} + ((h_1 \differential) k_0, (g_1 \differential) h_0) z_{\MooreComplex[0]} \\
& = (h_1 \differential, k_0) z_{\MooreComplex[0]} + ((h_1 \differential) k_0, (g_1 \differential) h_0) z_{\MooreComplex[0]} - (h_1 \differential, k_0 (g_1 \differential) h_0) z_{\MooreComplex[0]} - (k_0, (g_1 \differential) h_0) z_{\MooreComplex[0]} = 0
\end{align*}
for all \(g_0, h_0, k_0 \in \MooreComplex[0] G\), \(g_1, h_1, k_1 \in \MooreComplex[1] G\), \(g_2 \in \MooreComplex[2] G\), that is, \(z \in \CocycleGroup[2]_{\text{an}}(G, M)\).
\item This follows from~\ref{prop:characterising_properties_of_analysed_2-cocycles:simplicial_groups} by definition of the \(2\)-cocycles of \(V\) via \(\Coskeleton_1 V\) and the fact that \(\MooreComplex[0] \Coskeleton_1 V = \GroupPart V\), \(\MooreComplex[1] \Coskeleton_1 V = \ModulePart V\) and \(\MooreComplex[2] \Coskeleton_1 V = \{1\}\) (up to simplified notation). \qedhere
\end{enumerate}
\end{proof}

With the preceeding proposition we can now establish a description of the second analysed cocycle group of a simplicial group resp.\ of a crossed module as a pullback. This can be seen as a continuation of proposition~\ref{prop:analysed_1-cocycle_group_of_simplicial_groups_as_a_kernel}.

\begin{corollary} \label{cor:moore_decomposition_determines_the_2-cocycle} \
\begin{enumerate}
\item \label{cor:moore_decomposition_determines_the_2-cocycle:simplicial_groups} Given a simplicial group \(G\) and an abelian \(\HomotopyGroup[0](G)\)-module \(M\), the diagram
\[\begin{tikzpicture}[baseline=(m-2-1.base)]
  \matrix (m) [matrix of math nodes, row sep=2.5em, column sep=11.5em, column 1/.style={anchor=base east}, column 2/.style={anchor=base west}, text height=1.6ex, text depth=0.45ex, inner sep=0pt, nodes={inner sep=0.333em}]{
    \CocycleGroup[2]_{\text{an}}(G, M) & \CochainComplex[1](\MooreComplex[1] G, M) \\
    \CocycleGroup[2](\MooreComplex[0] G, M) & \CochainComplex[2](\MooreComplex[1] G, M) \directprod \CochainComplex[1](\MooreComplex[1] G \cart \MooreComplex[0] G, M) \directprod \CochainComplex[1](\MooreComplex[2] G, M) \\};
  \path[->, font=\scriptsize] let \p1=(1.25em, 0) in
    (m-1-1) edge node[above] {\({-}_{\MooreComplex[1]}|_{\CocycleGroup[2]_{\text{an}}(G, M)}\)} (m-1-2)
    ($(m-1-1.south east)-(\p1)$) edge node[left] {\({-}_{\MooreComplex[0]}|_{\CocycleGroup[2]_{\text{an}}(G, M)}^{\CocycleGroup[2](\MooreComplex[0] G, M)}\)} ($(m-2-1.north east)-(\p1)$)
    ($(m-1-2.south west)+(\p1)$) edge node[right] {\(\begin{smallpmatrix} \differential^{\CochainComplex(\MooreComplex G, M)} & \,\,\alpha_1\,\, & \CochainComplex[1](\differential^{\MooreComplex G}, M) \end{smallpmatrix}\)} ($(m-2-2.north west)+(\p1)$)
    (m-2-1) edge node[above] {\(\inc \begin{smallpmatrix} \CochainComplex[2](\differential^{\MooreComplex G}, M) & \,\,\alpha_0\,\, & \Map(1, M) \end{smallpmatrix}\)} (m-2-2);
\end{tikzpicture}\]
is a pullback of abelian groups, where \((g_1, g_0) (c_1 \alpha_1) := ({^{g_0 \degeneracy{0}}}g_1) c_1 - g_0 \BoundaryGroup[0] \MooreComplex G \act (g_1) c_1\) and \((g_1, g_0) (c_0 \alpha_0) := ({^{g_0}}(g_1 \differential), g_0) c_0 - (g_0, g_1 \differential) c_0\) for \(g_1 \in \MooreComplex[1] G\), \(g_0 \in \MooreComplex[0] G\), \(c_1 \in \CochainComplex[1](\MooreComplex[1] G, M)\), \(c_0 \in \CochainComplex[2](\MooreComplex[0] G, M)\), and where \(M\) is considered as a trivial \(\MooreComplex[1] G\)-module.
\item \label{cor:moore_decomposition_determines_the_2-cocycle:crossed_modules} Given a crossed module \(V\) and an abelian \(\HomotopyGroup[0](V)\)-module \(M\), the diagram
\[\begin{tikzpicture}[baseline=(m-2-1.base)]
  \matrix (m) [matrix of math nodes, row sep=2.5em, column sep=7.0em, column 1/.style={anchor=base east}, column 2/.style={anchor=base west}, text height=1.6ex, text depth=0.45ex, inner sep=0pt, nodes={inner sep=0.333em}]{
    \CocycleGroup[2](V, M) & \CochainComplex[1](\ModulePart V, M) \\
    \CocycleGroup[2](\GroupPart V, M) & \CochainComplex[2](\ModulePart V, M) \directprod \CochainComplex[1](\ModulePart V \cart \GroupPart V, M) \\};
  \path[->, font=\scriptsize] let \p1=(1.25em, 0) in
    (m-1-1) edge node[above] {\({-}_{\ModulePart}|_{\CocycleGroup[2](V, M)}\)} (m-1-2)
    ($(m-1-1.south east)-(\p1)$) edge node[left] {\({-}_{\GroupPart}|_{\CocycleGroup[2](V, M)}^{\CocycleGroup[2](\GroupPart V, M)}\)} ($(m-2-1.north east)-(\p1)$)
    ($(m-1-2.south west)+(\p1)$) edge node[right] {\(\begin{smallpmatrix} \differential^{\CochainComplex(\ModulePart V, M)} & \,\,\alpha_1 \end{smallpmatrix}\)} ($(m-2-2.north west)+(\p1)$)
    (m-2-1) edge node[above] {\(\inc \begin{smallpmatrix} \CochainComplex[2](\structuremorphism, M) & \,\,\alpha_0 \end{smallpmatrix}\)} (m-2-2);
\end{tikzpicture}\]
is a pullback of abelian groups, where \((m, g) (c_1 \alpha_1) := ({^g}m) c_1 - g (\Image \structuremorphism) \act (m) c_1\) and \((m, g) (c_0 \alpha_0) := ({^g}m, g) c_0 - (g, m) c_0\) for \(m \in \ModulePart V\), \(g \in \GroupPart V\), \(c_1 \in \CochainComplex[1](\ModulePart V, M)\), \(c_0 \in \CochainComplex[2](\GroupPart V, M)\), and where \(M\) is considered as a trivial \(\ModulePart V\)-module. In particular, we have an isomorphism
\begin{align*}
\CocycleGroup[2](V, M) & \map \{(c_1, z_0) \in \CochainComplex[1](\ModulePart V, M) \directprod \CocycleGroup[2](\GroupPart V, M) \mid \text{\((n m) c_1 = n c_1 + m c_1 - (n, m) z_0\) and } \\
& \qquad \text{\(({^g}m) c_1 = g (\Image \structuremorphism) \act (m) c_1 + ({^g}m, g) z_0 - (g, m) z_0\) for all \(m, n \in \ModulePart V\), \(g \in \GroupPart V\)}\}, \\
z & \mapsto (z_{\MooreComplex[1]}, z_{\MooreComplex[0]}).
\end{align*}
\end{enumerate}
\end{corollary}
\begin{proof} \
\begin{enumerate}
\item We note that \(\alpha_0\) and \(\alpha_1\) are group homomorphisms. By proposition~\ref{prop:characterising_properties_of_analysed_2-cocycles}\ref{prop:characterising_properties_of_analysed_2-cocycles:simplicial_groups}\ref{prop:characterising_properties_of_analysed_2-cocycles:simplicial_groups:group_part_is_2-cocycle} to~\ref{prop:characterising_properties_of_analysed_2-cocycles:simplicial_groups:m2-images}, the diagram is well-defined and commutes. To show that it is a pullback, we suppose given an arbitrary abelian group \(T\) and group homomorphisms \(\varphi_0\colon T \map \CocycleGroup[2](\MooreComplex[0] G, M)\) and \(\varphi_1\colon T \map \CochainComplex[1](\MooreComplex[1] G, M)\) with \(\varphi_1 \differential^{\CochainComplex(\MooreComplex G, M)} = \varphi_0 \inc \CochainComplex[2](\differential^{\MooreComplex G}, M)\), \(\varphi_1 \alpha_1 = \varphi_0 \inc \alpha_0\) and \(\varphi_1 \CochainComplex[1](\differential^{\MooreComplex G}, M) = \varphi_0 \inc \Map(1, M)\). For every \(t \in T\), we define a \(2\)-cochain \(t \varphi \in \CochainComplex[2]_{\text{an}}(G, M)\) by \((g_1, h_0, g_0) (t \varphi) := (g_1) (t \varphi_1) - (g_1 \differential, h_0) (t \varphi_0) + (h_0, g_0) (t \varphi_0)\) for \(g_1 \in \MooreComplex[1] G\), \(g_0, h_0 \in \MooreComplex[0] G\). Since
\[(g_1) (t \varphi)_{\MooreComplex[1]} = (g_1, 1, 1) (t \varphi) = (g_1) (t \varphi_1) - (g_1 \differential, 1) (t \varphi_0) + (1, 1) (t \varphi_0) = (g_1) (t \varphi_1)\]
for all \(g_1 \in \MooreComplex[1] G\) and
\[(h_0, g_0) (t \varphi)_{\MooreComplex[0]} = (1, h_0, g_0) (t \varphi) = (1) (t \varphi_1) - (1, h_0) (t \varphi_0) + (h_0, g_0) (t \varphi_0) = (h_0, g_0) (t \varphi_0)\]
for all \(g_0, h_0 \in \MooreComplex[0] G\), it follows that \((t \varphi)_{\MooreComplex[1]} = t \varphi_1\) and \((t \varphi)_{\MooreComplex[0]} = t \varphi_0\) and hence \(t \varphi \in \CocycleGroup[2]_{\text{an}}(G, M)\) for all \(t \in T\) by proposition~\ref{prop:characterising_properties_of_analysed_2-cocycles}\ref{prop:characterising_properties_of_analysed_2-cocycles:simplicial_groups}. Thus we obtain a well-defined group homomorphism \(\varphi\colon T \map \CocycleGroup[2]_{\text{an}}(G, M)\) with \((t \varphi)_{\MooreComplex[1]} = t \varphi_1\) and \((t \varphi)_{\MooreComplex[0]} = t \varphi_0\) for all \(t \in T\). The uniqueness of such a map follows from~\ref{prop:characterising_properties_of_analysed_2-cocycles}\ref{prop:characterising_properties_of_analysed_2-cocycles:simplicial_groups}\ref{prop:characterising_properties_of_analysed_2-cocycles:simplicial_groups:decomposition}. \qedhere
\end{enumerate}
\end{proof}

Now we are able to show that the second cohomology group of a simplicial group only depends on its \(1\)-segment.

\begin{proposition} \label{prop:second_cohomology_group_of_a_simplicial_group_and_its_1-truncation}
Given a simplicial group \(G\) and an abelian \(\HomotopyGroup[0](G)\)-module \(M\), the unit component \(\unit_G\colon G \map \Coskeleton_1 \Truncation^1 G\) of the adjunction \(\Truncation^1 \leftadjoint \Coskeleton_1\) induces an isomorphism
\[\CocycleGroup[2]_{\text{an}}(\unit_G, M)\colon \CocycleGroup[2](\Truncation^1 G, M) \map \CocycleGroup[2]_{\text{an}}(G, M),\]
which in turn induces isomorphisms \(\CoboundaryGroup[2](\unit_G, M)\) and \(\CohomologyGroup[2](\unit_G, M)\). In particular, we have
\[\CohomologyGroup[2](G, M) \isomorphic \CohomologyGroup[2](\Truncation^1 G, M).\]
\end{proposition}
\begin{proof}
For \(n \in \Naturals_0\), we denote by \(\varphi_n\) the isomorphisms from \(G_n\) to its semidirect product decomposition, cf.\ section~\ref{ssec:semidirect_product_decomposition}. Then we have \((g_0) \varphi_0^{- 1} (\unit_G)_0 = (g_0)\) and \((g_1, h_0) \varphi_1^{- 1} (\unit_G)_1 = (g_1 \canonicalepimorphism, h_0)\) for \(g_1 \in \MooreComplex[1] G\), \(g_0, h_0 \in \MooreComplex[0] G\), where we let \(\canonicalepimorphism\colon \MooreComplex[1] G \map \MooreComplex[1] G / \BoundaryGroup[1] \MooreComplex G = \ModulePart \Truncation^1 G\) denote the canonical epimorphism, cf.\ section~\ref{ssec:truncation_and_coskeleton}. Therefore the group homomorphism \(\CocycleGroup[2]_{\text{an}}(\unit_G, M)\) is given by \((g_1, h_0, g_0) (z' \CocycleGroup[2]_{\text{an}}(\unit_G, M)) = (g_1 \canonicalepimorphism, h_0, g_0) z'\) for \(g_1 \in \MooreComplex[1] G\), \(g_0, h_0 \in \MooreComplex[0] G\), \(z' \in \CocycleGroup[2](\Truncation^1 G, M)\). Thus we have \(z' \CocycleGroup[2]_{\text{an}}(\unit_G, M) = 0\) if and only if already \(z' = 0\), that is, \(\CocycleGroup[2]_{\text{an}}(\unit_G, M)\) is injective.

To show surjectivity, we suppose given an analysed \(2\)-cochain \(z \in \CocycleGroup[2]_{\text{an}}(G, M)\). We choose a section of the underlying pointed map of \(\canonicalepimorphism\), that is, a pointed map \(s\colon \ModulePart \Truncation^1 G \map \MooreComplex[1] G\) with \(s \canonicalepimorphism = \id_{\ModulePart \Truncation^1 G}\). Then \((n s) (m s) ((n m) s)^{- 1} \in \Kernel \canonicalepimorphism = \BoundaryGroup[1] \MooreComplex G\) and therefore
\[((n s) (m s)) z_{\MooreComplex[1]} = ((n s) (m s) ((n m) s)^{- 1} ((n m) s)) z_{\MooreComplex[1]} = ((n m) s) z_{\MooreComplex[1]}\]
for all \(m, n \in \ModulePart \Truncation^1 G\). Moreover, \((({^g}m) s) ({^{g \degeneracy{0}}}(m s))^{- 1} \in \Kernel \canonicalepimorphism = \BoundaryGroup[1] \MooreComplex G\) implies
\[(({^g}m) s) z_{\MooreComplex[1]} = ((({^g}m) s) ({^{g \degeneracy{0}}}(m s))^{- 1} \, {^{g \degeneracy{0}}}(m s)) z_{\MooreComplex[1]} = ({^{g \degeneracy{0}}}(m s)) z_{\MooreComplex[1]}\]
for all \(m \in \ModulePart \Truncation^1 G\), \(g \in \GroupPart \Truncation^1 G\). Defining \(c_1'\colon \ModulePart \Truncation^1 G \map M\) by \((m) c_1' := (m s) z_{\MooreComplex[1]}\) for \(m \in \ModulePart \Truncation^1 G\), we obtain
\[(n m) c_1' = ((n m) s) z_{\MooreComplex[1]} = ((n s) (m s)) z_{\MooreComplex[1]} = (n s) z_{\MooreComplex[1]} + (m s) z_{\MooreComplex[1]} - (n s \differential, m s \differential) = (n) c_1' + (m) c_1' - (n, m) z_{\MooreComplex[0]}\]
for all \(m, n \in \ModulePart \Truncation^1 G\) as well as
\begin{align*}
({^g}m) c_1' & = (({^g}m) s) z_{\MooreComplex[1]} = ({^{g \degeneracy{0}}}(m s)) z_{\MooreComplex[1]} = g \BoundaryGroup[0] \MooreComplex G \act (m s) z_{\MooreComplex[1]} + ({^g}(m s \differential), g) z_{\MooreComplex[0]} - (g, m s \differential) z_{\MooreComplex[0]} \\
& = g (\Image \structuremorphism) \act (m) c_1' + ({^g}m, g) z_{\MooreComplex[0]} - (g, m) z_{\MooreComplex[0]}
\end{align*}
for all \(m \in \ModulePart \Truncation^1 G\), \(g \in \GroupPart \Truncation^1 G\). Thus we get a well-defined \(2\)-cocycle \(z' \in \CocycleGroup[2](\Truncation^1 G, M)\) with \((m) z'_{\ModulePart} = (m s) z_{\MooreComplex[1]}\) for \(m \in \ModulePart \Truncation^1 G\) and \(z'_{\GroupPart} = z_{\MooreComplex[0]}\) by corollary~\ref{cor:moore_decomposition_determines_the_2-cocycle}\ref{cor:moore_decomposition_determines_the_2-cocycle:crossed_modules}. Further, \(g_1 (g_1 \canonicalepimorphism s)^{- 1} \in \Kernel \canonicalepimorphism = \BoundaryGroup[1] \MooreComplex G\) implies
\begin{align*}
0 & = (g_1 (g_1 \canonicalepimorphism s)^{- 1}) z_{\MooreComplex[1]} - (1) z_{\MooreComplex[1]} = (g_1) z_{\MooreComplex[1]} + ((g_1 \canonicalepimorphism s)^{- 1}) z_{\MooreComplex[1]} - (g_1 \differential, (g_1 \canonicalepimorphism s)^{- 1} \differential) z_{\MooreComplex[0]} - (1) z_{\MooreComplex[1]} \\
& = (g_1) z_{\MooreComplex[1]} + ((g_1 \canonicalepimorphism s)^{- 1}) z_{\MooreComplex[1]} - (g_1 \canonicalepimorphism s \differential, (g_1 \canonicalepimorphism s)^{- 1} \differential) z_{\MooreComplex[0]} - ((g_1 \canonicalepimorphism s) (g_1 \canonicalepimorphism s)^{- 1}) z_{\MooreComplex[1]} = (g_1) z_{\MooreComplex[1]} - (g_1 \canonicalepimorphism s) z_{\MooreComplex[1]}
\end{align*}
for all \(g_1 \in \MooreComplex[1] G\). But now it follows that \(z' \CocycleGroup[2]_{\text{an}}(\unit_G, M) = z\) since
\begin{align*}
(g_1, h_0, g_0) (z' \CocycleGroup[2]_{\text{an}}(\unit_G, M)) & = (g_1 \canonicalepimorphism, h_0, g_0) z' = (g_1 \canonicalepimorphism) z'_{\ModulePart} - (g_1 \canonicalepimorphism, h_0) z'_{\GroupPart} + (h_0, g_0) z'_{\GroupPart} \\
& = (g_1 \canonicalepimorphism s) z_{\MooreComplex[1]} - (g_1 \differential, h_0) z_{\MooreComplex[0]} + (h_0, g_0) z_{\MooreComplex[0]} \\
& = (g_1) z_{\MooreComplex[1]} - (g_1 \differential, h_0) z_{\MooreComplex[0]} + (h_0, g_0) z_{\MooreComplex[0]} = (g_1, h_0, g_0) z
\end{align*}
for all \(g_1 \in \MooreComplex[1] G\), \(g_0, h_0 \in \MooreComplex[0] G\). Thus \(\CocycleGroup[2]_{\text{an}}(\unit_G, M)\) is surjective. Altogether, \(\CocycleGroup[2]_{\text{an}}(\unit_G, M)\) is bijective and hence an isomorphism of abelian groups.

The injectivity of \(\CocycleGroup[2]_{\text{an}}(\unit_G, M)\) implies the injectivity of the restriction \(\CoboundaryGroup[2]_{\text{an}}(\unit_G, M)\). To show that this is also an isomorphism, it remains to show that for a given analysed \(2\)-coboundary \(b \in \CoboundaryGroup[2]_{\text{an}}(G, M)\), the \(2\)-cocycle \(b' \in \CocycleGroup[2](\Truncation^1 G, M)\) given by \((m) b'_{\ModulePart} = (m s) b_{\MooreComplex[1]}\) for \(m \in \ModulePart \Truncation^1 G\) and \(b'_{\GroupPart} = b_{\MooreComplex[0]}\) is in fact a \(2\){\nbd}coboundary in \(\CoboundaryGroup[2](\Truncation^1 G, M)\). 

We choose \(c \in \CochainComplex[1]_{\text{an}}(G, M) = \CochainComplex[1](\Truncation^1 G, M)\) with \(b = c \differential^{\CochainComplex_{\text{an}}(G, M)}\), that is, with \((g_1, h_0, g_0) b = ((g_1 \differential) h_0) c - (h_0 g_0) c + h_0 \BoundaryGroup[0] \MooreComplex G \act (g_0) c\) for \(g_1 \in \MooreComplex[1] G\), \(g_0, h_0 \in \MooreComplex[0] G\). It follows that
\[(m) b'_{\ModulePart} = (m s) b_{\MooreComplex[1]} = (m s \differential) c = (m) c = (m) (c \differential^{\CochainComplex(\Truncation^1 G, M)})_{\ModulePart}\]
for all \(m \in \ModulePart \Truncation^1 G\), that is, \(b'_{\ModulePart} = (c \differential^{\CochainComplex(\Truncation^1 G, M)})_{\ModulePart}\), as well as
\[b'_{\GroupPart} = (c \differential^{\CochainComplex_{\text{an}}(G, M)})_{\MooreComplex[0]} = (c \differential^{\CochainComplex(\Truncation^1 G, M)})_{\GroupPart}.\]
Hence we have \(b' = c \differential^{\CochainComplex(\Truncation^1 G, M)} \in \CoboundaryGroup[2](\Truncation^1 G, M)\).

We have shown that \(\CocycleGroup[2]_{\text{an}}(\unit_G, M)\) and \(\CoboundaryGroup[2]_{\text{an}}(\unit_G, M)\) are isomorphisms, and hence \(\CohomologyGroup[2]_{\text{an}}(\unit_G, M)\) is also an isomorphism. In particular, we have
\[\CohomologyGroup[2](G, M) \isomorphic \CohomologyGroup[2]_{\text{an}}(G, M) \isomorphic \CohomologyGroup[2](\Truncation^1 G, M). \qedhere\]
\end{proof}

\section{\texorpdfstring{Crossed module extensions and standard $2$-cocycles}{Crossed module extensions and standard 2-cocycles}} \label{sec:crossed_module_extensions_and_standard_2-cocycles}

Throughout this section, we suppose given a group \(\Pi_0\) and abelian \(\Pi_0\)-modules \(\Pi_1\) and \(M\), where \(\Pi_1\) is written multiplicatively. Moreover, we suppose given a crossed module extension \(E\) of \(\Pi_0\) with \(\Pi_1\) and a section system \((s^1, s^0)\) for \(E\). The lifting system coming from \((s^1, s^0)\) will be denoted by \((Z^2, Z^1)\), that is, \(Z^1 = s^0\) and \(Z^2 = \cocycleofextension{2} s^1\). Cf.\ section~\ref{ssec:crossed_module_extensions}.

\begin{notation} \label{not:section_systems}
In this section, we use the following conventions and notations: For \(p, q, r \in \Pi_0\), we write \([p] := p Z^1\), \([q, p] := (q, p) Z^2\) and \([r, q, p] := (r, q, p) \cocycleofextension{3}\). For \(g \in \Image \structuremorphism\), we write \([g] := g s^1\). So for \(m \in \ModulePart E\), we usually write \([m] = [m \structuremorphism] = m \structuremorphism s^1\), following our convention from section~\ref{ssec:crossed_modules}. Finally, for \(g \in \GroupPart E\), we write \(\overline{g} := g \canonicalepimorphism\).

Wih these conventions, we have \(\overline{[p]} = p\), \([q, p] = [[q] [p] [q p]^{- 1}]\) and \([r, q, p] \canonicalmonomorphism = [r, q] [r q, p] [r, q p]^{- 1} \, {^{[r]}}([q, p]^{- 1})\) for \(p, q, r \in \Pi_0\) and \([m] \structuremorphism = m \structuremorphism\) for \(m \in \ModulePart E\).
\end{notation}

We have seen in section~\ref{ssec:pointed_cochains}, how the computation of cohomology groups in positive dimension can be reduced to that of pointed cohomology groups. In this section, we will see a further reduction in the case where we consider the second cohomology group of the underlying crossed module of a crossed module extension.

\begin{definition}[standardisation of pointed \(2\)-cocycles] \label{def:standardisation_of_pointed_2-cocycles} \
\begin{enumerate}
\item \label{def:standardisation_of_pointed_2-cocycles:standardisation_and_standardiser} Given a pointed \(2\)-cocycle \(z \in \CocycleGroup[2]_{\text{pt}}(E, M)\), the \newnotion{standardisation} of \(z\) (with respect to \((s^1, s^0)\)) is given by
\[\standardisation{z} = \standardisation[(s^1, s^0)]{z} := z - \standardiser{z} \differential,\]
where the \newnotion{standardiser} of \(z\) (with respect to \((s^1, s^0)\)) is defined to be the pointed \(1\)-cochain \(\standardiser{z} = \standardiser[(s^1, s^0)]{z} \in \CochainComplex[1]_{\text{pt}}(E, M)\) given by
\[(g) \standardiser{z} := ([g [\overline{g}]^{- 1}], [\overline{g}], 1) z\]
for \(g \in \GroupPart E\).
\item \label{def:standardisation_of_pointed_2-cocycles:standard_2-cocycles} A pointed \(2\)-cocycle \(z \in \CocycleGroup[2]_{\text{pt}}(E, M)\) is said to be \newnotion{standard} (with respect to \((s^1, s^0)\)) (or a \newnotion{standard \(2\)-cocycle}, for short) if \(\standardisation{z} = z\). The subgroup of \(\CocycleGroup[2]_{\text{pt}}(E, M)\) consisting of all standard \(2\)-cocycles of \(E\) with coefficients in \(M\) will be denoted by
\[\CocycleGroup[2]_{\text{st}}(E, M) = \CocycleGroup[2]_{\text{st}, (s^1, s^0)}(E, M) := \{z \in \CocycleGroup[2]_{\text{pt}}(E, M) \mid \standardisation{z} = z\}.\]
Likewise, the subgroup of \(\CoboundaryGroup[2]_{\text{pt}}(E, M)\) consisting of all standard \(2\)-coboundaries of \(E\) with coefficients in \(M\) will be denoted by
\[\CoboundaryGroup[2]_{\text{st}}(E, M) = \CoboundaryGroup[2]_{\text{st}, (s^1, s^0)}(E, M) := \{b \in \CoboundaryGroup[2]_{\text{pt}}(E, M) \mid \standardisation{b} = b\}.\]
Moreover, we set
\[\CohomologyGroup[2]_{\text{st}}(E, M) = \CohomologyGroup[2]_{\text{st}, (s^1, s^0)}(E, M) := \CocycleGroup[2]_{\text{st}}(E, M) / \CoboundaryGroup[2]_{\text{st}}(E, M).\]
\end{enumerate}
\end{definition}

\begin{remark} \label{rem:standardiser_splitted}
We have
\[(g) \standardiser{z} = ([g [\overline{g}]^{- 1}]) z_{\ModulePart} - (g [\overline{g}]^{- 1}, [\overline{g}]) z_{\GroupPart}\]
for \(g \in \GroupPart E\), \(z \in \CocycleGroup[2]_{\text{pt}}(E, M)\).
\end{remark}
\begin{proof}
This follows from proposition~\ref{prop:characterising_properties_of_analysed_2-cocycles}\ref{prop:characterising_properties_of_analysed_2-cocycles:crossed_modules}\ref{prop:characterising_properties_of_analysed_2-cocycles:crossed_modules:decomposition}.
\end{proof}

In the next proposition, we give more detailed formulas for the standardisation.

\begin{proposition} \label{prop:module_part_and_group_part_of_standardisation_of_pointed_2-cocycles_and_coboundaries} \
\begin{enumerate}
\item \label{prop:module_part_and_group_part_of_standardisation_of_pointed_2-cocycles_and_coboundaries:cocycles} For every pointed \(2\)-cocycle \(z \in \CocycleGroup[2]_{\text{pt}}(E, M)\), we have
\[(m) \standardisation{z}_{\ModulePart} = (m [m]^{- 1}) z_{\ModulePart}\]
for \(m \in \ModulePart E\), and
\[(h, g) \standardisation{z}_{\GroupPart} = ([h [\overline{h}]^{- 1}]^{- 1} \, {^h}([g [\overline{g}]^{- 1}]^{- 1}) [h g [\overline{h g}]^{- 1}]) z_{\ModulePart} - ([\overline{h}, \overline{g}], [\overline{h} \overline{g}]) z_{\GroupPart} + ([\overline{h}], [\overline{g}]) z_{\GroupPart}\]
for \(g, h \in \GroupPart E\).
\item \label{prop:module_part_and_group_part_of_standardisation_of_pointed_2-cocycles_and_coboundaries:coboundaries} For every pointed \(2\)-coboundary \(b \in \CoboundaryGroup[2]_{\text{pt}}(E, M)\), we have
\[(m) \standardisation{b}_{\ModulePart} = 0\]
for \(m \in \ModulePart E\), and, given \(c \in \CochainComplex[1]_{\text{pt}}(E, M)\) with \(b = c \differential\), we have
\[(h, g) \standardisation{b}_{\GroupPart} = (\overline{h}, \overline{g}) (c_0 \differential)\]
for \(g, h \in \GroupPart E\), where \(c_0 \in \CochainComplex[1](\Pi_0, M)\) is given by \((p) c_0 := ([p]) c\).
\end{enumerate}
\end{proposition}
\pagebreak % manual format
\begin{proof} \
\begin{enumerate}
\item We suppose given a pointed \(2\)-cocycle \(z \in \CocycleGroup[2]_{\text{pt}}(E, M)\). By proposition~\ref{prop:characterising_properties_of_analysed_2-cocycles}\ref{prop:characterising_properties_of_analysed_2-cocycles:crossed_modules}, we have
\begin{align*}
(m) \standardisation{z}_{\ModulePart} & = (m) z_{\ModulePart} - (m) (\standardiser{z} \differential)_{\ModulePart} = (m) z_{\ModulePart} - (m) \standardiser{z} \\
& = (m) z_{\ModulePart} - ([m]) z_{\ModulePart} = (m) z_{\ModulePart} + ([m]^{- 1}) z_{\ModulePart} - (m, m^{- 1}) z_{\GroupPart} = (m [m]^{- 1}) z_{\ModulePart}
\end{align*}
for \(m \in \ModulePart E\), and
\begin{align*}
& (h, g) \standardisation{z}_{\GroupPart} = (h, g) z_{\GroupPart} - (h, g) (\standardiser{z} \differential)_{\GroupPart} = (h, g) z_{\GroupPart} - (h) \standardiser{z} + (h g) \standardiser{z} - \overline{h} \act (g) \standardiser{z} \\
& = (h, g) z_{\GroupPart} - ([h [\overline{h}]^{- 1}]) z_{\ModulePart} + (h [\overline{h}]^{- 1}, [\overline{h}]) z_{\GroupPart} + ([h g [\overline{h g}]^{- 1}]) z_{\ModulePart} - (h g [\overline{h g}]^{- 1}, [\overline{h g}]) z_{\GroupPart} \\
& \qquad - \overline{h} \act ([g [\overline{g}]^{- 1}]) z_{\ModulePart} + \overline{h} \act (g [\overline{g}]^{- 1}, [\overline{g}]) z_{\GroupPart} \\
& = (h, g) z_{\GroupPart} + ([h [\overline{h}]^{- 1}]^{- 1}) z_{\ModulePart} - (h [\overline{h}]^{- 1}, [\overline{h}] h^{- 1}) z_{\GroupPart} + (h [\overline{h}]^{- 1}, [\overline{h}]) z_{\GroupPart} + ([h g [\overline{h g}]^{- 1}]) z_{\ModulePart} \\
& \qquad - (h g [\overline{h g}]^{- 1}, [\overline{h g}]) z_{\GroupPart} + \overline{h} \act ([g [\overline{g}]^{- 1}]^{- 1}) z_{\ModulePart} - \overline{h} \act (g [\overline{g}]^{- 1}, [\overline{g}] g^{- 1}) z_{\GroupPart} + \overline{h} \act (g [\overline{g}]^{- 1}, [\overline{g}]) z_{\GroupPart} \\
& = (h, g) z_{\GroupPart} + ([h [\overline{h}]^{- 1}]^{- 1}) z_{\ModulePart} - (h [\overline{h}]^{- 1}, [\overline{h}] h^{- 1}) z_{\GroupPart} + (h [\overline{h}]^{- 1}, [\overline{h}]) z_{\GroupPart} + ([h g [\overline{h g}]^{- 1}]) z_{\ModulePart} \\
& \qquad - (h g [\overline{h g}]^{- 1}, [\overline{h g}]) z_{\GroupPart} + ({^h}([g [\overline{g}]^{- 1}]^{- 1})) z_{\ModulePart} - ({^h}([\overline{g}] g^{- 1}), h) z_{\GroupPart} + (h, [\overline{g}] g^{- 1}) z_{\GroupPart} \\
& \qquad - \overline{h} \act (g [\overline{g}]^{- 1}, [\overline{g}] g^{- 1}) z_{\GroupPart} + \overline{h} \act (g [\overline{g}]^{- 1}, [\overline{g}]) z_{\GroupPart} \\
& = ([h [\overline{h}]^{- 1}]^{- 1}) z_{\ModulePart} + ({^h}([g [\overline{g}]^{- 1}]^{- 1})) z_{\ModulePart} + ([h g [\overline{h g}]^{- 1}]) z_{\ModulePart} + (h, g) z_{\GroupPart} - (h [\overline{h}]^{- 1}, [\overline{h}] h^{- 1}) z_{\GroupPart} \\
& \qquad + (h [\overline{h}]^{- 1}, [\overline{h}]) z_{\GroupPart} - (h g [\overline{h g}]^{- 1}, [\overline{h g}]) z_{\GroupPart} - ({^h}([\overline{g}] g^{- 1}), h) z_{\GroupPart} + (h, [\overline{g}] g^{- 1}) z_{\GroupPart} \\
& \qquad - \overline{h} \act (g [\overline{g}]^{- 1}, [\overline{g}] g^{- 1}) z_{\GroupPart} + \overline{h} \act (g [\overline{g}]^{- 1}, [\overline{g}]) z_{\GroupPart} \\
& = ([h [\overline{h}]^{- 1}]^{- 1}) z_{\ModulePart} + ({^h}([g [\overline{g}]^{- 1}]^{- 1}) [h g [\overline{h g}]^{- 1}]) z_{\ModulePart} + (h [\overline{g}] g^{- 1} h^{- 1}, h g [\overline{h g}]^{- 1}) z_{\GroupPart} + (h, g) z_{\GroupPart} \\
& \qquad - (h [\overline{h}]^{- 1}, [\overline{h}] h^{- 1}) z_{\GroupPart} + (h [\overline{h}]^{- 1}, [\overline{h}]) z_{\GroupPart} - (h g [\overline{h g}]^{- 1}, [\overline{h g}]) z_{\GroupPart} - ({^h}([\overline{g}] g^{- 1}), h) z_{\GroupPart} \\
& \qquad + (h, [\overline{g}] g^{- 1}) z_{\GroupPart} - \overline{h} \act (g [\overline{g}]^{- 1}, [\overline{g}] g^{- 1}) z_{\GroupPart} + \overline{h} \act (g [\overline{g}]^{- 1}, [\overline{g}]) z_{\GroupPart} \\
& = ([h [\overline{h}]^{- 1}]^{- 1} \, {^h}([g [\overline{g}]^{- 1}]^{- 1}) [h g [\overline{h g}]^{- 1}]) z_{\ModulePart} + ([\overline{h}] h^{- 1}, h [\overline{g}] [\overline{h g}]^{- 1}) z_{\GroupPart} + (h [\overline{g}] g^{- 1} h^{- 1}, h g [\overline{h g}]^{- 1}) z_{\GroupPart} \\
& \qquad + (h, g) z_{\GroupPart} - (h [\overline{h}]^{- 1}, [\overline{h}] h^{- 1}) z_{\GroupPart} + (h [\overline{h}]^{- 1}, [\overline{h}]) z_{\GroupPart} - (h g [\overline{h g}]^{- 1}, [\overline{h g}]) z_{\GroupPart} - (h [\overline{g}] g^{- 1} h^{- 1}, h) z_{\GroupPart} \\
& \qquad + (h, [\overline{g}] g^{- 1}) z_{\GroupPart} - \overline{h} \act (g [\overline{g}]^{- 1}, [\overline{g}] g^{- 1}) z_{\GroupPart} + \overline{h} \act (g [\overline{g}]^{- 1}, [\overline{g}]) z_{\GroupPart} \\
& = ([h [\overline{h}]^{- 1}]^{- 1} \, {^h}([g [\overline{g}]^{- 1}]^{- 1}) [h g [\overline{h g}]^{- 1}]) z_{\ModulePart} + ([\overline{h}] [\overline{g}], [\overline{h g}]^{- 1}) z_{\GroupPart} - (h [\overline{g}], [\overline{h g}]^{- 1}) z_{\GroupPart} \\
& \qquad + ([\overline{h}] h^{- 1}, h [\overline{g}]) z_{\GroupPart} + (h [\overline{g}], [\overline{h g}]^{- 1}) z_{\GroupPart} - (h g, [\overline{h g}]^{- 1}) z_{\GroupPart} + (h [\overline{g}] g^{- 1} h^{- 1}, h g) z_{\GroupPart} + (h, g) z_{\GroupPart} \\
& \qquad + (h g, [\overline{h g}]^{- 1}) z_{\GroupPart} - (h, h^{- 1}) z_{\GroupPart} + ([\overline{h}], h^{- 1}) z_{\GroupPart} - \overline{h g} \act ([\overline{h g}]^{- 1}, [\overline{h g}]) z_{\GroupPart} - (h [\overline{g}], g^{- 1}) z_{\GroupPart} \\
& \qquad + (h [\overline{g}], g^{- 1} h^{- 1}) z_{\GroupPart} - \overline{h g} \act (g^{- 1} h^{- 1}, h) z_{\GroupPart} + (h [\overline{g}], g^{- 1}) z_{\GroupPart} - \overline{h} \act ([\overline{g}], g^{- 1}) z_{\GroupPart} + (h, [\overline{g}]) z_{\GroupPart} \\
& \qquad - \overline{h} \act (g, g^{- 1}) z_{\GroupPart} + \overline{h} \act ([\overline{g}], g^{- 1}) z_{\GroupPart} \\
& = ([h [\overline{h}]^{- 1}]^{- 1} \, {^h}([g [\overline{g}]^{- 1}]^{- 1}) [h g [\overline{h g}]^{- 1}]) z_{\ModulePart} + ([\overline{h}] [\overline{g}], [\overline{h g}]^{- 1}) z_{\GroupPart} + ([\overline{h}] h^{- 1}, h [\overline{g}]) z_{\GroupPart} \\
& \qquad + (h [\overline{g}] g^{- 1} h^{- 1}, h g) z_{\GroupPart} + (h, g) z_{\GroupPart} - (h, h^{- 1}) z_{\GroupPart} + ([\overline{h}], h^{- 1}) z_{\GroupPart} - \overline{h g} \act ([\overline{h g}]^{- 1}, [\overline{h g}]) z_{\GroupPart} \\
& \qquad + (h [\overline{g}], g^{- 1} h^{- 1}) z_{\GroupPart} - \overline{h g} \act (g^{- 1} h^{- 1}, h) z_{\GroupPart} + (h, [\overline{g}]) z_{\GroupPart} - \overline{h} \act (g, g^{- 1}) z_{\GroupPart} \\
& = ([h [\overline{h}]^{- 1}]^{- 1} \, {^h}([g [\overline{g}]^{- 1}]^{- 1}) [h g [\overline{h g}]^{- 1}]) z_{\ModulePart} + ([\overline{h}] [\overline{g}], [\overline{h g}]^{- 1}) z_{\GroupPart} + ([\overline{h}], [\overline{g}]) z_{\GroupPart} - (h, [\overline{g}]) z_{\GroupPart} \\
& \qquad + ([\overline{h}] h^{- 1}, h) z_{\GroupPart} - (h [\overline{g}], g^{- 1} h^{- 1}) z_{\GroupPart} + \overline{h g} \act (g^{- 1} h^{- 1}, h g) z_{\GroupPart} + (h, g) z_{\GroupPart} - (h, h^{- 1}) z_{\GroupPart} \\
& \qquad + ([\overline{h}], h^{- 1}) z_{\GroupPart} - \overline{h g} \act ([\overline{h g}]^{- 1}, [\overline{h g}]) z_{\GroupPart} + (h [\overline{g}], g^{- 1} h^{- 1}) z_{\GroupPart} - \overline{h g} \act (g^{- 1} h^{- 1}, h) z_{\GroupPart} + (h, [\overline{g}]) z_{\GroupPart} \\
& \qquad - \overline{h} \act (g, g^{- 1}) z_{\GroupPart} \\
& = ([h [\overline{h}]^{- 1}]^{- 1} \, {^h}([g [\overline{g}]^{- 1}]^{- 1}) [h g [\overline{h g}]^{- 1}]) z_{\ModulePart} + ([\overline{h}] [\overline{g}], [\overline{h g}]^{- 1}) z_{\GroupPart} + ([\overline{h}], [\overline{g}]) z_{\GroupPart} + ([\overline{h}] h^{- 1}, h) z_{\GroupPart} \\
& \qquad + \overline{h g} \act (g^{- 1} h^{- 1}, h g) z_{\GroupPart} + (h, g) z_{\GroupPart} - (h, h^{- 1}) z_{\GroupPart} + ([\overline{h}], h^{- 1}) z_{\GroupPart} - \overline{h g} \act ([\overline{h g}]^{- 1}, [\overline{h g}]) z_{\GroupPart} \\
& \qquad - \overline{h g} \act (g^{- 1} h^{- 1}, h) z_{\GroupPart} - \overline{h} \act (g, g^{- 1}) z_{\GroupPart} \\
& = ([h [\overline{h}]^{- 1}]^{- 1} \, {^h}([g [\overline{g}]^{- 1}]^{- 1}) [h g [\overline{h g}]^{- 1}]) z_{\ModulePart} + ([\overline{h}] [\overline{g}], [\overline{h g}]^{- 1}) z_{\GroupPart} + ([\overline{h}], [\overline{g}]) z_{\GroupPart} - ([\overline{h}], h^{- 1}) z_{\GroupPart} \\
& \qquad + \overline{h} \act (h^{- 1}, h) z_{\GroupPart} + \overline{h g} \act (g^{- 1}, g) z_{\GroupPart} - (h, g) z_{\GroupPart} + \overline{h g} \act (g^{- 1} h^{- 1}, h) z_{\GroupPart} + (h, g) z_{\GroupPart} - (h, h^{- 1}) z_{\GroupPart} \\
& \qquad + ([\overline{h}], h^{- 1}) z_{\GroupPart} - \overline{h g} \act ([\overline{h g}]^{- 1}, [\overline{h g}]) z_{\GroupPart} - \overline{h g} \act (g^{- 1} h^{- 1}, h) z_{\GroupPart} - \overline{h} \act (g, g^{- 1}) z_{\GroupPart} \\
& = ([h [\overline{h}]^{- 1}]^{- 1} \, {^h}([g [\overline{g}]^{- 1}]^{- 1}) [h g [\overline{h g}]^{- 1}]) z_{\ModulePart} + ([\overline{h}] [\overline{g}], [\overline{h g}]^{- 1}) z_{\GroupPart} + ([\overline{h}], [\overline{g}]) z_{\GroupPart} + \overline{h} \act (h^{- 1}, h) z_{\GroupPart} \\
& \qquad + \overline{h g} \act (g^{- 1}, g) z_{\GroupPart} - (h, h^{- 1}) z_{\GroupPart} - \overline{h g} \act ([\overline{h g}]^{- 1}, [\overline{h g}]) z_{\GroupPart} - \overline{h} \act (g, g^{- 1}) z_{\GroupPart} \\
& = ([h [\overline{h}]^{- 1}]^{- 1} \, {^h}([g [\overline{g}]^{- 1}]^{- 1}) [h g [\overline{h g}]^{- 1}]) z_{\ModulePart} + ([\overline{h}] [\overline{g}], [\overline{h g}]^{- 1}) z_{\GroupPart} + ([\overline{h}], [\overline{g}]) z_{\GroupPart} + (h, h^{- 1}) z_{\GroupPart} \\
& \qquad + \overline{h} \act (g, g^{- 1}) z_{\GroupPart} - (h, h^{- 1}) z_{\GroupPart} - \overline{h g} \act ([\overline{h g}]^{- 1}, [\overline{h g}]) z_{\GroupPart} - \overline{h} \act (g, g^{- 1}) z_{\GroupPart} \\
& = ([h [\overline{h}]^{- 1}]^{- 1} \, {^h}([g [\overline{g}]^{- 1}]^{- 1}) [h g [\overline{h g}]^{- 1}]) z_{\ModulePart} + ([\overline{h}] [\overline{g}], [\overline{h g}]^{- 1}) z_{\GroupPart} - \overline{h g} \act ([\overline{h g}]^{- 1}, [\overline{h g}]) z_{\GroupPart} + ([\overline{h}], [\overline{g}]) z_{\GroupPart} \\
& = ([h [\overline{h}]^{- 1}]^{- 1} \, {^h}([g [\overline{g}]^{- 1}]^{- 1}) [h g [\overline{h g}]^{- 1}]) z_{\ModulePart} - ([\overline{h}] [\overline{g}] [\overline{h g}]^{- 1}, [\overline{h g}]) z_{\GroupPart} + ([\overline{h}], [\overline{g}]) z_{\GroupPart} \\
& = ([h [\overline{h}]^{- 1}]^{- 1} \, {^h}([g [\overline{g}]^{- 1}]^{- 1}) [h g [\overline{h g}]^{- 1}]) z_{\ModulePart} - ([\overline{h}, \overline{g}], [\overline{h g}]) z_{\GroupPart} + ([\overline{h}], [\overline{g}]) z_{\GroupPart}
\end{align*}
for \(g, h \in \GroupPart E\).
\item By~\ref{prop:module_part_and_group_part_of_standardisation_of_pointed_2-cocycles_and_coboundaries:cocycles}, we have
\[(m) \standardisation{b}_{\ModulePart} = (m [m]^{- 1}) b_{\ModulePart} = (m [m]^{- 1}) (c \differential)_{\ModulePart} = (m m^{- 1}) c = 0\]
for \(m \in \ModulePart E\) and
\begin{align*}
& (h, g) \standardisation{b}_{\GroupPart} = ([h [\overline{h}]^{- 1}]^{- 1} \, {^h}([g [\overline{g}]^{- 1}]^{- 1}) [h g [\overline{h g}]^{- 1}]) b_{\ModulePart} - ([\overline{h}, \overline{g}], [\overline{h} \overline{g}]) b_{\GroupPart} + ([\overline{h}], [\overline{g}]) b_{\GroupPart} \\
& = ([h [\overline{h}]^{- 1}]^{- 1} \, {^h}([g [\overline{g}]^{- 1}]^{- 1}) [h g [\overline{h g}]^{- 1}]) (c \differential)_{\ModulePart} - ([\overline{h}, \overline{g}], [\overline{h} \overline{g}]) (c \differential)_{\GroupPart} + ([\overline{h}], [\overline{g}]) (c \differential)_{\GroupPart} \\
& = ((h [\overline{h}]^{- 1})^{- 1} \, {^h}((g [\overline{g}]^{- 1})^{- 1}) (h g [\overline{h g}]^{- 1})) c - ([\overline{h}, \overline{g}]) c + ([\overline{h}, \overline{g}] [\overline{h} \overline{g}]) c - ([\overline{h} \overline{g}]) c + ([\overline{h}]) c - ([\overline{h}] [\overline{g}]) c \\
& \qquad + \overline{h} \act ([\overline{g}]) c \\
& = - ([\overline{h} \overline{g}]) c + ([\overline{h}]) c + \overline{h} \act ([\overline{g}]) c = (\overline{h}) c_0 - (\overline{h} \overline{g}) c_0 + \overline{h} \act (\overline{g}) c_0 = (\overline{h}, \overline{g}) (c_0 \differential)
\end{align*}
for \(g, h \in \GroupPart E\). \qedhere
\end{enumerate}
\end{proof}

\begin{corollary} \label{cor:second_cohomology_group_computable_by_standard_2-cocycles} \
\begin{enumerate}
\item \label{cor:second_cohomology_group_computable_by_standard_2-cocycles:properties_of_standardised_cocycles} Given a pointed \(2\)-cocycle \(z \in \CocycleGroup[2]_{\text{pt}}(E, M)\), we have
\[([m]) \standardisation{z}_{\ModulePart} = (g [\overline{g}]^{- 1}, [\overline{g}]) \standardisation{z}_{\GroupPart} = 0\]
for \(m \in \ModulePart E\), \(g \in \GroupPart E\).
\item \label{cor:second_cohomology_group_computable_by_standard_2-cocycles:characterisation_of_standard_cocycles} We have
\[\CocycleGroup[2]_{\text{st}}(E, M) = \{z \in \CocycleGroup[2]_{\text{pt}}(E, M) \mid \text{\(([m]) z_{\ModulePart} = (g [\overline{g}]^{- 1}, [\overline{g}]) z_{\GroupPart} = 0\) for all \(m \in \ModulePart E\), \(g \in \GroupPart E\)}\}.\]
In particular, the standardisation \(\standardisation{z}\) of every \(z \in \CocycleGroup[2](E, M)\) is standard.
\item \label{cor:second_cohomology_group_computable_by_standard_2-cocycles:isomorphism_on_cohomology} The embedding \(\CocycleGroup[2]_{\text{st}}(E, M) \map \CocycleGroup[2]_{\text{pt}}(E, M)\) and the standardisation homomorphism \(\CocycleGroup[2]_{\text{pt}}(E, M) \map \CocycleGroup[2]_{\text{st}}(E, M)\), \(z \mapsto \standardisation{z}\) induce mutually inverse isomorphisms between \(\CohomologyGroup[2]_{\text{st}}(E, M)\) and \(\CohomologyGroup[2]_{\text{pt}}(E, M)\). In particular,
\[\CohomologyGroup[2](E, M) \isomorphic \CohomologyGroup[2]_{\text{st}}(E, M).\]
\end{enumerate}
\end{corollary}
\begin{proof} \
\begin{enumerate}
\item We suppose given a pointed \(2\)-cocycle \(z \in \CocycleGroup[2]_{\text{pt}}(E, M)\). Proposition~\ref{prop:module_part_and_group_part_of_standardisation_of_pointed_2-cocycles_and_coboundaries}\ref{prop:module_part_and_group_part_of_standardisation_of_pointed_2-cocycles_and_coboundaries:cocycles} implies
\[([m]) \standardisation{z}_{\ModulePart} = ([m] [m]^{- 1}) z_{\ModulePart} = 0\]
for \(m \in \ModulePart E\) and
\[(g [\overline{g}]^{- 1}, [\overline{g}]) \standardisation{z}_{\GroupPart} = ([g [\overline{g}]^{- 1}]^{- 1} [g [\overline{g}]^{- 1}]) z_{\ModulePart} - ([1, \overline{g}], [\overline{g}]) z_{\GroupPart} + (1, [\overline{g}]) z_{\GroupPart} = 0\]
for \(g \in \GroupPart E\).
\item Given a standard \(2\)-cocycle \(z \in \CocycleGroup[2]_{\text{st}}(E, M)\), we have \(([m]) z_{\ModulePart} = ([m]) \standardisation{z}_{\ModulePart} = 0\) for all \(m \in \ModulePart E\) and \((g [\overline{g}]^{- 1}, [\overline{g}]) z_{\GroupPart} = (g [\overline{g}]^{- 1}, [\overline{g}]) \standardisation{z}_{\GroupPart} = 0\) for all \(g \in \GroupPart E\) by~\ref{cor:second_cohomology_group_computable_by_standard_2-cocycles:properties_of_standardised_cocycles}. Conversely, given a pointed \(2\)-cocycle \(z \in \CocycleGroup[2]_{\text{pt}}(E, M)\) with \(([m]) z_{\ModulePart} = (g [\overline{g}]^{- 1}, [\overline{g}]) z_{\GroupPart} = 0\) for all \(m \in \ModulePart E\), \(g \in \GroupPart E\), it follows that
\[(g) \standardiser{z} = ([g [\overline{g}]^{- 1}]) z_{\ModulePart} - (g [\overline{g}]^{- 1}, [\overline{g}]) z_{\GroupPart} = 0\]
for all \(g \in \GroupPart E\), that is, \(\standardiser{z} = 0\). Hence \(\standardisation{z} = z - \standardiser{z} \differential = z\), that is, \(z\) is standard. Altogether, we have
\[\CocycleGroup[2]_{\text{st}}(E, M) = \{z \in \CocycleGroup[2]_{\text{pt}}(E, M) \mid \text{\(([m]) z_{\ModulePart} = (g [\overline{g}]^{- 1}, [\overline{g}]) z_{\GroupPart} = 0\) for all \(m \in \ModulePart E\), \(g \in \GroupPart E\)}\}\]
and a further application of~\ref{cor:second_cohomology_group_computable_by_standard_2-cocycles:properties_of_standardised_cocycles} shows that \(\standardisation{z} \in \CocycleGroup[2]_{\text{st}}(E, M)\) for all \(z \in \CocycleGroup[2](E, M)\).
\item By definition of the standardisation, we have \(z = \standardisation{z} + \standardiser{z} \differential\) for every pointed \(2\)-cocycle \(z \in \CocycleGroup[2]_{\text{pt}}(E, M)\) and since the standardisation \(\standardisation{z}\) is standard by~\ref{cor:second_cohomology_group_computable_by_standard_2-cocycles:characterisation_of_standard_cocycles}, it follows that
\[\CohomologyGroup[2]_{\text{pt}}(E, M) = \CocycleGroup[2]_{\text{pt}}(E, M) / \CoboundaryGroup[2]_{\text{pt}}(E, M) = (\CocycleGroup[2]_{\text{st}}(E, M) + \CoboundaryGroup[2]_{\text{pt}}(E, M)) / \CoboundaryGroup[2]_{\text{pt}}(E, M).\]
Moreover,
\[\CohomologyGroup[2]_{\text{st}}(E, M) = \CocycleGroup[2]_{\text{st}}(E, M) / \CoboundaryGroup[2]_{\text{st}}(E, M) = \CocycleGroup[2]_{\text{st}}(E, M) / (\CocycleGroup[2]_{\text{st}}(E, M) \intersection \CoboundaryGroup[2]_{\text{pt}}(E, M)),\]
and thus Noether's first law of isomorphism provides the asserted isomorphisms
\begin{align*}
& \CohomologyGroup[2]_{\text{st}}(E, M) \map \CohomologyGroup[2]_{\text{pt}}(E, M), z + \CoboundaryGroup[2]_{\text{st}}(E, M) \mapsto z + \CoboundaryGroup[2]_{\text{pt}}(E, M) \text{ and} \\
& \CohomologyGroup[2]_{\text{pt}}(E, M) \map \CohomologyGroup[2]_{\text{st}}(E, M), z + \CoboundaryGroup[2]_{\text{pt}}(E, M) \mapsto \standardisation{z} + \CoboundaryGroup[2]_{\text{st}}(E, M).
\end{align*}
In particular, we have
\[\CohomologyGroup[2](E, M) \isomorphic \CohomologyGroup[2]_{\text{pt}}(E, M) \isomorphic \CohomologyGroup[2]_{\text{st}}(E, M),\]
cf.\ section~\ref{ssec:pointed_cochains}. \qedhere
\end{enumerate}
\end{proof}

Similarly to proposition~\ref{prop:characterising_properties_of_analysed_2-cocycles}, we will give in proposition~\ref{prop:characterising_properties_of_standard_2-cocycles_and_coboundaries_of_crossed_module_extensions} a characterisation of standard \(2\)-cocycles and \(2\)-coboundaries. For convenience, we introduce the following abbreviation first.

\begin{notation} \label{not:kernel_element}
For \(g, h \in \GroupPart E\), we abbreviate \((h, g) \kappa := [h [\overline{h}]^{- 1}]^{- 1} \, {^h}([g [\overline{g}]^{- 1}]^{- 1}) [h g [\overline{h g}]^{- 1}] [\overline{h}, \overline{g}]^{- 1} \in \Kernel \structuremorphism\).
\end{notation}

\begin{proposition} \label{prop:characterising_properties_of_standard_2-cocycles_and_coboundaries_of_crossed_module_extensions} \
\begin{enumerate}
\item \label{prop:characterising_properties_of_standard_2-cocycles_and_coboundaries_of_crossed_module_extensions:cocycles} A pointed \(2\)-cochain \(z \in \CochainComplex[2]_{\text{pt}}(E, M)\) is a standard \(2\)-cocycle if and only if the following conditions hold:
\begin{enumerate}
\item \label{prop:characterising_properties_of_standard_2-cocycles_and_coboundaries_of_crossed_module_extensions:cocycles:decomposition} We have \((m, h, g) z = (m) z_{\ModulePart} - (m, h) z_{\GroupPart} + (h, g) z_{\GroupPart}\) for \(m \in \ModulePart V\), \(g, h \in \GroupPart V\).
\item \label{prop:characterising_properties_of_standard_2-cocycles_and_coboundaries_of_crossed_module_extensions:cocycles:module_part} We have \((m) z_{\ModulePart} = (m [m]^{- 1}) z_{\ModulePart}\) for \(m \in \ModulePart E\).
\item \label{prop:characterising_properties_of_standard_2-cocycles_and_coboundaries_of_crossed_module_extensions:cocycles:group_part} We have \((h, g) z_{\GroupPart} = ((h, g) \kappa) z_{\ModulePart} + ([\overline{h}], [\overline{g}]) z_{\GroupPart}\)
 for \(g, h \in \GroupPart E\).
\item \label{prop:characterising_properties_of_standard_2-cocycles_and_coboundaries_of_crossed_module_extensions:cocycles:pi1} We have \(\canonicalmonomorphism z_{\ModulePart} \in \Hom_{\Pi_0}(\Pi_1, M)\).
\item \label{prop:characterising_properties_of_standard_2-cocycles_and_coboundaries_of_crossed_module_extensions:cocycles:pi0} We have \(([r, q, p] \canonicalmonomorphism) z_{\ModulePart} = (r, q, p) (((s^0 \cart s^0) z_{\GroupPart}) \differential)\) for \(p, q, r \in \Pi_0\).
\end{enumerate}
\item \label{prop:characterising_properties_of_standard_2-cocycles_and_coboundaries_of_crossed_module_extensions:coboundaries} A pointed \(2\)-cochain \(b \in \CochainComplex[2]_{\text{pt}}(E, M)\) is a standard \(2\)-coboundary if and only if the following conditions hold:
\begin{enumerate}
\item \label{prop:characterising_properties_of_standard_2-cocycles_and_coboundaries_of_crossed_module_extensions:coboundaries:module_part} We have \(b_{\ModulePart} = 0\).
\item \label{prop:characterising_properties_of_standard_2-cocycles_and_coboundaries_of_crossed_module_extensions:coboundaries:group_part} There exists a pointed \(1\)-cochain \(c_0 \in \CochainComplex[1]_{\text{pt}}(\Pi_0, M)\) such that \((h, g) b_{\GroupPart} = (\overline{h}, \overline{g}) (c_0 \differential)\) for \(g, h \in \GroupPart E\).
\end{enumerate}
\end{enumerate}
\end{proposition}
\begin{proof} \
\begin{enumerate}
\item First, we suppose given a standard \(2\)-cocycle \(z \in \CocycleGroup[2]_{\text{st}}(E, M)\). We verify the asserted formulas:
\begin{enumerate}
\item Since \(z\) is in particular a \(2\)-cocycle, this property holds by proposition~\ref{prop:characterising_properties_of_analysed_2-cocycles}\ref{prop:characterising_properties_of_analysed_2-cocycles:crossed_modules}\ref{prop:characterising_properties_of_analysed_2-cocycles:crossed_modules:decomposition}.
\item By corollary~\ref{cor:second_cohomology_group_computable_by_standard_2-cocycles}\ref{cor:second_cohomology_group_computable_by_standard_2-cocycles:characterisation_of_standard_cocycles}, we have
\[(m) z_{\ModulePart} = (m [m]^{- 1} [m]) z_{\ModulePart} = (m [m]^{- 1}) z_{\ModulePart} + ([m])z_{\ModulePart} - (1, m) z_{\GroupPart} = (m [m]^{- 1}) z_{\ModulePart}\]
for \(m \in \ModulePart E\).
\item By proposition~\ref{prop:module_part_and_group_part_of_standardisation_of_pointed_2-cocycles_and_coboundaries}\ref{prop:module_part_and_group_part_of_standardisation_of_pointed_2-cocycles_and_coboundaries:cocycles}, proposition~\ref{prop:characterising_properties_of_analysed_2-cocycles}\ref{prop:characterising_properties_of_analysed_2-cocycles:crossed_modules}\ref{prop:characterising_properties_of_analysed_2-cocycles:crossed_modules:multiplicativity}, corollary~\ref{cor:second_cohomology_group_computable_by_standard_2-cocycles}\ref{cor:second_cohomology_group_computable_by_standard_2-cocycles:characterisation_of_standard_cocycles} and~\ref{prop:characterising_properties_of_standard_2-cocycles_and_coboundaries_of_crossed_module_extensions:cocycles:module_part}, we have
\begin{align*}
& (h, g) z_{\GroupPart} = (h, g) \standardisation{z}_{\GroupPart} \\
& = ([h [\overline{h}]^{- 1}]^{- 1} \, {^h}([g [\overline{g}]^{- 1}]^{- 1}) [h g [\overline{h g}]^{- 1}]) z_{\ModulePart} - ([\overline{h}, \overline{g}], [\overline{h} \overline{g}]) z_{\GroupPart} + ([\overline{h}], [\overline{g}]) z_{\GroupPart} \\
& = (([h [\overline{h}]^{- 1}]^{- 1} \, {^h}([g [\overline{g}]^{- 1}]^{- 1}) [h g [\overline{h g}]^{- 1}] [\overline{h}, \overline{g}]^{- 1}) [\overline{h}, \overline{g}]) z_{\ModulePart} - ([\overline{h}] [\overline{g}] [\overline{h} \overline{g}]^{- 1}, [\overline{h} \overline{g}]) z_{\GroupPart} + ([\overline{h}], [\overline{g}]) z_{\GroupPart} \\
& = ([h [\overline{h}]^{- 1}]^{- 1} \, {^h}([g [\overline{g}]^{- 1}]^{- 1}) [h g [\overline{h g}]^{- 1}] [\overline{h}, \overline{g}]^{- 1}) z_{\ModulePart} + ([\overline{h}], [\overline{g}]) z_{\GroupPart} \\
& = ((h, g) \kappa) z_{\ModulePart} + ([\overline{h}], [\overline{g}]) z_{\GroupPart}
\end{align*}
for \(g, h \in \GroupPart E\).
\item We have \(\canonicalmonomorphism z_{\ModulePart} \in \Hom_{\Pi_0}(\Pi_1, M)\) by proposition~\ref{prop:characterising_properties_of_analysed_2-cocycles}\ref{prop:characterising_properties_of_analysed_2-cocycles:crossed_modules}\ref{prop:characterising_properties_of_analysed_2-cocycles:crossed_modules:multiplicativity} and~\ref{prop:characterising_properties_of_analysed_2-cocycles:crossed_modules:action}.
\item Using proposition~\ref{prop:characterising_properties_of_analysed_2-cocycles}\ref{prop:characterising_properties_of_analysed_2-cocycles:crossed_modules} and corollary~\ref{cor:second_cohomology_group_computable_by_standard_2-cocycles}\ref{cor:second_cohomology_group_computable_by_standard_2-cocycles:characterisation_of_standard_cocycles}, we compute
\begin{align*}
& ([r, q, p] \canonicalmonomorphism) z_{\ModulePart} = ([r, q] [r q, p] [r, q p]^{- 1} ({^{[r]}}[q, p])^{- 1}) z_{\ModulePart} = ([r, q] [r q, p]) z_{\ModulePart} - ({^{[r]}}[q, p] [r, q p]) z_{\ModulePart} \\
& = ([r, q]) z_{\ModulePart} + ([r q, p]) z_{\ModulePart} - ([r, q], [r q, p]) z_{\GroupPart} - ({^{[r]}}[q, p]) z_{\ModulePart} - ([r, q p]) z_{\ModulePart} \\
& \qquad + ({^{[r]}}[q, p], [r, q p]) z_{\GroupPart} \\
& = - ([r, q], [r q, p]) z_{\GroupPart} - r \act ([q, p]) z_{\ModulePart} - ({^{[r]}}[q, p], [r]) z_{\GroupPart} + ([r], [q, p]) z_{\GroupPart} + ({^{[r]}}[q, p], [r, q p]) z_{\GroupPart} \\
& = - ([r, q], [r q] [p] [r q p]^{- 1}) z_{\GroupPart} - ([r] [q, p] [r]^{- 1}, [r]) z_{\GroupPart} + ([r], [q, p]) z_{\GroupPart} \\
& \qquad + ([r] [q, p] [r]^{- 1}, [r] [q p] [r q p]^{- 1}) z_{\GroupPart} \\
& = - ([r, q] [r q], [p] [r q p]^{- 1}) z_{\GroupPart} + ([r q], [p] [r q p]^{- 1}) z_{\GroupPart} - ([r, q], [r q]) z_{\GroupPart} + ([r], [q, p] [r]^{- 1}) z_{\GroupPart} \\
& \qquad - r \act ([q, p] [r]^{- 1}, [r]) z_{\GroupPart} + ([r], [q, p] [q p] [r q p]^{- 1}) z_{\GroupPart} - ([r], [q, p] [r]^{- 1}) z_{\GroupPart} \\
& \qquad + r \act ([q, p] [r]^{- 1}, [r] [q p] [r q p]^{- 1}) z_{\GroupPart} \\
& = - ([r] [q], [p] [r q p]^{- 1}) z_{\GroupPart} + ([r q], [p] [r q p]^{- 1}) z_{\GroupPart} - ([r] [q] [r q]^{- 1}, [r q]) z_{\GroupPart} - r \act ([q, p] [r]^{- 1}, [r]) z_{\GroupPart} \\
& \qquad + ([r], [q] [p] [r q p]^{- 1}) z_{\GroupPart} + r \act ([q, p] [r]^{- 1}, [r] [q p] [r q p]^{- 1}) z_{\GroupPart} \\
& = ([r], [q]) z_{\GroupPart} - r \act ([q], [p] [r q p]^{- 1}) z_{\GroupPart} + ([r q] [p], [r q p]^{- 1}) z_{\GroupPart} - r q \act ([p], [r q p]^{- 1}) z_{\GroupPart} + ([r q], [p]) z_{\GroupPart} \\
& \qquad + r \act ([q, p], [q p] [r q p]^{- 1}) z_{\GroupPart} - ([r], [q p] [r q p]^{- 1}) z_{\GroupPart} \\
& = ([r], [q]) z_{\GroupPart} - r \act ([q] [p], [r q p]^{- 1}) z_{\GroupPart} - r \act ([q], [p]) z_{\GroupPart} + ([r q] [p], [r q p]^{- 1}) z_{\GroupPart} + ([r q], [p]) z_{\GroupPart} \\
& \qquad + r \act ([q, p] [q p], [r q p]^{- 1}) z_{\GroupPart} - r \act ([q p], [r q p]^{- 1}) z_{\GroupPart} + r \act ([q, p], [q p]) z_{\GroupPart} - ([r] [q p], [r q p]^{- 1}) z_{\GroupPart} \\
& \qquad + r \act ([q p], [r q p]^{- 1}) z_{\GroupPart} - ([r], [q p]) z_{\GroupPart} \\
& = ([r], [q]) z_{\GroupPart} - r \act ([q], [p]) z_{\GroupPart} + ([r q, p] [r q p], [r q p]^{- 1}) z_{\GroupPart} + ([r q], [p]) z_{\GroupPart} \\
& \qquad - ([r, q p] [r q p], [r q p]^{- 1}) z_{\GroupPart} - ([r], [q p]) z_{\GroupPart} \\
& = ([r], [q]) z_{\GroupPart} - r \act ([q], [p]) z_{\GroupPart} + ([r q, p], 1) z_{\GroupPart} - ([r q, p], [r q p]) z_{\GroupPart} + ([r q p], [r q p]^{- 1}) z_{\GroupPart} \\
& \qquad + ([r q], [p]) z_{\GroupPart} - ([r, q p], 1) z_{\GroupPart} + ([r, q p], [r q p]) z_{\GroupPart} - ([r q p], [r q p]^{- 1}) z_{\GroupPart} - ([r], [q p]) z_{\GroupPart} \\
& = ([r], [q]) z_{\GroupPart} - ([r], [q p]) z_{\GroupPart} + ([r q], [p]) z_{\GroupPart} - r \act ([q], [p]) z_{\GroupPart} \\
& = (r, q) ((s^0 \cart s^0) z_{\GroupPart}) - (r, q p) ((s^0 \cart s^0) z_{\GroupPart}) + (r q, p) ((s^0 \cart s^0) z_{\GroupPart}) - r \act (q, p) ((s^0 \cart s^0) z_{\GroupPart}) \\
& = (r, q, p) (((s^0 \cart s^0) z_{\GroupPart}) \differential)
\end{align*}
for \(p, q, r \in \Pi_0\).
\end{enumerate}

Conversely, we suppose given a pointed \(2\)-cochain \(z \in \CochainComplex[2]_{\text{pt}}(E, M)\) that fulfills conditions~\ref{prop:characterising_properties_of_standard_2-cocycles_and_coboundaries_of_crossed_module_extensions:cocycles:decomposition} to~\ref{prop:characterising_properties_of_standard_2-cocycles_and_coboundaries_of_crossed_module_extensions:cocycles:pi0}. To show that \(z\) is a \(2\)-cocycle, we use the characterisation given in proposition~\ref{prop:characterising_properties_of_analysed_2-cocycles}\ref{prop:characterising_properties_of_analysed_2-cocycles:crossed_modules}. First of all, we show that \(z_{\GroupPart} \in \CocycleGroup[2](\GroupPart E, M)\). Indeed, we have
\begin{align*}
& (k, h) \kappa (k h, g) \kappa ((k, h g) \kappa)^{- 1} ({^{[\overline{k}]}}((h, g) \kappa))^{- 1} [\overline{k}, \overline{h}, \overline{g}] \canonicalmonomorphism \\
& = ((k, h) \kappa [\overline{k}, \overline{h}]) ((k h, g) \kappa [\overline{k} \, \overline{h}, \overline{g}]) ((k, h g) \kappa [\overline{k}, \overline{h} \, \overline{g}])^{- 1} \, ({^{[\overline{k}]}}((h, g) \kappa [\overline{h}, \overline{g}]))^{- 1} \\
& = ({^{[\overline{k}]}}([h [\overline{h}]^{- 1}]^{- 1}) [k [\overline{k}]^{- 1}]^{- 1} [k h [\overline{k h}]^{- 1}]) ({^{[\overline{k h}]}}([g [\overline{g}]^{- 1}]^{- 1}) [k h [\overline{k h}]^{- 1}]^{- 1} [k h g [\overline{k h g}]^{- 1}]) \\
& \qquad ({^{[\overline{k}]}}([h g [\overline{h g}]^{- 1}]^{- 1}) [k [\overline{k}]^{- 1}]^{- 1} [k h g [\overline{k h g}]^{- 1}])^{- 1} ({^{[\overline{k}][\overline{h}]}}([g [\overline{g}]^{- 1}]^{- 1}) \, {^{[\overline{k}]}}[h [\overline{h}]^{- 1}]^{- 1} \, {^{[\overline{k}]}}[h g [\overline{h g}]^{- 1}])^{- 1} \\
& = {^{[\overline{k}]}}([h [\overline{h}]^{- 1}]^{- 1}) [k [\overline{k}]^{- 1}]^{- 1} [k h [\overline{k h}]^{- 1}] \, {^{[\overline{k h}]}}([g [\overline{g}]^{- 1}]^{- 1}) [k h [\overline{k h}]^{- 1}]^{- 1} [k h g [\overline{k h g}]^{- 1}] [k h g [\overline{k h g}]^{- 1}]^{- 1} [k [\overline{k}]^{- 1}] \\
& \qquad {^{[\overline{k}]}}[h g [\overline{h g}]^{- 1}] ({^{[\overline{k}]}}[h g [\overline{h g}]^{- 1}])^{- 1} \, {^{[\overline{k}]}}[h [\overline{h}]^{- 1}] \, {^{[\overline{k}][\overline{h}]}}[g [\overline{g}]^{- 1}] \\
& = {^{[\overline{k}]}}([h [\overline{h}]^{- 1}]^{- 1}) [k [\overline{k}]^{- 1}]^{- 1} [k h [\overline{k h}]^{- 1}] \, {^{[\overline{k h}]}}([g [\overline{g}]^{- 1}]^{- 1}) [k h [\overline{k h}]^{- 1}]^{- 1} [k [\overline{k}]^{- 1}] \, {^{[\overline{k}]}}[h [\overline{h}]^{- 1}] \, {^{[\overline{k}][\overline{h}]}}[g [\overline{g}]^{- 1}] \\
& = {^{[\overline{k}]}}([h [\overline{h}]^{- 1}]^{- 1}) \, {^{[k [\overline{k}]^{- 1}]^{- 1} [k h [\overline{k h}]^{- 1}] [\overline{k h}]}}([g [\overline{g}]^{- 1}]^{- 1}) \, {^{[\overline{k}]}}[h [\overline{h}]^{- 1}] \, {^{[\overline{k}][\overline{h}]}}[g [\overline{g}]^{- 1}] \\
& = {^{[\overline{k}]}}([h [\overline{h}]^{- 1}]^{- 1}) \, {^{[\overline{k}] h}}([g [\overline{g}]^{- 1}]^{- 1}) \, {^{[\overline{k}]}}[h [\overline{h}]^{- 1}] \, {^{[\overline{k}][\overline{h}]}}[g [\overline{g}]^{- 1}] \\
& = {^{[\overline{k}]}}([h [\overline{h}]^{- 1}]^{- 1} \, {^h}([g [\overline{g}]^{- 1}]^{- 1}) [h [\overline{h}]^{- 1}] \, {^{[\overline{h}]}}[g [\overline{g}]^{- 1}]) = {^{[\overline{k}]}}( {^{[h [\overline{h}]^{- 1}]^{- 1} h}}([g [\overline{g}]^{- 1}]^{- 1}) \, {^{[\overline{h}]}}[g [\overline{g}]^{- 1}]) \\
& = {^{[\overline{k}]}}( {^{[\overline{h}]}}([g [\overline{g}]^{- 1}]^{- 1}) \, {^{[\overline{h}]}}[g [\overline{g}]^{- 1}]) = 1
\end{align*}
and hence
\begin{align*}
& (k, h, g) (z_{\GroupPart} \differential) = (k, h) z_{\GroupPart} - (k, h g) z_{\GroupPart} + (k h, g) z_{\GroupPart} - \overline{k} \act (h, g) z_{\GroupPart} \\
& = ((k, h) \kappa) z_{\ModulePart} + ([\overline{k}], [\overline{h}]) z_{\GroupPart} - ((k, h g) \kappa) z_{\ModulePart} - ([\overline{k}], [\overline{h g}]) z_{\GroupPart} + ((k h, g) \kappa) z_{\ModulePart} + ([\overline{k h}], [\overline{g}]) z_{\GroupPart} \\
& \qquad - \overline{k} \act ((h, g) \kappa) z_{\ModulePart} - \overline{k} \act ([\overline{h}], [\overline{g}]) z_{\GroupPart} \\
& = ((k, h) \kappa (k h, g) \kappa ((k, h g) \kappa)^{- 1} \, {^{[\overline{k}]}}(((h, g) \kappa)^{- 1})) z_{\ModulePart} + (\overline{k}, \overline{h}, \overline{g}) (((s^0 \cart s^0) z_{\GroupPart}) \differential) \\
& = ((k, h) \kappa (k h, g) \kappa ((k, h g) \kappa)^{- 1} \, ({^{[\overline{k}]}}((h, g) \kappa))^{- 1} [\overline{k}, \overline{h}, \overline{g}] \canonicalmonomorphism) z_{\ModulePart} = 0
\end{align*}
for \(g, h, k \in \GroupPart E\), that is, \(z_{\GroupPart} \in \CocycleGroup[2]_{\text{pt}}(\GroupPart E, M)\). Moreover, we have
\begin{align*}
& (n m) z_{\ModulePart} - (n) z_{\ModulePart} - (m) z_{\ModulePart} + (n, m) z_{\GroupPart} \\
& = (n m [n m]^{- 1}) z_{\ModulePart} - (n [n]^{- 1}) z_{\ModulePart} - (m [m]^{- 1}) z_{\ModulePart} + ((n, m) \kappa) z_{\ModulePart} \\
& = ((n m [n m]^{- 1}) (n [n]^{- 1})^{- 1} (m [m]^{- 1})^{- 1} (n, m) \kappa) z_{\ModulePart} \\
& = ((n [n]^{- 1})^{- 1} n (m [m]^{- 1})^{- 1} m (n, m) \kappa [n m]^{- 1}) z_{\ModulePart} \\
& = ([n] n^{- 1} n [m] m^{- 1} m [n]^{- 1} \, {^n}([m]^{- 1}) [n m] [n m]^{- 1}) z_{\ModulePart} \\
& = ([n] [m] [n]^{- 1} \, {^n}([m]^{- 1})) z_{\ModulePart} = ({^{[n]}}[m] \, {^n}([m]^{- 1})) z_{\ModulePart} = 0
\end{align*}
for \(m, n \in \ModulePart E\) and
\begin{align*}
& ({^g}m) z_{\ModulePart} - \overline{g} \act (m) z_{\ModulePart} - ({^g}m, g) z_{\GroupPart} + (g, m) z_{\GroupPart} \\
& = ({^g}m [{^g}m]^{- 1}) z_{\ModulePart} - \overline{g} \act (m [m]^{- 1}) z_{\ModulePart} - (({^g}m, g) \kappa) z_{\ModulePart} + ((g, m) \kappa) z_{\ModulePart} \\
& = (({^g}m [{^g}m]^{- 1}) \, {^g}((m [m]^{- 1})^{- 1}) (({^g}m, g) \kappa)^{- 1} (g, m) \kappa) z_{\ModulePart} \\
& = (({^g}[m] \, {^g}(m^{- 1})) ({^g}m [{^g}m]^{- 1}) (({^g}m, g) \kappa)^{- 1} (g, m) \kappa) z_{\ModulePart} \\
& = ({^g}[m] (g, m) \kappa (({^g}m, g) \kappa)^{- 1} [{^g}m]^{- 1}) z_{\ModulePart} \\
& = ({^g}[m] ([g [\overline{g}]^{- 1}]^{- 1} \, {^g}([m]^{- 1}) [g m [\overline{g}]^{- 1}]) ([{^g}m]^{- 1} \, {^{{^g}m}}([g [\overline{g}]^{- 1}]^{- 1}) [{^g}m g [\overline{g}]^{- 1}])^{- 1} [{^g}m]^{- 1}) z_{\ModulePart} \\
& = ({^g}[m] [g [\overline{g}]^{- 1}]^{- 1} \, {^g}([m]^{- 1}) [g m [\overline{g}]^{- 1}] [{^g}m g [\overline{g}]^{- 1}]^{- 1} \, {^{{^g}m}}[g [\overline{g}]^{- 1}] [{^g}m] [{^g}m]^{- 1}) z_{\ModulePart} \\
& = ({^g}[m] [g [\overline{g}]^{- 1}]^{- 1} \, {^g}([m]^{- 1}) \, {^{{^g}m}}[g [\overline{g}]^{- 1}]) z_{\ModulePart} = ({^{{^g}[m]}}([g [\overline{g}]^{- 1}]^{- 1}) \, {^{{^g}m}}[g [\overline{g}]^{- 1}]) z_{\ModulePart} = 0
\end{align*}
for \(m \in \ModulePart E\) and \(g \in \GroupPart E\). Altogether, \(z \in \CocycleGroup[2]_{\text{pt}}(E, M)\). Finally, we have
\[([m]) z_{\ModulePart} = ([m] [m]^{- 1}) z_{\ModulePart} = 0\]
for \(m \in \ModulePart E\) and
\[(g [\overline{g}]^{- 1}, [\overline{g}]) z_{\GroupPart} = ((g [\overline{g}]^{- 1}, [\overline{g}]) \kappa) z_{\ModulePart} + (1, [\overline{g}]) z_{\GroupPart} = ([g [\overline{g}]^{- 1}]^{- 1} [g [\overline{g}]^{- 1}]) z_{\ModulePart} = 0\]
for \(g \in \GroupPart E\). Hence \(z \in \CocycleGroup[2]_{\text{st}}(E, M)\) by corollary~\ref{cor:second_cohomology_group_computable_by_standard_2-cocycles}\ref{cor:second_cohomology_group_computable_by_standard_2-cocycles:characterisation_of_standard_cocycles}.
\item We suppose given a standard \(2\)-coboundary \(b \in \CoboundaryGroup[2]_{\text{st}}(E, M)\) and we choose \(c \in \CochainComplex[1]_{\text{pt}}(E, M)\) such that \(b = c \differential\). Letting \(c_0 \in \CochainComplex[1]_{\text{pt}}(\Pi_0, M)\) be defined by \((p) c_0 := ([p]) c\), proposition~\ref{prop:module_part_and_group_part_of_standardisation_of_pointed_2-cocycles_and_coboundaries}\ref{prop:module_part_and_group_part_of_standardisation_of_pointed_2-cocycles_and_coboundaries:coboundaries} implies that \((m) b_{\ModulePart} = (m) \standardisation{b}_{\ModulePart} = 0\) for \(m \in \ModulePart E\) and \((h, g) b_{\GroupPart} = (h, g) \standardisation{b}_{\GroupPart} = (\overline{h}, \overline{g}) (c_0 \differential)\) for \(g, h \in \GroupPart E\).

Conversely, let us suppose that \(b_{\ModulePart} = 0\) and suppose given a pointed \(1\)-cochain \(c_0\in \CochainComplex[1]_{\text{pt}}(\Pi_0, M)\) with \((h, g) b_{\GroupPart} = (\overline{h}, \overline{g}) (c_0 \differential)\) for \(g, h \in \GroupPart E\). Defining \(c \in \CochainComplex[1]_{\text{pt}}(E, M)\) by \((g) c := (\overline{g}) c_0\) for \(g \in \GroupPart E\), we have
\[(m) (c \differential)_{\ModulePart} = (m) c = (\overline{m}) c_0 = 0\]
for \(m \in \ModulePart E\) and
\[(h, g) (c \differential)_{\GroupPart} = (h) c - (h g) c + \overline{h} \act (g) c = (\overline{h}) c_0 - (\overline{h g}) c_0 + \overline{h} \act (\overline{g}) c_0 = (\overline{h}, \overline{g}) (c_0 \differential)\]
for \(g, h \in \GroupPart E\), that is, \(c \differential = b\). Moreover, \(([m]) b_{\ModulePart} = 0\) for all \(m \in \ModulePart E\) and \((g [\overline{g}]^{- 1}, [\overline{g}]) b_{\GroupPart} = (1, \overline{g}) (c_0 \differential) = 0\) for all \(g \in \GroupPart E\). Hence \(b \in \CocycleGroup[2]_{\text{st}}(E, M) \intersection \CoboundaryGroup[2]_{\text{pt}}(E, M) = \CoboundaryGroup[2]_{\text{st}}(E, M)\) by corollary~\ref{cor:second_cohomology_group_computable_by_standard_2-cocycles}\ref{cor:second_cohomology_group_computable_by_standard_2-cocycles:characterisation_of_standard_cocycles}. \qedhere
\end{enumerate}
\end{proof}

\begin{definition}[cocycle, coboundary and cohomology group of a \(3\)-cocycle] \label{def:cocycle_coboundary_and_cohomology_group_of_a_3-cocycle}
For a \(3\)-cocycle \(z^3 \in \CocycleGroup[3](\Pi_0, \Pi_1)\), we set 
\begin{align*}
& \CocycleGroup[2]((\Pi_0, \Pi_1, z^3), M) := \Hom_{\Pi_0}(\Pi_1, M) \fibreprod{\Map(z^3, M)|_{\Hom_{\Pi_0}(\Pi_1, M)}}{\differential} \CochainComplex[2]_{\text{cpt}}(\Pi_0, M), \\
& \CoboundaryGroup[2]((\Pi_0, \Pi_1, z^3), M) := \{0\} \directprod \CoboundaryGroup[2]_{\text{cpt}}(\Pi_0, M) \text{ and} \\
& \CohomologyGroup[2]((\Pi_0, \Pi_1, z^3), M) := \CocycleGroup[2]((\Pi_0, \Pi_1, z^3), M) / \CoboundaryGroup[2]((\Pi_0, \Pi_1, z^3), M).
\end{align*}
\end{definition}

\begin{corollary} \label{cor:second_cohomology_group_of_crossed_module_extension_via_3-cocycle}
We have group homomorphisms \(\Phi_1\colon \CocycleGroup[2]_{\text{st}}(E, M) \map \Hom_{\Pi_0}(\Pi_1, M)\) and \(\Phi_0\colon \CocycleGroup[2]_{\text{st}}(E, M) \map \CochainComplex[2]_{\text{pt}}(\Pi_0, M)\) given by \((k) (z \Phi_1) := (k \canonicalmonomorphism) z_{\ModulePart}\) for \(k \in \Pi_1\) and \((q, p) (z \Phi_0) := ([q], [p]) z_{\GroupPart}\) for \(p, q \in \Pi_0\), \(z \in \CocycleGroup[2]_{\text{st}}(E, M)\). These group homomorphisms fit into the following diagram, which is a pullback of abelian groups.
\[\begin{tikzpicture}[baseline=(m-2-1.base)]
  \matrix (m) [matrix of math nodes, row sep=2.5em, column sep=5.0em, column 1/.style={anchor=base east}, column 2/.style={anchor=base west}, text height=1.6ex, text depth=0.45ex, inner sep=0pt, nodes={inner sep=0.333em}]{
    \CocycleGroup[2]_{\text{st}}(E, M) & \Hom_{\Pi_0}(\Pi_1, M) \\
    \CochainComplex[2]_{\text{cpt}}(\Pi_0, M) & \CochainComplex[3]_{\text{cpt}}(\Pi_0, M) \\};
  \path[->, font=\scriptsize] let \p1=(1.25em, 0) in
    (m-1-1) edge node[above] {\(\Phi_1\)} (m-1-2)
    ($(m-1-1.south east)-(\p1)$) edge node[left] {\(\Phi_0\)} ($(m-2-1.north east)-(\p1)$)
    ($(m-1-2.south west)+(\p1)$) edge node[right] {\(\Map(\cocycleofextension{3}, M)|_{\Hom_{\Pi_0}(\Pi_1, M)}\)} ($(m-2-2.north west)+(\p1)$)
    (m-2-1) edge node[above] {\(\differential\)} (m-2-2);
\end{tikzpicture}\]
The induced isomorphism
\[\Phi\colon \CocycleGroup[2]_{\text{st}}(E, M) \map \CocycleGroup[2]((\Pi_0, \Pi_1, \cocycleofextension{3}), M), z \mapsto (z \Phi_1, z \Phi_0),\]
whose inverse
\[\Psi\colon \CocycleGroup[2]((\Pi_0, \Pi_1, \cocycleofextension{3}), M) \map \CocycleGroup[2]_{\text{st}}(E, M)\]
is given by \((m, h, g) ((z_1, c_0) \Psi) = ((m [m]^{- 1} ((m, h) \kappa)^{- 1} (h, g) \kappa) (\canonicalmonomorphism|^{\Image \canonicalmonomorphism})^{- 1}) z_1 + (\overline{h}, \overline{g}) c_0\) for \(m \in \ModulePart E\), \(g, h \in \GroupPart E\), induces in turn isomorphisms \(\CoboundaryGroup[2]_{\text{st}}(E, M) \map \CoboundaryGroup[2]((\Pi_0, \Pi_1, \cocycleofextension{3}), M)\) and \(\CohomologyGroup[2]_{\text{st}}(E, M) \map \CohomologyGroup[2]((\Pi_0, \Pi_1, \cocycleofextension{3}), M)\). In particular, we have
\[\CohomologyGroup[2](E, M) \isomorphic \CohomologyGroup[2]((\Pi_0, \Pi_1, \cocycleofextension{3}), M).\]
\end{corollary}
\begin{proof}
By proposition~\ref{prop:characterising_properties_of_standard_2-cocycles_and_coboundaries_of_crossed_module_extensions}\ref{prop:characterising_properties_of_standard_2-cocycles_and_coboundaries_of_crossed_module_extensions:cocycles}\ref{prop:characterising_properties_of_standard_2-cocycles_and_coboundaries_of_crossed_module_extensions:cocycles:pi1} and~\ref{prop:characterising_properties_of_standard_2-cocycles_and_coboundaries_of_crossed_module_extensions:cocycles:pi0}, the group homomorphisms \(\Phi_0\) and \(\Phi_1\) are well-defined and the quadrangle commutes. To show that it is a pullback of abelian groups, we suppose given an arbitrary abelian group \(T\) as well as group homomorphisms \(\varphi_0\colon T \map \CochainComplex[2]_{\text{cpt}}(\Pi_0, M)\) and \(\varphi_1\colon T \map \Hom_{\Pi_0}(\Pi_1, M)\) such that \(\varphi_1 \Map(\cocycleofextension{3}, M)|_{\Hom_{\Pi_0}(\Pi_1, M)} = \varphi_0 \differential\), that is, with \(([r, q, p]) (t \varphi_1) = (r, q, p) ((t \varphi_0) \differential)\) for all \(p, q, r \in \Pi_0\), \(t \in T\). For \(t \in T\), we define a pointed \(2\)-cochain \(t \varphi \in \CochainComplex[2]_{\text{pt}}(E, M)\) by
\[(m, h, g) (t \varphi) := ((m [m]^{- 1} ((m, h) \kappa)^{- 1} (h, g) \kappa) (\canonicalmonomorphism|^{\Image \canonicalmonomorphism})^{- 1}) (t \varphi_1) + (\overline{h}, \overline{g}) (t \varphi_0)\]
for \(m \in \ModulePart E\), \(g, h \in \GroupPart E\). Then we obtain \((m) (t \varphi)_{\ModulePart} = ((m [m]^{- 1}) (\canonicalmonomorphism|^{\Image \canonicalmonomorphism})^{- 1}) (t \varphi_1)\) for \(m \in \ModulePart E\) and \((h, g) (t \varphi)_{\GroupPart} = ((h, g) \kappa (\canonicalmonomorphism|^{\Image \canonicalmonomorphism})^{- 1}) (t \varphi_1) + (\overline{h}, \overline{g}) (t \varphi_0)\) for \(g, h \in \GroupPart E\). To show that \(t \varphi\) is a standard \(2\)-cocycle, we verify the conditions in proposition~\ref{prop:characterising_properties_of_standard_2-cocycles_and_coboundaries_of_crossed_module_extensions}\ref{prop:characterising_properties_of_standard_2-cocycles_and_coboundaries_of_crossed_module_extensions:cocycles}. Indeed, using \([m [m]^{- 1}] = ([\overline{h}], [\overline{g}]) \kappa = (1, h) \kappa = 1\) for \(m \in \ModulePart E\), \(g, h \in \GroupPart E\), we have
\begin{align*}
& (m, h, g) (t \varphi) = ((m [m]^{- 1} ((m, h) \kappa)^{- 1} (h, g) \kappa) (\canonicalmonomorphism|^{\Image \canonicalmonomorphism})^{- 1}) (t \varphi_1) + (\overline{h}, \overline{g}) (t \varphi_0) \\
& = ((m [m]^{- 1}) (\canonicalmonomorphism|^{\Image \canonicalmonomorphism})^{- 1}) (t \varphi_1) - ((m, h) \kappa (\canonicalmonomorphism|^{\Image \canonicalmonomorphism})^{- 1}) (t \varphi_1) + ((h, g) \kappa (\canonicalmonomorphism|^{\Image \canonicalmonomorphism})^{- 1}) (t \varphi_1) + (\overline{h}, \overline{g}) (t \varphi_0) \\
& = (m) (t \varphi)_{\ModulePart} - (m, h) (t \varphi)_{\GroupPart} + (h, g) (t \varphi)_{\GroupPart}
\end{align*}
since \(t \varphi_1\) is componentwise pointed as well as
\[(m) (t \varphi)_{\ModulePart} = (m [m]^{- 1}) (t \varphi_1) = (m [m]^{- 1}) (t \varphi)_{\ModulePart}\]
and
\[(h, g) (t \varphi)_{\GroupPart} = ((h, g) \kappa (\canonicalmonomorphism|^{\Image \canonicalmonomorphism})^{- 1}) (t \varphi_1) + (\overline{h}, \overline{g}) (t \varphi_0) = ((h, g) \kappa) (t \varphi)_{\ModulePart} + ([\overline{h}], [\overline{g}]) (t \varphi)_{\GroupPart}\]
for \(m \in \ModulePart E\), \(g, h \in \GroupPart E\). Moreover, \(\canonicalmonomorphism (t \varphi)_{\ModulePart} = t \varphi_1 \in \Hom_{\Pi_0}(\Pi_1, M)\) and
\[([r, q, p]) (t \varphi)_{\ModulePart} = ([r, q, p] \canonicalmonomorphism) (t \varphi_1) = (r, q, p) ((t \varphi_0) \differential) = (r, q, p) (((s^0 \cart s^0) (t \varphi)_{\GroupPart}) \differential)\]
for \(p, q, r \in \Pi_0\). Altogether, \(t \varphi \in \CocycleGroup[2]_{\text{st}}(E, M)\) for all \(t \in T\), and we have constructed a well-defined group homomorphism \(\varphi\colon T \map \CocycleGroup[2]_{\text{st}}(E, M)\). Finally, we have
\[(k) ((t \varphi) \Phi_1) = (k \canonicalmonomorphism) (t \varphi)_{\ModulePart} = (k) (t \varphi_1)\]
for \(k \in \Pi_1\), \(t \in T\), and
\[(q, p) ((t \varphi) \Phi_0) = ([q], [p]) (t \varphi)_{\GroupPart} = (([q], [p]) \kappa (\canonicalmonomorphism|^{\Image \canonicalmonomorphism})^{- 1}) (t \varphi_1) + (q, p) (t \varphi_0) = (q, p) (t \varphi_0)\]
for \(p, q \in \Pi_0\), \(t \in T\), that is, \(\varphi \Phi_1 = \varphi_1\) and \(\varphi \Phi_0 = \varphi_0\).

Conversely, given an arbitrary group homomorphism \(\varphi\colon T \map \CocycleGroup[2]_{\text{st}}(E, M)\) with \(\varphi \Phi_1 = \varphi_1\) and \(\varphi \Phi_0 = \varphi_0\), we necessarily have
\[(m) (t \varphi)_{\ModulePart} = (m [m]^{- 1}) (t \varphi)_{\ModulePart} = ((m [m]^{- 1}) (\canonicalmonomorphism|^{\Image \canonicalmonomorphism})^{- 1}) (t \varphi \Phi_1) = ((m [m]^{- 1}) (\canonicalmonomorphism|^{\Image \canonicalmonomorphism})^{- 1}) (t \varphi_1)\]
for \(m \in \ModulePart E\), and
\begin{align*}
(h, g) (t \varphi)_{\GroupPart} & = ((h, g) \kappa) (t \varphi)_{\ModulePart} + ([\overline{h}], [\overline{g}]) (t \varphi)_{\GroupPart} = ((h, g) \kappa) (t \varphi)_{\ModulePart} + (\overline{h}, \overline{g}) (t \varphi \Phi_0) \\
& = ((h, g) \kappa (\canonicalmonomorphism|^{\Image \canonicalmonomorphism})^{- 1}) (t \varphi_1) + (\overline{h}, \overline{g}) (t \varphi_0)
\end{align*}
for \(g, h \in \GroupPart E\). This shows the uniqueness of the induced homomorphism. Altogether, the diagram under consideration is a pullback of abelian groups.

Our next step is to show that the induced isomorphism \(\Phi\colon \CocycleGroup[2]_{\text{st}}(E, M) \map \CocycleGroup[2]((\Pi_0, \Pi_1, \cocycleofextension{3}), M)\) restricts to an isomorphism \(\CoboundaryGroup[2]_{\text{st}}(E, M) \map \CoboundaryGroup[2]((\Pi_0, \Pi_1, \cocycleofextension{3}), M)\). Given a standard \(2\)-coboundary \(b \in \CoboundaryGroup[2]_{\text{st}}(E, M)\), proposition~\ref{prop:characterising_properties_of_standard_2-cocycles_and_coboundaries_of_crossed_module_extensions}\ref{prop:characterising_properties_of_standard_2-cocycles_and_coboundaries_of_crossed_module_extensions:coboundaries} states that \(b_{\ModulePart} = 0\) and that there exists a pointed \(1\)-cochain \(c_0 \in \CochainComplex[1]_{\text{pt}}(\Pi_0, M)\) with \((h, g) b_{\GroupPart} = (\overline{h}, \overline{g}) (c_0 \differential)\) for \(g, h \in \GroupPart E\). In particular, \(b \Phi_1 = 0\) and
\[(q, p) (b \Phi_0) = ([q], [p]) b_{\GroupPart} = (q, p) (c_0 \differential)\]
for \(p, q \in \Pi_0\) and hence \(b \Phi_0 \in \CoboundaryGroup[2](\Pi_0, M)\). Conversely, we suppose given a standard \(2\)-cocycle \(b \in \CocycleGroup[2]_{\text{st}}(E, M)\) with \(b \Phi_1 = 0\) and \(b \Phi_0 \in \CoboundaryGroup[2]_{\text{pt}}(\Pi_0, M)\), that is, there exists a pointed \(1\)-cochain \(c_0 \in \CochainComplex[1]_{\text{pt}}(\Pi_0, M)\) with \(b \Phi_0 = c_0 \differential\). Then
\[(m) b_{\ModulePart} = (m [m]^{- 1}) b_{\ModulePart} = ((m [m]^{- 1}) (\canonicalmonomorphism|^{\Image \canonicalmonomorphism})^{- 1}) (b \Phi_1) = 0\]
for all \(m \in \ModulePart E\) and
\[(h, g) b_{\GroupPart} = ((h, g) \kappa) b_{\ModulePart} + ([\overline{h}], [\overline{g}]) b_{\GroupPart} = (\overline{h}, \overline{g}) (b \Phi_0) = (\overline{h}, \overline{g}) (c_0 \differential)\]
for all \(g, h \in \GroupPart E\). Hence \(b\) is a standard \(2\)-coboundary by proposition~\ref{prop:characterising_properties_of_standard_2-cocycles_and_coboundaries_of_crossed_module_extensions}\ref{prop:characterising_properties_of_standard_2-cocycles_and_coboundaries_of_crossed_module_extensions:coboundaries}.

Altogether, \(\Phi\) restricts to an isomorphism \(\CoboundaryGroup[2]_{\text{st}}(E, M) \map \CoboundaryGroup[2]((\Pi_0, \Pi_1, \cocycleofextension{3}), M)\) and hence induces also an isomorphism \(\CohomologyGroup[2]_{\text{st}}(E, M) \map \CohomologyGroup[2]((\Pi_0, \Pi_1, \cocycleofextension{3}), M)\). Moreover, corollary~\ref{cor:second_cohomology_group_computable_by_standard_2-cocycles}\ref{cor:second_cohomology_group_computable_by_standard_2-cocycles:isomorphism_on_cohomology} implies that
\[\CohomologyGroup[2](E, M) \isomorphic \CohomologyGroup[2]_{\text{st}}(E, M) \isomorphic \CohomologyGroup[2]((\Pi_0, \Pi_1, \cocycleofextension{3}), M). \qedhere\]
\end{proof}

\begin{corollary} \label{cor:cohomology_groups_of_cohomologous_3-cocycles_are_isomorphic}
For \(z^3, \tilde z^3 \in \CocycleGroup[3]_{\text{cpt}}(\Pi_0, \Pi_1)\) with \(z^3 \CoboundaryGroup[3]_{\text{cpt}}(\Pi_0, \Pi_1) = \tilde z^3 \CoboundaryGroup[3]_{\text{cpt}}(\Pi_0, \Pi_1)\), we have
\[\CohomologyGroup[2]((\Pi_0, \Pi_1, z^3), M) \isomorphic \CohomologyGroup[2]((\Pi_0, \Pi_1, \tilde z^3), M).\]
\end{corollary}
\begin{proof}
We suppose given \(3\)-cocycles \(z^3, \tilde z^3 \in \CocycleGroup[3]_{\text{cpt}}(\Pi_0, \Pi_1)\) with \(z^3 \CoboundaryGroup[3]_{\text{cpt}}(\Pi_0, \Pi_1) = \tilde z^3 \CoboundaryGroup[3]_{\text{cpt}}(\Pi_0, \Pi_1)\). By construction of the standard extension \(\StandardExtension(z^3)\), the \(3\)-cocycle of the standard extension \(\StandardExtension(z^3)\) with respect to the standard section system \((\standardsectionsystem_{z^3}^1, \standardsectionsystem_{z^3}^0)\) is given by \(\cocycleofextension{3}_{\StandardExtension(z^3), (\standardsectionsystem_{z^3}^1, \standardsectionsystem_{z^3}^0)} = z^3\), cf.\ section~\ref{ssec:crossed_module_extensions}. Moreover, by~\cite[prop.~(6.5)]{thomas:2009:the_third_cohomology_group_classifies_crossed_module_extensions} there exists a section system \((\tilde s^1, \standardsectionsystem_{z^3}^0)\) for \(\StandardExtension(z^3)\) such that \(\cocycleofextension{3}_{\StandardExtension(z^3), (\tilde s^1, \standardsectionsystem_{z^3}^0)} = \tilde z^3\). Thus corollary~\ref{cor:second_cohomology_group_of_crossed_module_extension_via_3-cocycle} implies
\[\CohomologyGroup[2]((\Pi_0, \Pi_1, z^3), M) \isomorphic \CohomologyGroup[2](\StandardExtension(z^3), M) \isomorphic \CohomologyGroup[2]((\Pi_0, \Pi_1, \tilde z^3), M). \qedhere\]
\end{proof}

We finish this section by a direct algebraic proof that extension equivalent crossed module extensions yield the same second cohomology group, as to be expected from a weak homotopy equivalence, cf.\ for example~\cite[rem.~(4.5)]{thomas:2009:the_third_cohomology_group_classifies_crossed_module_extensions}.

\begin{proposition} \label{prop:compatible_choices_of_section_systems_along_an_extension_equivalence_lead_to_isomorphic_standard_cohomology_groups}
We suppose given crossed module extensions \(E\) and \(\tilde E\) of \(\Pi_0\) with \(\Pi_1\) and an extension equivalence \(\varphi\colon E \map \tilde E\). Moreover, we suppose given a section system \((s^1, s^0)\) for \(E\) and a section system \((\tilde s^1, \tilde s^0)\) for \(\tilde E\) such that \(\tilde s^0 = s^0 (\GroupPart \varphi)\) and \(s^1 (\ModulePart \varphi) = (\GroupPart \varphi)|_{\Image \structuremorphism[E]}^{\Image \structuremorphism[\tilde E]} \tilde s^1\). (\footnote{Such section systems exist, cf.\ for example~\cite[prop.~(5.16)(b)]{thomas:2009:the_third_cohomology_group_classifies_crossed_module_extensions}.})

The induced group homomorphism \(\CocycleGroup[2](\varphi, M)\colon \CocycleGroup[2](\tilde E, M) \map \CocycleGroup[2](E, M)\) restricts to a group isomorphism \linebreak % manual format
\(\CocycleGroup[2]_{\text{st}, (\tilde s^1, \tilde s^0)}(\tilde E, M) \map \CocycleGroup[2]_{\text{st}, (s^1, s^0)}(E, M)\), which induces in turn isomorphisms \(\CoboundaryGroup[2]_{\text{st}, (\tilde s^1, \tilde s^0)}(\tilde E, M) \map \CoboundaryGroup[2]_{\text{st}, (s^1, s^0)}(E, M)\) and \(\CohomologyGroup[2]_{\text{st}, (\tilde s^1, \tilde s^0)}(\tilde E, M) \map \CohomologyGroup[2]_{\text{st}, (s^1, s^0)}(E, M)\).
\end{proposition}
\begin{proof}
To show that \(\CocycleGroup[2](\varphi, M)\) restricts to a homomorphism \(\CocycleGroup[2]_{\text{st}}(\tilde E, M) \map \CocycleGroup[2]_{\text{st}}(E, M)\), we have to show that \(\tilde z \CocycleGroup[2](\varphi, M) \in \CocycleGroup[2]_{\text{st}}(E, M)\) for every given standard \(2\)-cocycle \(\tilde z \in \CocycleGroup[2]_{\text{st}}(\tilde E, M)\). By corollary~\ref{cor:second_cohomology_group_computable_by_standard_2-cocycles}\ref{cor:second_cohomology_group_computable_by_standard_2-cocycles:characterisation_of_standard_cocycles}, we have
\((\tilde m \tilde s^1) \tilde z_{\ModulePart} = (\tilde g (\tilde g \canonicalepimorphism[\tilde E] \tilde s^0)^{- 1}, \tilde g \canonicalepimorphism[\tilde E] \tilde s^0) \tilde z_{\GroupPart} = 0\) for all \(\tilde m \in \ModulePart \tilde E\), \(\tilde g \in \GroupPart \tilde E\). Since \(s^0 (\GroupPart \varphi) = \tilde s^0\) and \(s^1 (\ModulePart \varphi) = (\GroupPart \varphi)|_{\Image \structuremorphism[E]}^{\Image \structuremorphism[\tilde E]} \tilde s^1\), it follows that
\[(m s^1) (\tilde z \CocycleGroup[2](\varphi, M))_{\ModulePart} = (m s^1 \varphi) \tilde z_{\ModulePart} = (m \varphi \tilde s^1) \tilde z_{\ModulePart} = 0\]
for all \(m \in \ModulePart E\) and
\begin{align*}
(g (g \canonicalepimorphism[E] s^0)^{- 1}, g \canonicalepimorphism[E] s^0) (\tilde z \CocycleGroup[2](\varphi, M))_{\GroupPart} & = ((g \varphi) (g \canonicalepimorphism[E] s^0 \varphi)^{- 1}, g \canonicalepimorphism[E] s^0 \varphi) \tilde z_{\GroupPart} \\
& = ((g \varphi) ((g \varphi) \canonicalepimorphism[\tilde E] \tilde s^0)^{- 1}, (g \varphi) \canonicalepimorphism[\tilde E] \tilde s^0) \tilde z_{\GroupPart} = 0
\end{align*}
for all \(g \in \GroupPart E\), that is, \(\tilde z \CocycleGroup[2](\varphi, M) \in \CocycleGroup[2]_{\text{st}}(E, M)\) by corollary~\ref{cor:second_cohomology_group_computable_by_standard_2-cocycles}\ref{cor:second_cohomology_group_computable_by_standard_2-cocycles:characterisation_of_standard_cocycles}. Hence \(\CocycleGroup[2](\varphi, M)\) restricts to a well-defined homomorphism
\[\CocycleGroup[2](\varphi, M)|_{\CocycleGroup[2]_{\text{st}}(\tilde E, M)}^{\CocycleGroup[2]_{\text{st}}(E, M)}\colon \CocycleGroup[2]_{\text{st}}(\tilde E, M) \map \CocycleGroup[2]_{\text{st}}(E, M).\]

Now,~\cite[prop.~(5.14)(c)]{thomas:2009:the_third_cohomology_group_classifies_crossed_module_extensions} implies that \(\cocycleofextension{3}_{E, (s^1, s^0)} = \cocycleofextension{3}_{\tilde E, (\tilde s^1, \tilde s^0)}\). By corollary~\ref{cor:second_cohomology_group_of_crossed_module_extension_via_3-cocycle}, we have isomorphisms
\[\Phi\colon \CocycleGroup[2]_{\text{st}}(E, M) \map \CocycleGroup[2]((\Pi_0, \Pi_1, \cocycleofextension{3}), M), z \mapsto (z \Phi_1, z \Phi_0)\]
given by \((k) (z \Phi_1) := (k \canonicalmonomorphism[E]) z_{\ModulePart}\) for \(k \in \Pi_1\) and \((q, p) (z \Phi_0) := (q s^0, p s^0) z_{\GroupPart}\) for \(p, q \in \Pi_0\), \(z \in \CocycleGroup[2]_{\text{st}}(E, M)\), and
\[\tilde \Phi\colon \CocycleGroup[2]_{\text{st}}(\tilde E, M) \map \CocycleGroup[2]((\Pi_0, \Pi_1, \cocycleofextension{3}), M), \tilde z \mapsto (\tilde z \tilde \Phi_1, \tilde z \tilde \Phi_0)\]
given by \((k) (\tilde z \tilde \Phi_1) := (k \canonicalmonomorphism[\tilde E]) \tilde z_{\ModulePart}\) for \(k \in \Pi_1\) and \((q, p) (\tilde z \tilde \Phi_0) := (q \tilde s^0, p \tilde s^0) \tilde z_{\GroupPart}\) for \(p, q \in \Pi_0\), \(\tilde z \in \CocycleGroup[2]_{\text{st}}(\tilde E, M)\). To show that \(\CocycleGroup[2](\varphi, M)|_{\CocycleGroup[2]_{\text{st}}(\tilde E, M)}^{\CocycleGroup[2]_{\text{st}}(E, M)}\) is an isomorphism, it suffices to verify that \(\tilde \Phi = (\CocycleGroup[2](\varphi, M))|_{\CocycleGroup[2]_{\text{st}}(\tilde E, M)}^{\CocycleGroup[2]_{\text{st}}(E, M)} \Phi\). Indeed, given \(\tilde z \in \CocycleGroup[2]_{\text{st}}(\tilde E, M)\), we have
\[k (\tilde z \CocycleGroup[2](\varphi, M) \Phi_1) = (k \canonicalmonomorphism[E]) (\tilde z \CocycleGroup[2](\varphi, M))_{\ModulePart} = (k \canonicalmonomorphism[E] \varphi) \tilde z_{\ModulePart} = (k \canonicalmonomorphism[\tilde E]) \tilde z_{\ModulePart} = k (\tilde z \tilde \Phi_1)\]
for all \(k \in \Pi_1\) and
\[(q, p) (\tilde z \CocycleGroup[2](\varphi, M) \Phi_0) = (q s^0, p s^0) (\tilde z \CocycleGroup[2](\varphi, M))_{\GroupPart} = (q s^0 \varphi, p s^0 \varphi) \tilde z_{\GroupPart} = (q \tilde s^0, p \tilde s^0) \tilde z_{\GroupPart} = (q, p) (\tilde z \tilde \Phi_0)\]
for all \(p, q \in \Pi_0\), that is, \(\tilde \Phi = (\CocycleGroup[2](\varphi, M))|_{\CocycleGroup[2]_{\text{st}}(\tilde E, M)}^{\CocycleGroup[2]_{\text{st}}(E, M)} \Phi\).

Moreover, the induced homomorphism \(\CoboundaryGroup[2](\varphi, M)\) also restricts to a well-defined homomorphism
\[\CoboundaryGroup[2](\varphi, M)|_{\CoboundaryGroup[2]_{\text{st}}(\tilde E, M)}^{\CoboundaryGroup[2]_{\text{st}}(E, M)}\colon \CoboundaryGroup[2]_{\text{st}}(\tilde E, M) \map \CoboundaryGroup[2]_{\text{st}}(E, M),\]
cf.\ definition~\ref{def:standardisation_of_pointed_2-cocycles}\ref{def:standardisation_of_pointed_2-cocycles:standard_2-cocycles}, which is an isomorphism since
\[\tilde \Phi|_{\CoboundaryGroup[2]_{\text{st}}(\tilde E, M)}^{\CoboundaryGroup[2]_{\text{EM}}((\Pi_0, \Pi_1, \cocycleofextension{3}), M)} = (\CoboundaryGroup[2](\varphi, M)|_{\CoboundaryGroup[2]_{\text{st}}(\tilde E, M)}^{\CoboundaryGroup[2]_{\text{st}}(E, M)}) (\Phi|_{\CoboundaryGroup[2]_{\text{st}}(E, M)}^{\CoboundaryGroup[2]_{\text{EM}}((\Pi_0, \Pi_1, \cocycleofextension{3}), M)})\]
and since \(\Phi|_{\CoboundaryGroup[2]_{\text{st}}(E, M)}^{\CoboundaryGroup[2]_{\text{EM}}((\Pi_0, \Pi_1, \cocycleofextension{3}), M)}\) and \(\tilde \Phi|_{\CoboundaryGroup[2]_{\text{st}}(\tilde E, M)}^{\CoboundaryGroup[2]_{\text{EM}}((\Pi_0, \Pi_1, \cocycleofextension{3}), M)}\) are isomorphisms by corollary~\ref{cor:second_cohomology_group_of_crossed_module_extension_via_3-cocycle}.

Finally, it follows that we get an induced isomorphism
\[\CohomologyGroup[2]_{\text{st}}(\tilde E, M) \map \CohomologyGroup[2]_{\text{st}}(E, M). \qedhere\]
\end{proof}

\section{\texorpdfstring{Second Eilenberg-Mac\,Lane cohomology group}{Second Eilenberg-Mac Lane cohomology group}} \label{sec:second_eilenberg-maclane_cohomology_group}

Until now, we have worked with crossed module extensions. Since every crossed module gives rise to a canonical crossed module extension, we can now formulate \eigenname{Eilenberg}s and \eigenname{Mac\,Lane}s theorem in the context of crossed modules and simplicial groups.

\begin{definition}[first Postnikov invariant] \label{def:first_postnikov_invariant} \
\begin{enumerate}
\item \label{def:first_postnikov_invariant:crossed_module} Given a crossed module \(V\), the cohomology class associated to the canonical extension
\[\HomotopyGroup[1](V) \morphism[\inc] \ModulePart V \morphism[\structuremorphism] \GroupPart V \morphism[\quo] \HomotopyGroup[0](V)\]
will be denoted by \(\postnikovinvariant{3}{V} := \cohomologyclassofextension(V) \in \CohomologyGroup[3]_{\text{cpt}}(\HomotopyGroup[0](V), \HomotopyGroup[1](V))\) and is called the (\newnotion{first}) \newnotion{Postnikov invariant} of \(V\).
\item \label{def:first_postnikov_invariant:simplicial_group} Given a simplicial group \(G\), we call \(\postnikovinvariant{3}{G} := \cohomologyclassofextension(\Truncation^1 G) \in \CohomologyGroup[3]_{\text{cpt}}(\HomotopyGroup[0](G), \HomotopyGroup[1](G))\) the \newnotion{first Postnikov invariant} of \(G\).
\end{enumerate}
\end{definition}

\begin{definition}[{second Eilenberg-Mac\,Lane cohomology group, cf.~\cite[sec.~3]{eilenberg_maclane:1946:determination_of_the_second_homology_and_cohomology_groups_of_a_space_by_means_of_homotopy_invariants}}] \label{def:second_eilenberg-maclane_cohomology_group} \
\begin{enumerate}
\item \label{def:second_eilenberg-maclane_cohomology:crossed_module} We suppose given a crossed module \(V\) and a componentwise pointed \(3\)-cocycle \(z^3 \in \CocycleGroup[3]_{\text{cpt}}(\HomotopyGroup[0](V), \HomotopyGroup[1](V))\) with \(\postnikovinvariant{3}{V} = z^3 \CoboundaryGroup[3]_{\text{cpt}}(\HomotopyGroup[0](V), \HomotopyGroup[1](V))\). The \newnotion{second Eilenberg-Mac\,Lane cohomology group} of \(V\) with respect to \(z^3\) and with coefficients in \(M\) is defined by
\[\CohomologyGroup[2]_{\text{EM}, z^3}(V, M) := \CohomologyGroup[2]((\HomotopyGroup[0](V), \HomotopyGroup[1](V), z^3), M).\]
\item \label{def:second_eilenberg-maclane_cohomology:simplicial_group} We suppose given a simplicial group \(G\) and a componentwise pointed \(3\)-cocycle \(z^3 \in \CocycleGroup[3]_{\text{cpt}}(\HomotopyGroup[0](G), \HomotopyGroup[1](G))\) with \(\postnikovinvariant{3}{G} = z^3 \CoboundaryGroup[3]_{\text{cpt}}(\HomotopyGroup[0](G), \HomotopyGroup[1](G))\). The \newnotion{second Eilenberg-Mac\,Lane cohomology group} of \(G\) with respect to \(z^3\) and with coefficients in \(M\) is defined by
\[\CohomologyGroup[2]_{\text{EM}, z^3}(G, M) := \CohomologyGroup[2]((\HomotopyGroup[0](G), \HomotopyGroup[1](G), z^3), M).\]
\end{enumerate}
\end{definition}

We have already seen that the isomorphism class of the second Eilenberg-Mac\,Lane cohomology group of a crossed module does not depend on the choice of a specific \(3\)-cocycle in its associated cohomology class:

\begin{remark} \label{rem:eilenberg-maclane_cohomology_groups_depend_only_on_the_cohomology_class_of_the_crossed_module}
Given a crossed module \(V\) and componentwise pointed \(3\)-cocycles \(z^3, \tilde z^3 \in \CocycleGroup[3]_{\text{cpt}}(\HomotopyGroup[0](V), \HomotopyGroup[1](V))\) with \(\postnikovinvariant{3}{V} = z^3 \CoboundaryGroup[3]_{\text{cpt}}(\HomotopyGroup[0](V), \HomotopyGroup[1](V)) = \tilde z^3 \CoboundaryGroup[3]_{\text{cpt}}(\HomotopyGroup[0](V), \HomotopyGroup[1](V))\), we have
\[\CohomologyGroup[2]_{\text{EM}, z^3}(V, M) \isomorphic \CohomologyGroup[2]_{\text{EM}, \tilde z^3}(V, M).\]
\end{remark}
\begin{proof}
This follows from corollary~\ref{cor:cohomology_groups_of_cohomologous_3-cocycles_are_isomorphic}.
\end{proof}

\pagebreak % manual format

\begin{theorem}[{cf.~\cite[th.~2]{eilenberg_maclane:1946:determination_of_the_second_homology_and_cohomology_groups_of_a_space_by_means_of_homotopy_invariants}}] \label{th:second_cohomology_group_of_crossed_module_via_homotopy_invariants} \
\begin{enumerate}
\item \label{th:second_cohomology_group_of_crossed_module_via_homotopy_invariants:crossed_module} Given a crossed module \(V\), an abelian \(\HomotopyGroup[0](V)\)-module \(M\) and a componentwise pointed \(3\)-cocycle \(z^3 \in \CocycleGroup[3]_{\text{cpt}}(\HomotopyGroup[0](V), \HomotopyGroup[1](V))\) with \(\postnikovinvariant{3}{V} = z^3 \CoboundaryGroup[3]_{\text{cpt}}(\HomotopyGroup[0](V), \HomotopyGroup[1](V))\), we have 
\[\CohomologyGroup[2](V, M) \isomorphic \CohomologyGroup[2]_{\text{EM}, z^3}(V, M).\]
\item \label{th:second_cohomology_group_of_crossed_module_via_homotopy_invariants:simplicial_group} Given a simplicial group \(G\), an abelian \(\HomotopyGroup[0](G)\)-module \(M\) and a componentwise pointed \(3\)-cocycle \(z^3 \in \CocycleGroup[3]_{\text{cpt}}(\HomotopyGroup[0](G), \HomotopyGroup[1](G))\) with \(\postnikovinvariant{3}{G} = z^3 \CoboundaryGroup[3]_{\text{cpt}}(\HomotopyGroup[0](G), \HomotopyGroup[1](G))\), we have 
\[\CohomologyGroup[2](G, M) \isomorphic \CohomologyGroup[2]_{\text{EM}, z^3}(G, M).\]
\end{enumerate}
\end{theorem}
\begin{proof} \
\begin{enumerate}
\item This follows from corollary~\ref{cor:second_cohomology_group_of_crossed_module_extension_via_3-cocycle} and remark~\ref{rem:eilenberg-maclane_cohomology_groups_depend_only_on_the_cohomology_class_of_the_crossed_module}.
\item Applying proposition~\ref{prop:second_cohomology_group_of_a_simplicial_group_and_its_1-truncation} and~\ref{th:second_cohomology_group_of_crossed_module_via_homotopy_invariants:crossed_module}, we obtain
\[\CohomologyGroup[2](G, M) \isomorphic \CohomologyGroup[2](\Truncation^1 G, M) \isomorphic \CohomologyGroup[2]_{\text{EM}, z^3}(\Truncation^1 G, M) = \CohomologyGroup[2]_{\text{EM}, z^3}(G, M). \qedhere\]
\end{enumerate}
\end{proof}

\begin{corollary}[{cf.~\cite[sec.~4]{eilenberg_maclane:1946:determination_of_the_second_homology_and_cohomology_groups_of_a_space_by_means_of_homotopy_invariants}}] \label{cor:second_cohomology_group_of_a_crossed_module_with_trivial_cohomology_class_or_finite_first_homotopy_group} \
\begin{enumerate}
\item \label{cor:second_cohomology_group_of_a_crossed_module_with_trivial_cohomology_class_or_finite_first_homotopy_group:simplicial_groups} We suppose given a simplicial group \(G\) and an abelian \(\HomotopyGroup[0](G)\)-module \(M\).
\begin{enumerate}
\item \label{cor:second_cohomology_group_of_a_crossed_module_with_trivial_cohomology_class_or_finite_first_homotopy_group:simplicial_groups:trivial_cohomology_class} If \(\postnikovinvariant{3}{G} = 1\), then
\[\CohomologyGroup[2](G, M) \isomorphic \Hom_{\HomotopyGroup[0](G)}(\HomotopyGroup[1](G), M) \directsum \CohomologyGroup[2](\HomotopyGroup[0](G), M).\]
\item \label{cor:second_cohomology_group_of_a_crossed_module_with_trivial_cohomology_class_or_finite_first_homotopy_group:simplicial_groups:vanishing_homomorphism_group} If \(\Hom_{\HomotopyGroup[0](G)}(\HomotopyGroup[1](G), M) = \{0\}\), then
\[\CohomologyGroup[2](G, M) \isomorphic \CohomologyGroup[2](\HomotopyGroup[0](G), M).\]
\end{enumerate}
\item \label{cor:second_cohomology_group_of_a_crossed_module_with_trivial_cohomology_class_or_finite_first_homotopy_group:crossed_modules} We suppose given a crossed module \(V\) and an abelian \(\HomotopyGroup[0](V)\)-module \(M\).
\begin{enumerate}
\item \label{cor:second_cohomology_group_of_a_crossed_module_with_trivial_cohomology_class_or_finite_first_homotopy_group:crossed_modules:trivial_cohomology_class} If \(\postnikovinvariant{3}{V} = 1\), then
\[\CohomologyGroup[2](V, M) \isomorphic \Hom_{\HomotopyGroup[0](V)}(\HomotopyGroup[1](V), M) \directsum \CohomologyGroup[2](\HomotopyGroup[0](V), M).\]
\item \label{cor:second_cohomology_group_of_a_crossed_module_with_trivial_cohomology_class_or_finite_first_homotopy_group:crossed_modules:vanishing_homomorphism_group} If \(\Hom_{\HomotopyGroup[0](V)}(\HomotopyGroup[1](V), M) = \{0\}\), then
\[\CohomologyGroup[2](V, M) \isomorphic \CohomologyGroup[2](\HomotopyGroup[0](V), M).\]
\end{enumerate}
\end{enumerate}
\end{corollary}
\begin{proof} \
\begin{enumerate}
\item
\begin{enumerate}
\item If \(\postnikovinvariant{3}{G} = 1\), then we have \(\CocycleGroup[2]((\HomotopyGroup[0](G), \HomotopyGroup[1](G), 1), M) = \Hom_{\HomotopyGroup[0](G)}(\HomotopyGroup[1](G), M) \directprod \CocycleGroup[2]_{\text{cpt}}(\HomotopyGroup[0](G), M)\) and hence
\begin{align*}
\CohomologyGroup[2](G, M) & \isomorphic \CohomologyGroup[2]_{\text{EM}, 1}(G, M) = \CohomologyGroup[2](\HomotopyGroup[0](G), \HomotopyGroup[1](G), 1), M) \\
& \isomorphic \Hom_{\HomotopyGroup[0](G)}(\HomotopyGroup[1](G), M) \directprod \CohomologyGroup[2]_{\text{cpt}}(\HomotopyGroup[0](G), M) \isomorphic \Hom_{\HomotopyGroup[0](G)}(\HomotopyGroup[1](G), M) \directsum \CohomologyGroup[2](\HomotopyGroup[0](G), M)
\end{align*}
by theorem~\ref{th:second_cohomology_group_of_crossed_module_via_homotopy_invariants}.
\item If \(\Hom_{\HomotopyGroup[0](G)}(\HomotopyGroup[1](G), M) = \{0\}\), then we get
\[\CohomologyGroup[2](G, M) \isomorphic \CohomologyGroup[2]_{\text{EM}, z^3}(G, M) = \CohomologyGroup[2](\HomotopyGroup[0](G), \HomotopyGroup[1](G), z^3), M) \isomorphic \CohomologyGroup[2]_{\text{cpt}}(\HomotopyGroup[0](G), M) \isomorphic \CohomologyGroup[2](\HomotopyGroup[0](G), M),\]
where \(z^3 \in \CocycleGroup[3]_{\text{cpt}}(\HomotopyGroup[0](G), \HomotopyGroup[1](G))\) with \(\postnikovinvariant{3}{G} = z^3 \CoboundaryGroup[3]_{\text{cpt}}(\HomotopyGroup[0](G), \HomotopyGroup[1](G))\).
\end{enumerate}
\item This follows from~\ref{cor:second_cohomology_group_of_a_crossed_module_with_trivial_cohomology_class_or_finite_first_homotopy_group:simplicial_groups} applied to the simplicial group \(\Coskeleton_1 V\). \qedhere
\end{enumerate}
\end{proof}

\pagebreak % manual format

\begin{question}[{cf.~\cite[sec.~5]{eilenberg_maclane:1946:determination_of_the_second_homology_and_cohomology_groups_of_a_space_by_means_of_homotopy_invariants}}] \label{qu:descriptions_by_homotopy_invariants_in_higher_dimensions_and_homology} \
\begin{enumerate}
\item \label{qu:descriptions_by_homotopy_invariants_in_higher_dimensions_and_homology:crossed_modules} We suppose given a crossed module \(V\) and an abelian \(\HomotopyGroup[0](V)\)-module \(M\). How can theorem~\ref{th:second_cohomology_group_of_crossed_module_via_homotopy_invariants} be generalised to obtain a description of \(\CohomologyGroup[n](V, M)\) for \(n \geq 3\) in terms of \(\HomotopyGroup[0](V)\), \(\HomotopyGroup[1](V)\) and \(\postnikovinvariant{3}{V}\)? What about such descriptions for homology?
\item \label{qu:descriptions_by_homotopy_invariants_in_higher_dimensions_and_homology:simplicial_groups} We suppose given a simplicial group \(G\) and an abelian \(\HomotopyGroup[0](V)\)-module \(M\). How can theorem~\ref{th:second_cohomology_group_of_crossed_module_via_homotopy_invariants} be generalised to obtain a description of \(\CohomologyGroup[n](G, M)\) for \(n \geq 3\) in terms of homotopy groups and Postnikov invariants? What about such descriptions for homology?
\end{enumerate}
\end{question}

Finally, we discuss some examples.

\begin{example} \label{ex:second_cohomology_group_of_trivial_homomorphism_crossed_module}
We suppose given a group \(\Pi_0\) and abelian \(\Pi_0\)-modules \(\Pi_1\) and \(M\). We let \(E\) be the crossed module extension
\[\Pi_1 \morphism[\id_{\Pi_1}] \Pi_1 \morphism[\triv] \Pi_0 \morphism[\id_{\Pi_0}] \Pi_0.\]
Then we have
\[\CohomologyGroup[2](E, M) \isomorphic \Hom_{\Pi_0}(\Pi_1, M) \directsum \CohomologyGroup[2](\Pi_0, M).\]
\end{example}
\begin{proof}
The \(3\)-cocycle of \(E\) with respect to the unique section system \((\triv, \id_{\Pi_0})\) for \(E\) is trivial and hence
\[\CohomologyGroup[2](E, M) \isomorphic \Hom_{\Pi_0}(\Pi_1, M) \directsum \CohomologyGroup[2](\Pi_0, M)\]
by corollary~\ref{cor:second_cohomology_group_of_a_crossed_module_with_trivial_cohomology_class_or_finite_first_homotopy_group}\ref{cor:second_cohomology_group_of_a_crossed_module_with_trivial_cohomology_class_or_finite_first_homotopy_group:crossed_modules}\ref{cor:second_cohomology_group_of_a_crossed_module_with_trivial_cohomology_class_or_finite_first_homotopy_group:crossed_modules:trivial_cohomology_class}.
\end{proof}

\begin{example} \label{ex:second_integral_cohomology_group_when_first_homotopy_group_is_finite}
We suppose given a simplicial group \(G\) such that \(\HomotopyGroup[1](G)\) is finite. Then we have
\[\CohomologyGroup[2](G, \Integers) \isomorphic \CohomologyGroup[2](\HomotopyGroup[0](G), \Integers).\]
\end{example}
\begin{proof}
Since \(\HomotopyGroup[1](G)\) is finite, we have \(\Hom_{\HomotopyGroup[0](G)}(\HomotopyGroup[1](G), \Integers) = \{0\}\), whence corollary~\ref{cor:second_cohomology_group_of_a_crossed_module_with_trivial_cohomology_class_or_finite_first_homotopy_group}\ref{cor:second_cohomology_group_of_a_crossed_module_with_trivial_cohomology_class_or_finite_first_homotopy_group:simplicial_groups}\ref{cor:second_cohomology_group_of_a_crossed_module_with_trivial_cohomology_class_or_finite_first_homotopy_group:simplicial_groups:vanishing_homomorphism_group} applies.
\end{proof}

\begin{example} \label{ex:second_cohomology_group_with_coefficients_in_an_cyclic_group_when_first_homotopy_group_has_order_2}
We suppose given a simplicial group \(G\) with \(\HomotopyGroup[0](G) \isomorphic \HomotopyGroup[1](G) \isomorphic \CyclicGroup_{2}\). For \(n \in \Naturals_0\), we have
\begin{align*}
\CohomologyGroup[2](G, \Integers / n) & \isomorphic \left. \begin{cases} \Hom(\CyclicGroup_2, \Integers / n) \directsum \CohomologyGroup[2](\CyclicGroup_2, \Integers / n) & \text{if \(\postnikovinvariant{3}{G} = 1\),} \\ \CohomologyGroup[2](\CyclicGroup_2, \Integers / n) & \text{if \(\postnikovinvariant{3}{G} \neq 1\),} \end{cases} \right\} \\
& \isomorphic \begin{cases} \Integers / 2 & \text{if \(n = 0\),} \\ \{0\} & \text{if \(n \in \Naturals\), \(2 \nmid n\),} \\ \Integers / 2 \directsum \Integers / 2 & \text{if \(n \in \Naturals\), \(2 \mid n\), \(\postnikovinvariant{3}{G} = 1\),} \\ \Integers / 2 & \text{if \(n \in \Naturals\), \(2 \mid n\), \(\postnikovinvariant{3}{G} \neq 1\),} \end{cases}
\end{align*}
where \(\Integers / n\) is considered as a trivial \(\CyclicGroup_2\)-module.
\end{example}
\begin{proof}
The assertion for \(\postnikovinvariant{3}{G} = 1\) is a particular case of corollary~\ref{cor:second_cohomology_group_of_a_crossed_module_with_trivial_cohomology_class_or_finite_first_homotopy_group}\ref{cor:second_cohomology_group_of_a_crossed_module_with_trivial_cohomology_class_or_finite_first_homotopy_group:simplicial_groups}\ref{cor:second_cohomology_group_of_a_crossed_module_with_trivial_cohomology_class_or_finite_first_homotopy_group:simplicial_groups:trivial_cohomology_class}, so let us suppose that \(\postnikovinvariant{3}{G} \neq 1\). For \(n = 0\), we get the assertion from example~\ref{ex:second_integral_cohomology_group_when_first_homotopy_group_is_finite}. So let us suppose given an \(n \in \Naturals\). By the additivity of \(\CohomologyGroup[2](G, -)\) resp.\ \(\CohomologyGroup[2](\HomotopyGroup[0](G), -)\) and the Chinese Remainder Theorem, it suffices to consider the case where \(n = p^e\) for a prime \(p\) and \(e \in \Naturals\). If \(p > 2\), we have \(\Hom_{\HomotopyGroup[0](G)}(\HomotopyGroup[1](G), \Integers / p^e) = \{0\}\) and hence
\[\CohomologyGroup[2](G, \Integers / p^e) \isomorphic \CohomologyGroup[2](\HomotopyGroup[0](G), \Integers / p^e)\]
by corollary~\ref{cor:second_cohomology_group_of_a_crossed_module_with_trivial_cohomology_class_or_finite_first_homotopy_group}\ref{cor:second_cohomology_group_of_a_crossed_module_with_trivial_cohomology_class_or_finite_first_homotopy_group:simplicial_groups}\ref{cor:second_cohomology_group_of_a_crossed_module_with_trivial_cohomology_class_or_finite_first_homotopy_group:simplicial_groups:vanishing_homomorphism_group}. 

It remains to consider the case \(n = 2^e\) for some \(e \in \Naturals\). We let \(x\) be the generator of \(\HomotopyGroup[0](G)\), we let \(y\) be the generator of \(\HomotopyGroup[1](G)\) and we let \(z^3 \in \CocycleGroup[3]_{\text{cpt}}(\HomotopyGroup[0](G), \HomotopyGroup[1](G))\) be a componentwise pointed \(3\)-cocycle with \(\postnikovinvariant{3}{G} = z^3 \CoboundaryGroup[3]_{\text{cpt}}(\HomotopyGroup[0](G), \HomotopyGroup[1](G))\). Since \(\postnikovinvariant{3}{G} \neq 1\), we have \(z^3 \neq 1\) and hence
\[(r, q, p) z^3 = \begin{cases} 1 & \text{for \((r, q, p) \neq (x, x, x)\),} \\ y & \text{for \((r, q, p) = (x, x, x)\).} \end{cases}\]
Now \(\Hom_{\HomotopyGroup[0](G)}(\HomotopyGroup[1](G), \Integers / 2^e) = \Hom(\HomotopyGroup[1](G), \Integers / 2^e)\) has a unique non-trivial element \(z_1\colon \HomotopyGroup[1](G) \map \Integers / 2^e\), which maps \(y\) to \(y z_1 = 2^{e - 1}\). But for all \(c_0 \in \CochainComplex[2]_{\text{cpt}}(\HomotopyGroup[0](G), \Integers / 2^e)\), we have
\[(x, x, x) (c_0 \differential) = (x, x) c_0 - (x, 1) c_0 + (1, x) c_0 - (x, x) c_0 = 0 \neq 2^{e - 1} = y z_1 = (x, x, x) z^3 z_1.\]
Hence there does not exist a cochain \(c_0 \in \CochainComplex[2]_{\text{cpt}}(\HomotopyGroup[0](G), \Integers / 2^e)\) with \(z^3 z_1 = c_0 \differential\). It follows that
\[\CocycleGroup[2]_{\text{EM}, z^3}(G, \Integers / 2^e) = \{0\} \directprod \CocycleGroup[2]_{\text{cpt}}(\HomotopyGroup[0](G), \Integers / 2^e)\]
and thus
\[\CohomologyGroup[2](G, \Integers / 2^e) \isomorphic \CohomologyGroup[2]_{\text{EM}, z^3}(G, \Integers / 2^e) \isomorphic \CohomologyGroup[2]_{\text{cpt}}(\HomotopyGroup[0](G), \Integers / 2^e) \isomorphic \CohomologyGroup[2](\HomotopyGroup[0](G), \Integers / 2^e). \qedhere\]
\end{proof}

\begin{example} \label{ex:second_cohomology_group_of_crossed_module_considering_of_cyclic_groups_of_order_4_with_coefficients_in_an_cyclic_group}
We consider the crossed module \(V\) with group part \(\GroupPart V = \langle a \mid a^4 = 1 \rangle\), module part \(\ModulePart V = \langle b \mid b^4 = 1 \rangle\), structure morphism given by \(b \structuremorphism = a^2\) and action given by \({^a}b = b^{- 1}\), cf.~\cite[ex.~(5.6)]{thomas:2007:co_homology_of_crossed_modules}. Then we have
\[\CohomologyGroup[2](V, \Integers / n) \isomorphic \begin{cases} \Integers / 2 & \text{for \(n \in \Naturals_0\) even,} \\ \{0\} & \text{for \(n \in \Naturals_0\) odd.} \end{cases}\]
\end{example}
\begin{proof}
The homotopy groups of \(V\) are given by \(\HomotopyGroup[0](V) = \langle x \rangle\) with \(x := a (\Image \structuremorphism)\) and \(\HomotopyGroup[1](V) = \langle y \rangle\) with \(y := b^2\), and we have \(\HomotopyGroup[0](V) \isomorphic \HomotopyGroup[1](V) \isomorphic \CyclicGroup_{2}\). Now \((s^1, s^0)\) defined by \(s^0\colon \HomotopyGroup[0](V) \map \GroupPart V, 1 \mapsto 1, x \mapsto a\) and \(s^1\colon \Image \structuremorphism \map \ModulePart V, 1 \mapsto 1, a^2 \mapsto b\) is a section system for \(V\). We let \((Z^2, Z^1)\) be the lifting system coming from \((s^1, s^0)\). It follows that \((x, x) \cocycleofextension{2} = (x s^0) (x s^0) (1 s^0)^{- 1} = a^2\) and therefore \((x, x) Z^2 = a^2 s^1 = b\). Finally,
\[(x, x, x) \cocycleofextension{3} = (x, x) Z^2 (1, x) Z^2 ((x, 1) Z^2)^{- 1} ({^{x Z^1}}(x, x) Z^2)^{- 1} = b \, {^a}(b^{- 1}) = b^2 = y\]
and therefore \(\cocycleofextension{3} \neq 1\). Since
\[(x, x, x) (c^2 \differential) = (x, x) c^2 ((x, 1) c^2)^{- 1} (1, x) c^2 ({^x}(x, x) c^2)^{- 1} = (x, x) c^2 ((x, x) c^2)^{- 1} = 1\]
for every componentwise pointed \(2\)-cochain \(c^2 \in \CochainComplex[2]_{\text{cpt}}(\HomotopyGroup[0](V), \HomotopyGroup[1](V))\), we conclude that \(\cocycleofextension{3} \notin \CoboundaryGroup[3]_{\text{cpt}}(\HomotopyGroup[0](V), \HomotopyGroup[1](V))\) and hence \(\postnikovinvariant{3}{V} \neq 1\). The assertion follows now from example~\ref{ex:second_cohomology_group_with_coefficients_in_an_cyclic_group_when_first_homotopy_group_has_order_2}.
\end{proof}

% bibliography

\bigskip

{\raggedleft Sebastian Thomas \\ Lehrstuhl D f{\"u}r Mathematik \\ RWTH Aachen University \\ Templergraben 64 \\ 52062 Aachen \\ Germany \\ sebastian.thomas@math.rwth-aachen.de \\ \url{http://www.math.rwth-aachen.de/~Sebastian.Thomas/} \\}

\end{document}